\documentclass[english]{article}
\usepackage[T1]{fontenc}
\usepackage[latin9]{inputenc}
\usepackage{babel}
\usepackage{textcomp}
\usepackage{amsmath}
\usepackage{amsthm}
\usepackage{amssymb}
\usepackage[all]{xy}
\usepackage[unicode=true,pdfusetitle,
 bookmarks=true,bookmarksnumbered=false,bookmarksopen=false,
 breaklinks=false,pdfborder={0 0 1},backref=false,colorlinks=false]
 {hyperref}

\makeatletter
\numberwithin{equation}{section}
\numberwithin{figure}{section}
\theoremstyle{plain}
\newtheorem{thm}{\protect\theoremname}[section]
\theoremstyle{definition}
\newtheorem{example}[thm]{\protect\examplename}
\theoremstyle{definition}
\newtheorem{defn}[thm]{\protect\definitionname}
\theoremstyle{plain}
\newtheorem*{lem*}{\protect\lemmaname}
\theoremstyle{remark}
\newtheorem{notation}[thm]{\protect\notationname}
\theoremstyle{plain}
\newtheorem*{prop*}{\protect\propositionname}
\theoremstyle{plain}
\newtheorem*{thm*}{\protect\theoremname}
\theoremstyle{plain}
\newtheorem{lem}[thm]{\protect\lemmaname}
\theoremstyle{plain}
\newtheorem{cor}[thm]{\protect\corollaryname}
\theoremstyle{plain}
\newtheorem{prop}[thm]{\protect\propositionname}
\theoremstyle{remark}
\newtheorem{rem}[thm]{\protect\remarkname}

\usepackage[all]{xy}
\usepackage{mathrsfs}
\usepackage{eucal}
\usepackage{bbm}
\usepackage{fullpage}
\usepackage{mathscinet}
\usepackage{attrib}
\usepackage{authblk}
\usepackage{enumerate}

\hypersetup{hidelinks = true}

\newtheorem{propositiona}{Proposition}

\theoremstyle{plain}
\newtheorem{corollarya}{Corollary}

\newcommand{\hooklongrightarrow}{\lhook\joinrel\longrightarrow}
\newcommand{\twoheadlongrightarrow}{\relbar\joinrel\twoheadrightarrow}

\newcommand\<{\langle}
\renewcommand\>{\rangle}

\usepackage{graphicx}

\def\nicefracbig#1#2{
    \raise.5ex\hbox{$#1$}%
    \kern-.25em\bigr/\kern-.15em%
    \lower.35ex\hbox{$#2$}}

\newcommand{\xyR}[1]{%
\xydef@\xymatrixrowsep@{#1}
}

\newcommand{\xyC}[1]{%
\xydef@\xymatrixcolsep@{#1}
}

\AtEndDocument{%
  \par  
  \medskip 
  \begin{tabular}{@{}l@{}}%
    \textsc{Mathematical Institute, University of Oxford, Oxford OX2 6GG, United Kingdom}\\
      \textit{E-mail address}: \texttt{w.malten@gmail.com}   
  \end{tabular}}

\theoremstyle{definition}

\makeatother

\providecommand{\corollaryname}{Corollary}
\providecommand{\definitionname}{Definition}
\providecommand{\examplename}{Example}
\providecommand{\lemmaname}{Lemma}
\providecommand{\notationname}{Notation}
\providecommand{\propositionname}{Proposition}
\providecommand{\remarkname}{Remark}
\providecommand{\theoremname}{Theorem}

\begin{document}
\title{From braids to transverse slices in reductive groups}
\author{Wicher Malten}
\maketitle
\begin{abstract}
In 1965, Steinberg's study of conjugacy classes in connected reductive
groups led him to construct an affine subspace parametrising regular
conjugacy classes, which he noticed is also a cross section for the
conjugation action by the unipotent radical of a Borel subgroup on
another affine subspace. Recently, generalisations of this slice and
its cross section property have been obtained by Sevostyanov in the
context of quantum group analogues of $W$-algebras and by He-Lusztig
in the context of Deligne-Lusztig varieties. Such slices are often
thought of as group analogues of Slodowy slices.

In this paper we explain their relationship via common generalisations
associated to Weyl group elements and provide a simple criterion for
cross sections in terms of roots. In the most important class of examples
this criterion is equivalent to a statement about the Deligne-Garside
factors of their powers in the braid monoid being maximal in some
sense. Moreover, we show that these subvarieties transversely intersect
conjugacy classes and determine for a large class of factorisable
$r$-matrices when the Semenov-Tian-Shansky bracket reduces to a Poisson
structure on these slices.
\end{abstract}
\tableofcontents{}

\section{Introduction}

In 1965, Steinberg's study of conjugacy classes in a connected reductive
group $G$ led him to analyse its regular elements, characterising
them in various ways; they can be defined as the elements whose conjugacy
class has minimal codimension (which is equal to the rank of $G$),
and together these elements form a dense open subset. Fixing a Borel
subgroup $B$ and maximal torus $T$, denoting by $N_{+}=[B,B]$ the
unipotent radical of $B$, by $N_{-}$ its opposite, by $W=N_{G}(T)/T$
the Weyl group, writing for $w$ in $W$ then $N_{w}:=N_{+}\cap w^{-1}N_{-}w$
and $\dot{w}$ for arbitrary lifts to the normaliser $N_{G}(T)$ of
$T$ in $G$ and finally fixing $w$ to be a Coxeter element of minimal
length,\footnote{By Coxeter element we mean any conjugate of a product (in any order)
of all the simple reflections.} he proves that the Steinberg slice $\dot{w}N_{w}$ yields a cross
section of these regular conjugacy classes \cite[§1.4]{MR0180554}.
\begin{example}
[{\cite[§7.4]{MR0180554}}] Let $G=\mathrm{SL}_{\mathrm{rk}+1}$ over
a scheme $S$ and consider the Coxeter element $w=s_{\mathrm{rk}}\cdots s_{1}$.
The usual lift $\dot{w}$ yields the space of Frobenius companion
matrices
\[
\dot{w}N_{w}=\left\{ \begin{bmatrix}0 & 1 & \cdots & 0 & 0\\
\vdots & \vdots & \ddots & \vdots & \vdots\\
0 & 0 & \cdots & 1 & 0\\
0 & 0 & \cdots & 0 & 1\\
(-1)^{\mathrm{rk}} & c_{\mathrm{rk}} & \cdots & c_{2} & c_{1}
\end{bmatrix}:c_{1},\ldots,c_{\mathrm{rk}}\in\CMcal O_{S}\right\} .
\]
Furthermore, Steinberg settles the existence of regular unipotent
elements by locating some of them inside of the Bruhat cell $BwB$
\cite[§4]{MR0180554}. He subsequently remarks that actually $BwB$
consists entirely of regular elements, which he notes follows from
the same property of $\dot{w}N_{w}$ in combination with the isomorphism
\begin{equation}
N_{+}\times\dot{w}N_{w}\overset{\sim}{\longrightarrow}N_{+}\dot{w}N_{+},\qquad(n,g)\longmapsto n^{-1}gn,\label{eq:Steinberg-cross-section}
\end{equation}
 but the proof of that isomorphism is missing \cite[§8.9]{MR0180554}.
In the late 1970s this map was investigated by Spaltenstein, who found
an example in type $\mathsf{A}_{5}$ (see Example \ref{exa:spaltenstein})
showing that it is not necessarily an isomorphism when $w$ is replaced
by a Coxeter element which is not of minimal length \cite[§0.4]{MR2904572}
. Very recently, generalisations of this cross section appeared in
two different settings:
\end{example}

Perhaps the most natural way to construct $W$-algebras is by applying
quantum Hamiltonian reduction to a certain character of the Lie algebra
$\mathfrak{n}$ of a subgroup $N$ of $N_{+}$ generated by root subgroups
\cite{MR1255424,MR1929302,MR1876934}. For example, Kostant's study
in 1978 of Whittaker representations for the Langlands program led
him to construct the principal finite $W$-algebra \cite{MR507800},
which can be interpreted as applying this procedure to a regular character
of the Lie algebra $\mathfrak{n}_{+}$ of $N_{+}$. The character
of $\mathfrak{n}$ dequantises to a symplectic point of $\mathfrak{n}^{*}$
via Dixmier's map \cite{MR182682} and the semiclassical limit of
the corresponding finite $W$-algebra is then obtained by reducing
the inverse image of this point under the momentum map for the restriction
of the coadjoint action of $G$ on the dual $\mathfrak{g}^{*}$ of
its Lie algebra to $N$. The quantum group analogue of that reduces
inverse images of symplectic points under momentum maps for the restriction
of the dressing action of a factorisable Poisson-Lie group $G$ (with
the Drinfeld-Sklyanin bracket) on its dual Poisson-Lie group $G^{*}$
(which is quantised by the Drinfeld-Jimbo quantum group $U_{q}G$)
to coisotropic subgroups $N$ of $N_{+}$ generated by root subgroups.
One can exponentiate the symplectic point in $\mathfrak{n}_{+}^{*}$
to a point in the dual $N_{+}^{*}$ of $N_{+}$, and show directly
that it is symplectic if and only if the standard $r$-matrix on $G$
is modified by the Cayley transform of a Coxeter element of minimal
length \cite{WM-symplectic-points}. Denoting by $G_{*}$ the variety
$G$ equipped with the corresponding Semenov-Tian-Shansky bracket
(which is quantised by the integrable part $U_{q}^{\mathrm{int}}G$
of $U_{q}G$), the resulting reduction in $G^{*}$ covers a reduction
in $G_{*}$ along the $G$-equivariant factorisation mapping $G^{*}\rightarrow G_{*}$.
More precisely, this is the Poisson reduction of $N_{+}\dot{w}N_{+}\subset G_{*}$
by the conjugation action of $N_{+}\subset G$, so the cross section
suggests that a certain reduction of $U_{q}^{\mathrm{int}}G$ should
quantise the Steinberg slice $\dot{w}N_{w}$. The differential graded
algebra corresponding to the principal case was thoroughly studied
in work of Sevostyanov \cite{MR1793611}, the author \cite{WM-symplectic-points},
and Grojnowski \cite{Grojnowski}. Whilst searching for quantum group
analogues of non-principal $W$-algebras, Sevostyanov generalised
the Steinberg slice to some elements in every conjugacy class of the
Weyl group, and proved that for any those elements a conjugation map
similar to (\ref{eq:Steinberg-cross-section}) is an isomorphism \cite{MR2806525}.

All infinite families of finite simple groups, except those of the
alternating groups, consist of finite reductive groups. In 1976, Deligne
and Lusztig proved that all of their characters can be obtained from
the compactly supported étale cohomology of certain algebraic varieties
$X_{w}^{G}$ constructed out of such groups \cite{MR393266}, using
a twist $F$ (a Frobenius morphism) and an element $w$ of the Weyl
group. They showed that these virtual representations do not depend
on the twisted conjugacy class of $w$ \cite[Theorem 1.6]{MR393266},
reducing much of the study of these Deligne-Lusztig varieties to minimal
length elements. Any minimal length element is elliptic in a standard
parabolic corresponding to a Levi subgroup $L$, and one can show
that the orbit space satisfies $G^{F}\backslash X_{w}^{G}\simeq L^{F}\backslash X_{w}^{L}$,
thus reducing to the case where $w$ is elliptic and has minimal length.
Lusztig and his former student He generalised the cross section (\ref{eq:cross-section})
to these elements, and deduced from this that $G^{F}\backslash X_{w}^{G}$
is universally homeomorphic to the quotient of $\mathbb{A}^{\ell(w)}$
by a finite torus \cite{MR2904572} (which implies a cohomology vanishing
theorem \cite{MR2427642}, simplifying Lusztig's classification of
the representations of finite reductive groups \cite{MR742472}).
This statement was later generalised by He to obtain a dimension formula
for affine Deligne-Lusztig varieties \cite{MR3126571}, which are
closely related to the reduction of integral models of Shimura varieties
in arithmetic geometry. He and Lusztig mention that it
\begin{quote}
is not clear to us what is the relation of the Weyl group elements
considered {[}by Sevostyanov{]} with those considered in this paper
\attrib{\cite[§0.3]{MR2904572}}
\end{quote}
and the original aim of this paper and its prequel \cite{WM-mindom}
was to provide a common generalisation explaining this.

Following work by Springer \cite{MR0263830}, Grothendieck around
1969 obtained a simultaneous resulution of the singularities of the
fibres of $G\rightarrow G/\!\!/G\simeq T/W$ and conjectured that
a transverse slice to a subregular element of $G$ would yield a universal
deformation of the corresponding Du Val-Klein singularity, and similarly
for its Lie algebra. This was proven by Brieskorn \cite{MR0437798},
and Slodowy thoroughly studied these deformations through suitable
slices in the Lie algebra \cite{MR584445}. Already appearing in the
work of Harish-Chandra on invariant distributions on Lie algebras
\cite{MR161940}, Slodowy slices play a crucial rôle in the classification
of Whittaker representations as the semiclassical limits of $W$-algebras
\cite{MR507800,MR1929302}, have recently been applied to reconstruct
Khovanov homology \cite{MR2254624,MR3867999} and numerous physicists
are currently using them in their study of supersymmetric gauge theories
(e.g.\ \cite{MR2610576,MR2548595}). 

Meanwhile, Drinfeld showed that every reductive Lie algebra (interpreted
as a Casimir Lie algebra with the Killing form) admits a unique quantisation,
up to gauge transformations \cite{MR1047964}; semiclassically, this
corresponds to a choice of factorisable $r$-matrix, and those were
classified by Belavin-Drinfeld \cite{MR674005}. In his study of zero
curvature conditions for soliton equations, Semenov-Tian-Shansky discovered
a natural Poisson bracket on $G$ associated with these $r$-matrices,
which is compatible with conjugation \cite{MR842417}.

By using the Kazhdan-Lusztig map \cite{MR947819}, Sevostyanov associated
certain Weyl group elements to the subregular unipotent classes and
determined that his corresponding slices have the correct dimension,
yielding analogous results. He subsequently showed that after modifying
the standard $r$-matrix, the Semenov-Tian-Shansky bracket on $G$
reduces to his slices and lifts to Poisson structures on these minimal
surface singularities \cite{MR2806525}. Very recently, it was shown
that the analogue of the Grothendieck-Springer resolution for $G$-bundles
on elliptic curves \cite{MR3383168,2019arXiv190804140D} yields del
Pezzo surfaces \cite{MR4210255,2020arXiv200413268D}, whilst slices
in the affine Grassmannian were quantised \cite{MR3248988} and immediately
applied in the study of Coulomb branches in physics (e.g.\ \cite{MR3663621,MR4020310}).

Employing Lusztig's recent stratification of $G$ by unions of sheets
of conjugacy classes \cite{MR3495802}, Sevostyanov verified that
any conjugacy class of $G$ is strictly transversally intersected
by one of his slices \cite{MR3883243}. This forms the main ingredient
in Sevostyanov's approach \cite{sevostyanov2021qw} to the long-standing
De Concini-Kac-Procesi conjecture on the dimensions of irreducible
representations of quantum groups at roots of unity \cite[§6.8]{MR1124981}
(the analogous Kac-Weisfeiler conjecture for Lie algebras \cite{MR0285575}
was proven with Slodowy slices \cite{MR1345285}). Furthermore, Sevostyanov's
slices have been used in an attempt to obtain a group analogue \cite{MR3417486}
of Katsylo's theorem \cite{MR661143}, which itself has been applied
e.g.\ to analyse the space of one-dimensional representations of
finite $W$-algebras and to the theory of primitive ideals \cite{MR3260140}.\newline

It is natural in the theory of reductive Poisson-Lie groups and Drinfeld-Jimbo
quantum groups to fix a torus and Borel subgroup as before. Recall
from the introduction of \cite{WM-mindom} the notion and notation
for twisted finite Coxeter groups, (firmly) convex elements, reduced
braids, braid power bounds, roots, stable roots, convex sets of roots,
inversion sets, weak left Bruhat-Chevalley orders, right Deligne-Garside
factors and normal forms and the braid equation
\begin{equation}
\mathrm{DG}(b_{w}^{d})=\mathrm{pb}(w).\label{eq:braid-equation}
\end{equation}
The relevance of this equation to the main theorem is through its
close relationship with
\begin{defn}
\label{def:cross} Let $w$ be an element of a twisted Weyl group.
For any positive root $\beta$ of its root system, we construct a
subset of positive roots
\[
\mathrm{cross}_{w}(\beta):=\bigl\{ w(\beta+\sum_{i=1}^{m}\beta_{i})\in\mathfrak{R}:\beta_{1},\ldots,\beta_{m}\in\mathfrak{R}_{w},m\geq0\bigr\}\cap\mathfrak{R}_{+}.
\]
This extends to a map
\[
\mathrm{cross}_{w}:\{\textrm{subsets of }\mathfrak{R}_{+}\}\longrightarrow\{\textrm{subsets of }\mathfrak{R}_{+}\},\qquad\mathfrak{N}\longmapsto\bigcup_{\beta\in\mathfrak{N}}\mathrm{cross}_{w}(\beta),
\]
and we let $\mathrm{cross}_{w}^{d}(\cdot)$ denote its $d$-th iterate
for integers $d\geq0$.
\end{defn}

A more group-theoretic interpretation is given in Corollary \ref{cor:cross-root-subgroups},
where it is explained that $\mathrm{cross}_{w}(\cdot)$ is a (simplification
of a) first-order approximation to the polynomial equations appearing
in our generalisation of the cross section (\ref{eq:Steinberg-cross-section}).
In Lemma \ref{lem:braid-invariance-cross} we prove that $\mathrm{cross}_{w}(\cdot)$
naturally extends to twisted braid monoids, which is one of the ingredients
in the proof of the following
\begin{lem*}
\label{lem:cross} Let w be an element of a twisted Weyl group $W$.

\begin{enumerate}[\normalfont(i)]

\item For any simple root $\alpha$ not in the inversion set $\mathfrak{R}_{w}$,
the set $\mathrm{cross}_{w}(\alpha)$ contains simple roots.

\item For any other element $w'$ of $W$ and integer $d\geq0$,
we have
\[
\mathrm{DG}(b_{w}^{d})\geq w'\qquad\textrm{if and only if }\qquad\mathrm{cross}_{w}^{d}(\mathfrak{R}_{w'})=\varnothing,
\]
if and only if $\mathrm{cross}_{w}^{d}(\mathfrak{R}_{w'})$ does not
contain any simple roots.

\item Moreover, for any $d>|\mathfrak{R}_{+}\backslash\mathfrak{R}_{\mathrm{st}}^{w}|-\ell(w)$
we have
\[
w\textrm{ is convex and satisfies }(\ref{eq:braid-equation})\qquad\textrm{if and only if }\qquad\mathrm{cross}_{w}^{d}(\mathfrak{R}_{+}\backslash\mathfrak{R}_{\mathrm{st}}^{w})=\varnothing,
\]
if and only if $\mathrm{cross}_{w}^{d}(\mathfrak{R}_{+}\backslash\mathfrak{R}_{\mathrm{st}}^{w})$
does not contain any simple roots.

\end{enumerate}
\end{lem*}
\begin{example}
Consider $w=s_{2}$ in type $\mathsf{B}_{2}$, then the only simple
root in $\mathrm{cross}_{w}(\alpha_{1})$ is $\alpha_{1}=w(\alpha_{1}+2\alpha_{2})$.
\end{example}

\begin{example}
Consider $w=s_{1}s_{3}$ in type $\mathsf{A}_{3}$, then the only
simple root in $\mathrm{cross}_{w}(\alpha_{2})$ is $\alpha_{2}=w(\alpha_{2}+\alpha_{1}+\alpha_{3})$.
\end{example}

Our proof of transversality involves root combinatorics which through
this lemma is closely intertwined with

\begin{corollarya} \label{cor:inverse} Let $w$ be a convex element
of a twisted Weyl group. Then
\[
\mathrm{DG}(b_{w}^{d})=\mathrm{pb}(w)\qquad\textrm{if and only if }\qquad\mathrm{DG}(b_{w^{-1}}^{d})=\mathrm{pb}(w).
\]

\end{corollarya}

Typically however, these Deligne-Garside factors are not very similar:
\begin{example}
Consider $w=s_{1}s_{2}s_{3}s_{1}s_{2}$ in type $\mathsf{B}_{3}$;
it is convex and for any integer $d>1$ we have
\[
\mathrm{DGN}(b_{w}^{d})=b_{w}^{d}\qquad\textrm{and}\qquad\mathrm{DGN}(b_{w^{-1}}^{d})=b_{w^{-1}s_{1}}^{\,}b_{w^{-1}}^{d-2}b_{s_{1}w^{-1}}^{\,}.
\]
\end{example}

 The perspective furnished by Definition \ref{def:cross} allows
us to construct cross sections out of quite general data:
\begin{defn}
Let $w$ be an element of a twisted Weyl group and let $\mathfrak{N}\subseteq\mathfrak{R}_{+}\backslash\mathfrak{R}_{\mathrm{st}}^{w}$
be a convex subset of positive roots. Then we will say that $\mathfrak{N}$
is \emph{($w$-)nimble} if it contains $\mathfrak{R}_{w}$ and furthermore
\[
w(\mathfrak{N}\backslash\mathfrak{R}_{w})\subseteq\mathfrak{N}.
\]
\end{defn}

\begin{example}
From \cite[Proposition 4.21]{WM-mindom} it follows that the set $\mathfrak{R}_{+}\backslash\mathfrak{R}_{\mathrm{st}}^{w}$
is nimble itself if and only if the element $w$ is convex.
\end{example}

\begin{notation}
Given two subsets of roots $\mathfrak{N},\mathfrak{L}$ of a root
system $\mathfrak{R}$, we write
\[
\mathfrak{N}+\mathfrak{L}:=\{c_{0}\beta_{0}+c_{1}\beta_{1}\in\mathfrak{R}:\beta_{0}\in\mathfrak{N},\beta_{1}\in\mathfrak{L},c_{0},c_{1}\in\mathbb{R}_{>0}\}.
\]
\end{notation}

\begin{defn}
Let $w$ be an element of a twisted Weyl group and let $\mathfrak{N}$
a $w$-nimble set. We will say that a convex subset of roots $\mathfrak{L}\subseteq\mathfrak{R}\backslash\mathfrak{N}$
is a \emph{(}$\mathfrak{N}$\emph{-)leavener} when

\begin{enumerate}[(i)]

\item $w(\mathfrak{L})=\mathfrak{L}=-\mathfrak{L}$,

\item the set $\mathfrak{R}_{w}\sqcup\mathfrak{L}$ is convex, and

\item the set $\mathfrak{N}\sqcup\mathfrak{L}$ is convex.

\end{enumerate}

We will then refer to $(\mathfrak{N},\mathfrak{L})$ as a ($w$\emph{-)crossing
pair}. Note that (i) implies $\mathfrak{R}_{w^{-1}}\cap\mathfrak{L}=\varnothing$.
If furthermore $\mathfrak{R}_{+}\subseteq\mathfrak{N}\sqcup\mathfrak{L}$
(forcing $\mathfrak{L}$ to be a standard parabolic subsystem), then
we will say that it is a\emph{ }($w$\emph{-)slicing pair}.
\end{defn}

Proposition \ref{prop:dg-inversion-sets} explains that crossing pairs
arise quite naturally. In particular:
\begin{prop*}
Let $w$ be an element of a twisted Weyl group. If $w$ is convex,
then for any convex subset $\mathfrak{L}\subseteq\mathfrak{R}_{\mathrm{st}}^{w}$
satisfying $w(\mathfrak{L})=\mathfrak{L}=-\mathfrak{L}$, both 
\[
(\mathfrak{R}_{+}\backslash\mathfrak{R}_{\mathrm{st}}^{w},\mathfrak{L})\qquad\textrm{and}\qquad(\mathfrak{R}_{\mathrm{DG}(b_{w}^{d})},\mathfrak{L})
\]
 for any $d\geq1$ are crossing pairs.
\end{prop*}
I expect that the most important class of cross sections will be the
one arising from firmly convex elements with crossing pair $(\mathfrak{R}_{+}\backslash\mathfrak{R}^{w},\mathfrak{R}^{w})$.
In order to extend our results to twisted conjugacy classes and make
the proofs slightly cleaner, we will employ the language of twisted
reductive groups:
\begin{defn}
Let $\tilde{G}$ be a split reductive group over a scheme with chosen
Borel subgroup, and let $\Omega$ be a group of twists of its Weyl
group $\tilde{W}$; these extend to automorphisms of $\tilde{G}$.
Also let $\phi$ be an automorphism of the base scheme. We will then
call $G:=\phi\Omega\ltimes\tilde{G}$ a \emph{twisted split reductive
group}, and $W:=\Omega\ltimes\tilde{W}$ its \emph{Weyl group}. 
\end{defn}

\begin{notation}
Let $w$ be an element of its Weyl group and let $(\mathfrak{N},\mathfrak{L})$
be a $w$-crossing pair. Then we let $N\subseteq N_{+}$ denote the
unipotent subgroup corresponding to the roots in $\mathfrak{N}$,
and let $L$ denote the reductive subgroup of $G$ generated by a
chosen subgroup $T'$ of the torus $T$, and by the root subgroups
corresponding to $\mathfrak{L}$. 
\end{notation}

In the firmly convex case, one typically sets $T'$ to be the set
of fixed points $T^{w}$ under the action of $w$ on the torus $T$;
we denote its Lie algebra by $\mathfrak{t}^{w}$.
\begin{notation}
Factorisable $r$-matrices are parametrised by a \emph{Belavin-Drinfeld
triple} $\mathfrak{T}=(\mathfrak{T}_{0},\mathfrak{T}_{1},\tau)$ defining
a nilpotent isomorphism $\tau:\mathfrak{T}_{0}\rightarrow\mathfrak{T}_{1}$
between two subsets of simple roots, and an element $r_{0}$ in $\mathfrak{t}_{\mathfrak{T}}\wedge\mathfrak{t}_{\mathfrak{T}}$,
where $\mathfrak{t}_{\mathfrak{T}}$ is a certain subspace of the
Lie algebra of the maximal torus $T$. Given such a triple $\mathfrak{T}$
and a set of roots $\mathfrak{L}$, by $\mathfrak{T}\subseteq\mathfrak{L}$
we mean that $\mathfrak{T}_{0}$ and $\mathfrak{T}_{1}$ are both
contained in $\mathfrak{L}$; thus this is always satisfied for the
empty triple.\newline
\end{notation}

The main aim of this paper is to generalise He-Lusztig's, Sevostyanov's
and Steinberg's slices to the following
\begin{thm*}
\label{thm:slices} Let $G$ be a twisted split reductive group over
a scheme with chosen Borel subgroup containing a maximal torus $T$.
Pick any element $w$ in its Weyl group and let $(\mathfrak{N},\mathfrak{L})$
be a $w$-crossing pair such that for some integer $d\geq0$, or equivalently
for any integer $d>|\mathfrak{N}|-\ell(w)$, there is an identity
\begin{equation}
\mathrm{cross}_{w}^{d}(\mathfrak{N})=\varnothing.\label{eq:crossing}
\end{equation}

\begin{enumerate}[\normalfont(i)]

\item Then the (right) conjugation map
\begin{equation}
N\times\dot{w}LN_{w}\longrightarrow N\dot{w}LN,\qquad(n,g)\longmapsto n^{-1}gn\label{eq:cross-section}
\end{equation}
is an isomorphism.

\item If furthermore $L$ contains $T^{w}$ and the pair is slicing,
then the conjugation map
\begin{equation}
G\times\dot{w}LN_{w}\longrightarrow G,\qquad(n,g)\longmapsto n^{-1}gn\label{eq:transversality-again-1}
\end{equation}
 is smooth; in other words,
\[
\textrm{the subspace }\dot{w}LN_{w}\textrm{ of }G\textrm{ transversely intersects the conjugation orbits of }\tilde{G}.
\]
\item If $w$ is firmly closed and $(\mathfrak{N},\mathfrak{L})=(\mathfrak{R}_{+}\backslash\mathfrak{R}^{w},\mathfrak{R}^{w})$
and $L\cap T\subseteq T^{w}$, then the Semenov-Tian-Shansky bracket
associated to a factorisable $r$-matrix with Belavin-Drinfeld triple
$\mathfrak{T}$ satisfying $\mathfrak{T}\subseteq\mathfrak{L}$ reduces
to the subspace $\dot{w}LN_{w}$ if and only if

\begin{enumerate}[\normalfont(a)]

\item $L\cap T=T^{w}$,

\item $r_{0}$ preserves $\mathfrak{t}^{w}$, and

\item $r_{0}$ acts on its orthogonal complement as the Cayley transform
\[
\frac{1+w}{1-w}.
\]
\end{enumerate}

\end{enumerate}
\end{thm*}
My reasons for choosing this formulation are given at the start of
section §\ref{sec:cross-section}. 
\begin{example}
Let $G=\mathrm{SL}_{3}$, and pick the Borel subgroup and maximal
torus and of upper triangular matrices and diagonal matrices. Set
$w:=s_{1}s_{2}s_{1}$ and lift it to
\[
\dot{w}=\begin{bmatrix}0 & 0 & 1\\
0 & -1 & 0\\
1 & 0 & 0
\end{bmatrix}\in G.
\]
 As $w=w_{\circ}$, part (i) of the main Theorem implies that the
conjugation map sending
\[
\left(\begin{bmatrix}1 & n_{1} & n_{12}\\
0 & 1 & n_{2}\\
0 & 0 & 1
\end{bmatrix},\begin{bmatrix}0 & 0 & t\\
0 & -t^{-2} & x_{2}\\
t & x_{1} & x_{12}
\end{bmatrix}\right)\in N_{+}\times\dot{w}T^{w}N_{+}
\]
 to\small
\[
\begin{bmatrix}(n_{1}n_{2}-n_{12})t & n_{1}t^{-2}+(n_{1}n_{2}-n_{12})(n_{1}t+x_{1}) & n_{1}n_{2}t^{-2}+(n_{1}n_{2}-n_{12})(n_{12}t+x_{12}+n_{2}x_{1})+t-n_{1}x_{2}\\
-n_{2}t & -t^{-2}-n_{1}n_{2}t-n_{2}x_{1} & x_{2}-t^{-2}n_{2}(1+t^{2}(n_{12}t+n_{2}x_{1}+x_{12}))\\
t & x_{1}+n_{1}t & n_{12}t+n_{2}x_{1}+x_{12}
\end{bmatrix}
\]
\normalsize is an isomorphism onto $N\dot{w}T^{w}N$.  Over a field
the subregular conjugacy classes of $G$ can be parametrised by
\[
\left\{ \begin{bmatrix}0 & 0 & c_{2}\\
0 & -c_{2}^{-2} & 0\\
c_{2} & 0 & c_{1}
\end{bmatrix}:c_{1},c_{2}\in\CMcal O_{S}\right\} \subset\dot{w}T^{w}N_{+},
\]
 and from part (ii) of the main Theorem and a dimension count it now
follows that they are strictly transversally intersected by the subspace
$\dot{w}T^{w}N_{+}$. 
\end{example}

\begin{example}
Consider $w=s_{1}s_{2}s_{3}s_{1}s_{2}$ in type $\mathsf{B}_{3}$
again and let $\mathfrak{N}:=\mathfrak{R}_{w}\sqcup\{\alpha_{1233}\}$.
This set is convex and although $\mathrm{DG}(b_{w}^{i})=b_{w}^{i}$,
we have $w(\alpha_{1233})\in\mathfrak{R}_{w}$ so $\mathfrak{N}$
is nimble and hence the cross section holds.
\end{example}

\begin{example}
\label{exa:spaltenstein} Consider Spaltenstein's example of the (inverse
of the) element $w=s_{1}s_{2}s_{3}s_{4}s_{5}s_{3}s_{4}s_{1}s_{2}$
in type $\mathsf{A}_{5}$ with the conjugation map corresponding to
the crossing pair $(\mathfrak{N},\mathfrak{L})=(\mathfrak{R}_{+},\varnothing)$;
by induction on $i\geq0$ one can show that 
\[
\mathrm{DGN}(b_{w}^{i+2})=b_{ws_{5}}^{\,}b_{451234312}^{i}b_{s_{5}w}^{\,},
\]
 so as $w$ is elliptic the inequality $s_{5}w\neq w_{\circ}=\mathrm{pb}(w)$
and main Lemma imply that we cannot invoke the main Theorem here.
In fact, (some of) such Coxeter elements of length 9 in $\mathsf{A}_{5}$
are the ``smallest'' elliptic examples in type $\mathsf{A}$ such
that the braid equation (\ref{eq:braid-equation}) does not hold.
One computes that
\begin{equation}
\mathrm{cross}_{w}^{i+2}(\mathfrak{R}_{+})=\mathrm{cross}_{w}^{2}(\mathfrak{R}_{+})=\{\alpha_{23},\alpha_{2345},\alpha_{3},\alpha_{345},\alpha_{5}\}\label{eq:spaltenstein}
\end{equation}
 and as $\mathrm{cross}_{w}(\cdot)$ can be interpreted as a first-order
approximation to the equations of (\ref{eq:cross-section}) (after
parametrising both sides into products of root subgroups, see Corollary
\ref{cor:cross-root-subgroups}), this suggests that there might be
a linear relationship. Indeed, employing the usual lift
\[
\dot{w}=\begin{bmatrix}0 & 0 & 0 & 0 & 0 & 1\\
0 & 0 & -1 & 0 & 0 & 0\\
1 & 0 & 0 & 0 & 0 & 0\\
0 & 0 & 0 & 0 & -1 & 0\\
0 & -1 & 0 & 0 & 0 & 0\\
0 & 0 & 0 & 1 & 0 & 0
\end{bmatrix}
\]
 and studying the resulting equations for the root subgroups of the
roots in (\ref{eq:spaltenstein}), one obtains a family of first-order
counterexamples which exponentiates to (a slight generalisation of)
Spaltenstein's counterexample
\[
\left\{ \left(\begin{bmatrix}1 & 0 & 0 & 0 & 0 & 0\\
0 & 1 & 0 & -st & 0 & -s\\
0 & 0 & 1 & s & 0 & st^{-1}\\
0 & 0 & 0 & 1 & 0 & 0\\
0 & 0 & 0 & 0 & 1 & 0\\
0 & 0 & 0 & 0 & 0 & 1
\end{bmatrix},\begin{bmatrix}0 & 0 & 0 & 0 & 0 & 1\\
0 & 0 & -1 & 0 & 0 & 0\\
1 & 0 & t^{-1} & 0 & 0 & 0\\
0 & 0 & 0 & 0 & -1 & 0\\
0 & -1 & -t & 0 & 0 & 0\\
0 & 0 & 0 & 1 & t & t^{-1}
\end{bmatrix}\right)\in N_{+}\times\dot{w}N_{w}:s,t\in\CMcal O_{S}\right\} .
\]
\end{example}

Nevertheless, the converse to (i) does not hold, as linear combinations
of the polynomials in the first-order approximation can sometimes
cancel each other out:
\begin{example}
Consider the Coxeter element $w=s_{2}s_{3}s_{4}s_{5}s_{6}s_{1}s_{2}s_{3}s_{4}s_{5}s_{3}s_{2}$
in type $\mathsf{A}_{6}$. Then $b_{w}^{i}$ is in Deligne-Garside
normal form for any $i\geq0$ so from part (iii) of the main Lemma
it again follows that equation (\ref{eq:crossing}) is never satisfied,
but a (rather lengthy) calculation shows that the conjugation map
(\ref{eq:cross-section}) corresponding to the slicing pair $(\mathfrak{N},\mathfrak{L})=(\mathfrak{R}_{+},\varnothing)$
is an isomorphism.
\end{example}

Finally, Sevostyanov and He-Lusztig proved the cross section isomorphism
(\ref{eq:cross-section}) for certain firmly convex elements $w$
with slicing pair $(\mathfrak{N},\mathfrak{L})=(\mathfrak{R}_{+}\backslash\mathfrak{R}^{w},\mathfrak{R}^{w})$,
under some extra conditions on the base ring. The relationship between
the elements they consider is explained in the prequel \cite[§2.1]{WM-mindom},
and it is proven there that all of these indeed satisfy the braid
equation (\ref{eq:braid-equation}) when $d>|\mathfrak{R}_{+}\backslash\mathfrak{R}^{w}|-\ell(w)$
\cite[Theorem B]{WM-mindom}; there are many convex (and firmly convex)
elements satisfying (\ref{eq:braid-equation}) which are contained
are in neither, yielding more transverse slices. The main conclusion
of these two papers is now obtained by combining their main theorems
with Sevostyanov's dimension calculations \cite{MR3883243} to

\begin{corollarya} The Weyl group elements considered by He-Lusztig
and the elements Sevostyanov uses to construct strictly transverse
slices (for connected reductive groups over algebraically closed fields)
are all minimally dominant, and conversely all minimally dominant
elements of conjugacy classes appearing in Lusztig's partition furnish
strictly transverse slices with natural Poisson brackets.\end{corollarya}

\paragraph{Acknowledgements}

 I thank Ian Grojnowski for comments and I am grateful to Dominic
Joyce, Balázs Szendr\H{o}i and the Mathematical Institute of the University
of Oxford for an excellent stay, where some of this paper was written.
This visit was supported by the Centre for Quantum Geometry of Moduli
Spaces at Aarhus University and the Danish National Research Foundation.
This work was also supported by EPSRC grant EP/R045038/1.

\section{Crossing roots}

A common strategy, dating back to Killing's work around 1889, is to
study reductive groups and Weyl groups through their root systems;
we follow this perspective by developing properties of $\mathrm{cross}_{w}(\cdot)$.
Throughout this paper, root systems will always be assumed to be \emph{crystallographic};
we begin by proving a lemma on decomposing sums of roots, that we
will employ several times:
\begin{defn}
We will say that a sequence $(\beta_{1},\ldots,\beta_{m})$ of roots
in a crystallographic root system is a \emph{summing sequence }if
each of the partial sums $\sum_{i=1}^{j}\beta_{i}$ is a root for
$1\leq j\leq m$, and if their total sum is denoted by $\gamma:=\sum_{i=1}^{m}\beta_{i}$
then we will denote this sequence as
\[
\sum(\beta_{1},\ldots,\beta_{m})=\gamma.
\]
If furthermore each of these roots $\beta_{i}$ lies in a subset of
roots $\mathfrak{N}$, then we will call this a \emph{summing sequence
in $\mathfrak{N}$.}
\end{defn}

The following lemma shows how summing sequences may be constructed:
\begin{lem}
\label{lem:root-sums-sequence} Let $\beta_{0},\ldots,\beta_{m}$
be roots in a crystallographic root system such that their sum $\sum_{i=0}^{m}\beta_{i}$
is also a root.

\begin{enumerate}[\normalfont(i)]

\item \cite[Proposition VI.1.19]{MR0240238} Amongst these roots
there exists a root $\beta_{j}$ such that the difference
\[
\sum_{i=0}^{m}\beta_{i}-\beta_{j}
\]
 is either a root or is zero; it is strictly positive when $m>0$
and each root $\beta_{i}$ is positive.

\item Suppose briefly that $m=3$ and that $\beta_{i}+\beta_{j}\neq0$
for each distinct pair $i,j\in\{1,2,3\}$. Then at least two of the
three sums
\[
\beta_{1}+\beta_{2},\qquad\beta_{1}+\beta_{3},\qquad\beta_{2}+\beta_{3}
\]
 are also roots.

\item Hence we may obtain from these roots $\beta_{1},\ldots,\beta_{m}$
by reordering and deleting a summing sequence
\[
\sum(\beta_{i_{1}},\ldots,\beta_{i_{m'}})=\sum_{i=1}^{m}\beta_{i}.
\]
If each of the roots $\beta_{i}$ is positive, then the resulting
sequence $\beta_{i_{1}},\ldots,\beta_{i_{m}}$ is simply a reordering
and we may choose the sequence to start with any of the $\beta_{i}$.\end{enumerate}

In particular, if each of the $\beta_{i}$ lie in a convex subset
$\mathfrak{N}$, then $\sum_{i=0}^{m}\beta_{i}$ also lies in $\mathfrak{N}$.
\end{lem}

\begin{proof}
(i): We follow the proof of \cite[Proposition VI.1.19]{MR0240238};
let's write $\gamma:=\sum_{i=1}^{m}\beta_{i}$. If $(\beta_{i},\gamma)\leq0$
for each root $\beta_{i}$, then by linearity also $(\gamma,\gamma)=(\sum_{i=1}^{m}\beta_{i},\gamma)\leq0$
which contradicts that $\gamma$ is a root. Thus there must be an
inequality $(\beta_{j},\gamma)>0$ for some $j$, which in a crystallographic
root system implies that $\gamma-\beta_{j}$ is either a root or is
zero.

(ii): By (i), we may assume after relabelling that $\beta_{1}+\beta_{2}$
is a root (as it is not zero by assumption). If $(\beta_{i},\beta_{1}+\beta_{2})\leq0$
for both $i\in\{1,2\}$ then we derive the same contradiction
\[
(\beta_{1}+\beta_{2},\beta_{1}+\beta_{2})\leq0
\]
 as before, so we may assume that say $(\beta_{1},\beta_{1}+\beta_{2})>0$.
If $\beta_{1}+\beta_{3}$ is a root the claim follows and if it is
not then $(\beta_{1},\beta_{3})\geq0$. This would yield
\[
(\beta_{1},\beta_{1}+\beta_{2}+\beta_{3})=(\beta_{1},\beta_{1}+\beta_{2})+(\beta_{1},\beta_{3})>0,
\]
which implies that $(\beta_{1}+\beta_{2}+\beta_{3})-\beta_{1}=\beta_{2}+\beta_{3}$
is a root.

(iii): The first claim of (iii) follows from (i) by induction on $m$.
As a nontrivial sum of positive roots is never zero the sequence must
be a reordering if each of the roots $\beta_{i}$ are positive. For
the final claim we induct on $m$, so we may assume that $\beta_{i_{m}}$
is the chosen root we want the sequence to start with. Writing $\gamma_{<m-1}:=\sum_{j=1}^{m-2}\beta_{i_{j}}$,
part (ii) yields that at least one of
\[
\gamma_{m-1}+\beta_{i_{m}}\qquad\textrm{or}\qquad\beta_{i_{m-1}}+\beta_{i_{m}}
\]
 is a root. In the former case we apply the induction hypothesis to
that sum to conclude, and in the latter case we apply the induction
hypothesis to the set of roots $\beta_{i_{1}},\ldots,\beta_{i_{m-2}},\beta_{i_{m-1}}+\beta_{i_{m}}$
to find a summing sequence beginning with $\beta_{i_{m-1}}+\beta_{i_{m}}$,
which immediately yields one starting with $\beta_{i_{m}}$.
\end{proof}
\begin{example}
Let $\beta,\gamma$ be a pair of roots in a root system such that
neither $\beta+\gamma$ nor $\beta-\gamma$ are roots (e.g.\ $\gamma=\beta$),
then 
\[
\beta+\gamma+(-\gamma)
\]
 is a root but none of the three sums in (ii) are roots.
\end{example}

\subsection{Crossing for braids}

In this subsection we continue developing results on decomposing sums
of roots, and as a first application we show that the definition of
$\mathrm{cross}_{w}(\cdot)$ extends to the braid monoid. We will
give another proof of this using root subgroups in Proposition \ref{prop:braid-invariance-again};
that proof is arguably more intuitive, but since the main Lemma and
its various corollaries might be of interest independent of reductive
groups I decided to include a proof using only roots.

Part (ii) and (iii) of the previous lemma yield the following corollary:
\begin{cor}
\label{cor:separate-roots} Suppose $\mathfrak{N},\mathfrak{N}'$
are two subsets of roots of a crystallographic root system, such that
$\mathfrak{N},\mathfrak{N}'$ and also their union $\mathfrak{N}\cup\mathfrak{N}'$
are convex. Suppose furthermore that $\beta$ is a root such that
$-\beta$ does not lie in $\mathfrak{N}\cup\mathfrak{N}'$. Then from
any summing sequence 
\[
\sum(\beta,\tilde{\beta}_{1},\ldots,\tilde{\beta}_{m})=:\gamma
\]
 with $\tilde{\beta}_{i}\in\mathfrak{N}\cup\mathfrak{N}'$ for each
$i$, we can obtain one
\[
\sum(\beta,\beta_{1},\ldots,\beta_{k},\beta_{1}',\ldots,\beta_{k'}')=\gamma
\]
 with $\beta_{i}\in\mathfrak{N}$ and $\beta_{i}'\in\mathfrak{N}'$,
each of them obtained from the $\tilde{\beta}_{i}$ by rearranging
and summing some of them.
\end{cor}

\begin{proof}
We induct on $m$ and within that we induct on $2\leq l\leq m$, which
is the first time $\tilde{\beta}_{l}$ lies in $\mathfrak{N}$ but
$\tilde{\beta}_{l-1}$ lies in $\mathfrak{N}'$; if there is no such
$l$ we are already done. Set $\gamma_{<l-1}:=\beta+\sum_{i=1}^{l-2}\tilde{\beta}_{i}$
and consider the sum
\[
\gamma_{<l-1}+\tilde{\beta}_{l-1}+\tilde{\beta}_{l}.
\]
If either $\gamma_{<l-1}=0$ or $\tilde{\beta}_{l-1}+\tilde{\beta}_{l}=0$
then we may shorten the sequence and conclude from the induction hypothesis
on $m$. Moreover if
\[
\beta+\sum_{i=1}^{l-2}\tilde{\beta}_{i}+\tilde{\beta}_{l-1}=\gamma_{<l-1}+\tilde{\beta}_{l-1}=0
\]
 then rewriting yields $-\beta=\sum_{i=1}^{l-2}\tilde{\beta}_{i}+\tilde{\beta}_{l-1}$.
From the last statement in part (iii) it would follow that $-\beta$
lies in the convex set $\mathfrak{N}\cup\mathfrak{N}'$. The case
$\gamma_{<l-1}+\tilde{\beta}_{l}=0$ similarly yields a contradiction.
Hence by part (ii) at least one of $\gamma_{<l-1}+\tilde{\beta}_{l}$
or $\tilde{\beta}_{l-1}+\tilde{\beta}_{l}$ is a root; in the latter
case it lies in $\mathfrak{N}\cup\mathfrak{N}'$ and we may shorten
the sequence, whereas in the former case we may now swap the roots
$\tilde{\beta}_{l-1}$ and $\tilde{\beta}_{l}$ in the sequence to
lower $l$ and apply the induction hypothesis on $l$.
\end{proof}
\begin{defn}
Let $b:=b_{\underline{w}}:=b_{w_{d}}\cdots b_{w_{1}}$ be a braid
in a braid monoid, constructed out of a sequence of elements $\underline{w}=(w_{d},\ldots,w_{1})$
lying in the corresponding twisted Weyl group. Given a positive root
$\beta$ or a subset of positive roots $\mathfrak{N}\subseteq\mathfrak{R}_{+}$,
construct the set of roots
\[
\mathrm{cross}_{b}(\beta):=\mathrm{cross}_{w_{d}}\cdots\mathrm{cross}_{w_{2}}\mathrm{cross}_{w_{1}}(\beta)
\]
and extend it to a map
\[
\mathrm{cross}_{b}:\{\textrm{subsets of }\mathfrak{R}_{+}\}\longrightarrow\{\textrm{subsets of }\mathfrak{R}_{+}\},\qquad\mathfrak{N}\longmapsto\bigcup_{\beta\in\mathfrak{N}}\mathrm{cross}_{b}(\beta).
\]

We let $\mathrm{cross}_{\underline{w}}(\gamma,\beta)$ denote the
set of sequences of elements in the $\mathfrak{R}_{w_{i}}$ ``confirming''
that $\gamma$ lies in $\mathrm{cross}_{b}(\beta)$; more precisely,
it is the set of subsets of roots
\[
\mathfrak{N}_{m}\times\cdots\times\mathfrak{N}_{2}\times\mathfrak{N}_{1}\subseteq\mathfrak{R}_{w_{m}}\times\cdots\times\mathfrak{R}_{w_{2}}\times\mathfrak{R}_{w_{1}}
\]
 with the property that the corresponding sequence $\beta=\beta_{0}',\ldots,\beta_{d}'$
inductively constructed via
\[
\beta_{j}':=w_{j}(\beta_{j-1}'+\sum_{\tilde{\beta}\in\mathfrak{\mathfrak{N}}_{j}}\tilde{\beta})
\]
for $1\leq j\leq d$, consists solely of roots which are all positive,
and satisfies $\beta_{d}'=\gamma$.
\end{defn}

Analysing these sequences shows that the set of roots $\mathrm{cross}_{b}(\cdot)$
is well-defined:
\begin{lem}
\label{lem:braid-invariance-cross} Let $\underline{w},\underline{w}'$
be two sequences of elements in a twisted Weyl group such that there
is an equality $b_{\underline{w}}=b_{\underline{w}'}$ in the associated
braid monoid, and let $\beta,\gamma$ be positive roots in its root
system. One can non-canonically construct ``transfer maps''
\[
\mathrm{cross}_{\underline{w}}(\gamma,\beta)\rightleftarrows\mathrm{cross}_{\underline{w}'}(\gamma,\beta),
\]
mapping nontrivial sequences to nontrivial sequences.

In particular, for any subset of positive roots $\mathfrak{N}$ and
any braid $b$ in the corresponding braid monoid, the set of roots
$\mathrm{cross}_{b}(\mathfrak{N})$ does not depend on the chosen
decomposition of $b$ into reduced braids.
\end{lem}

\begin{proof}
We first show the claim for a reduced decomposition $w=xy$ of an
arbitrary element $w$ in the twisted Weyl group. A (nontrivial) sequence
in $\mathrm{cross}_{(x,y)}(\gamma,\beta)$ rewrites as
\begin{equation}
\gamma=x\bigl(y(\beta+\sum_{\beta'\in\mathfrak{R}_{y}}\beta')+\sum_{\beta''\in\mathfrak{R}_{x}}\beta''\bigr)=w\bigl(\beta+\sum_{\beta'\in\mathfrak{R}_{y}}\beta'+y^{-1}(\sum_{\beta''\in\mathfrak{R}_{x}}\beta'')\bigr),\label{eq:reduced-cross}
\end{equation}
which through the identity

\begin{equation}
\mathfrak{R}_{w}=y^{-1}(\mathfrak{R}_{x})\sqcup\mathfrak{R}_{y}\label{eq:decomp-roots}
\end{equation}
 immediately yields a (nontrivial) sequence in $\mathrm{cross}_{w}(\gamma,\beta)$.
Conversely, given roots $\beta_{1},\ldots,\beta_{m}$ in $\mathfrak{R}_{w}$
such that
\[
\beta+\sum_{i=1}^{m}\beta_{i}=w^{-1}(\gamma),
\]
obtain through part (iii) of the previous lemma a summing sequence
starting with $\beta$, and according to equation (\ref{eq:decomp-roots})
each of the subsequent roots lies in one of the convex sets $v^{-1}(\mathfrak{R}_{u})$
or $\mathfrak{R}_{v}$. By the previous corollary we may modify the
summing sequence to one which starts with $\beta$, is then followed
by roots in $\mathfrak{R}_{v}$, and then in $v^{-1}(\mathfrak{R}_{u})$.
In the process, roots are only reordered or summed, so since both
$v^{-1}(\mathfrak{R}_{u})$ and $\mathfrak{R}_{v}$ consist of positive
roots and the sum of positive roots is positive, it follows through
(\ref{eq:reduced-cross}) that a nontrivial sequence in $\mathrm{cross}_{w}(\gamma,\beta)$
yields another nontrivial sequence in $\mathrm{cross}_{(u,v)}(\gamma,\beta)$. 

Now let $\underline{w}$ and $\underline{w}'$ be as in the statement
of this lemma. We may decompose both $b_{\underline{w}}$ and $b_{\underline{w}'}$
into a product of elementary braids $b_{i}$ of length one and twists;
one first verifies that 
\[
\mathrm{cross}_{\delta x}=\mathrm{cross}_{\delta}\mathrm{cross}_{x}=\mathrm{cross}_{\delta x\delta^{-1}}\mathrm{cross}_{\delta^{\,}}
\]
 for any twist $\delta$ and element $x$ in the underlying untwisted
Weyl group, so we may move all of the twists to the left and combine
them. The equality $b_{\underline{w}}=b_{\underline{w}'}$ then implies
that these twists agree and can be safely ignored. This braid identity
on the untwisted part implies that the first sequence transforms into
the second one, through a finite sequence of braid moves $s_{i}s_{j}s_{i}\cdots=s_{j}s_{i}s_{j}\cdots$.
Now applying the result in the first paragraph several times for each
such braid move, we obtain transfer maps
\[
\mathrm{cross}_{(s_{i},s_{j},s_{i},\ldots)}(\gamma,\beta)\rightleftarrows\mathrm{cross}_{s_{i}s_{j}s_{i}\cdots}(\gamma,\beta)=\mathrm{cross}_{s_{j}s_{i}s_{j}\cdots}(\gamma,\beta)\rightleftarrows\mathrm{cross}_{(s_{j},s_{i},s_{j},\ldots)}(\gamma,\beta)
\]
for any positive root $\beta$. Hence we by induction on the number
of braid moves to be made, we obtain transfer maps $\mathrm{cross}_{\underline{w}}(\gamma,\beta)\rightleftarrows\mathrm{cross}_{\underline{w}'}(\gamma,\beta)$.
\end{proof}
Transferring \emph{nontrivial} sequences will play a key rôle in the
proof of Lemma \ref{lem:dg-simple-root}.

\subsection{Crossing simple roots}

In this subsection we prove part (i) of the main Lemma. It is independent
of the previous subsection and in contrast, I do not know of a simple
interpretation or proof in terms of root subgroups.
\begin{prop}
\label{prop:root-pairs-equality} Let $\beta_{0},\beta_{1},\gamma_{0},\gamma_{1}$
be positive roots (or zero) in a crystallographic root system, such
that their sums yield an equality
\[
\beta_{0}+\beta_{1}=\gamma_{0}+\gamma_{1}
\]
 between two positive roots. Then there exists a pair of indices $i,j\in\{0,1\}$
such that
\[
\beta_{i}-\gamma_{j}
\]
 is either a positive root or is zero.
\end{prop}

\begin{proof}
Since $(\beta_{0}+\beta_{1},\gamma_{0}+\gamma_{1})>0$, we must have
$(\beta_{0}+\beta_{1},\gamma_{j})>0$ for some $j\in\{0,1\}$, and
thus $(\beta_{i},\gamma_{j})>0$ for some $i\in\{0,1\}$. In a crystallographic
root system it follows that $\beta_{i}-\gamma_{j}$ must be either
a root or zero, so if it is not a negative root we are done. If it
is negative, then the complementary pair of indices $i':=1-i$ and
$j':=1-j$ yields a positive root
\[
\beta_{i'}-\gamma_{j'}=-(\beta_{i}-\gamma_{j}).\qedhere
\]
\end{proof}
\begin{cor}
\label{cor:root-sum-split} Let $\beta_{0},\ldots,\beta_{m}$ for
$m\geq0$ and $\gamma_{0},\gamma_{1}$ be positive roots in a crystallographic
root system $\mathfrak{R}$, such that their sums yield an equality
\[
\sum_{i=0}^{m}\beta_{i}=\gamma_{0}+\gamma_{1}
\]
 between two positive roots. Then for some $0\leq j\leq m$ the smaller
sum $\sum_{i=0,i\neq j}^{m}\beta_{i}$ is also a positive root (or
zero if $m=0$) and for some $t\in\{0,1\}$ either
\[
\sum_{i=0,i\neq j}^{m}\beta_{i}-\gamma_{t}\qquad\textrm{or}\qquad\beta_{j}-\gamma_{t}
\]
 lies in $\mathfrak{R}_{+}\sqcup\{0\}$.
\end{cor}

\begin{proof}
This now follows by combining the previous proposition with Lemma
\ref{lem:root-sums-sequence}(i).
\end{proof}
This is close to what we need; the next lemma refines this statement
to show that if $\beta_{0}$ is a simple root, then we can ensure
that this root does not appear in the conclusion.
\begin{lem}
Suppose that $\alpha$ is a simple root and $\beta_{1},\ldots,\beta_{m}$
for $m\geq1$ and $\gamma$ are positive roots in a crystallographic
root system, such that
\[
\alpha+\sum_{i=1}^{m}\beta_{i}\qquad\textrm{and}\qquad\alpha+\sum_{i=1}^{m}\beta_{i}-\gamma\qquad\textrm{are both positive roots}.
\]
Furthermore suppose that for any subset $\{i_{1},\ldots,i_{k}\}$
of $\{1,\ldots,m\}$, the expression
\begin{equation}
\alpha+\sum_{j=1}^{k}\beta_{i_{j}}-\gamma\qquad\textrm{is neither a negative root nor zero}.\label{eq:assumption}
\end{equation}
Then there exists a subset $\{i_{1},\ldots,i_{k}\}$ of $\{1,\ldots,m\}$
with the property that
\[
\sum_{j=1}^{k}\beta_{i_{j}}-\gamma\qquad\textrm{is a positive root or is zero}.
\]
\end{lem}

\begin{proof}
We induct on $m$, so we presume that the claim is true $<m$. Let
$0\leq k\leq m$ be the largest integer such that both $\gamma-\sum_{j=1}^{k}\beta_{i_{j}}$
and $\alpha+\sum_{j=k+1}^{m}\beta_{i_{j}}$ are positive roots, for
some partition 
\[
\{i_{1},\ldots,i_{k}\},\qquad\{i_{k+1},\ldots,i_{m}\}
\]
of $\{1,\ldots,m\}$ into two subsets. If $k=m$ then $\gamma-\sum_{j=1}^{m}\beta_{i_{j}}\in\mathfrak{R}_{+}$,
but this combines with the initial assumption $\alpha+\sum_{i=1}^{m}\beta_{i}-\gamma\in\mathfrak{R}_{+}$
and simpleness of $\alpha$ to a contradiction, so $k<m$. By Lemma
\ref{lem:root-sums-sequence}(iii) there then exists an integer $k+1\leq t\leq m$
such that
\begin{equation}
\alpha+\sum_{j=k+1,j\neq t}^{m}\beta_{i_{j}}\qquad\textrm{is a positive root}.\label{eq:remove-beta}
\end{equation}

Now applying Proposition \ref{prop:root-pairs-equality} to the equality
of roots
\[
(\alpha+\sum_{j=k+1,j\neq t}^{m}\beta_{i_{j}})+\beta_{i_{t}}=(\alpha+\sum_{i=j}^{m}\beta_{j}-\gamma)+(\gamma-\sum_{j=1}^{k}\beta_{i_{j}})
\]
yields that at least one of the expressions
\[
\gamma-\beta_{i_{t}}-\sum_{j=1}^{k}\beta_{i_{j}},\qquad\beta_{i_{t}}+\sum_{j=1}^{k}\beta_{i_{j}}-\gamma,\qquad\alpha+\sum_{j=1,j\neq t}^{m}\beta_{i_{j}}-\gamma,\qquad\gamma-\alpha-\sum_{j=1,j\neq t}^{m}\beta_{i_{j}},
\]
 lies in $\mathfrak{R}_{+}\sqcup\{0\}$. In the first case, if this
expression equals zero we're done and if it's a positive root then
combining this with equation (\ref{eq:remove-beta}) implies that
$k$ was not maximal. The second case would yield the claim immediately.
For the third case, the assumption of (\ref{eq:assumption}) implies
that this expression can't be zero and then the claim would follow
from the induction hypothesis on $m$. The fourth case is excluded
by the same assumption.
\end{proof}
\begin{cor}
Suppose that $\alpha$ is a simple root and $\beta_{1},\ldots,\beta_{m}$
for $m\geq1$ and $\gamma_{0},\gamma_{1}$ are positive roots in a
crystallographic root system, such that there is an equality
\[
\alpha+\sum_{i=1}^{m}\beta_{i}=\gamma_{0}+\gamma_{1}
\]
 of positive roots. Then for some subset $\{i_{1},\ldots,i_{k}\}$
of $\{1,\ldots,m\}$ and some $t\in\{0,1\}$, the expression
\[
\sum_{j=1}^{k}\beta_{i_{j}}-\gamma_{t}
\]
 is a positive root or is zero.
\end{cor}

\begin{proof}
We may assume that $\gamma_{t}\neq\alpha$ for both $t\in\{0,1\}$
(otherwise the claim is immediate), and we first invoke Corollary
\ref{cor:root-sum-split} with $\beta_{0}=\alpha$. If in its conclusion
$j=0$, then the claim follows immediately as $\alpha-\gamma_{t}$
in $\mathfrak{R}_{+}\sqcup\{0\}$ would imply that $\gamma_{t}=\alpha$.
Suppose on the other hand that $j\neq0$ and that moreover $\beta_{j}-\gamma_{t}$
is not in $\mathfrak{R}_{+}\sqcup\{0\}$ (otherwise the claim follows).
Then this corollary states that for certain $\gamma_{t}$, the expression
\[
\alpha+\sum_{i=1,i\neq j}^{m}\beta_{i}-\gamma_{t}\qquad\textrm{lies in }\qquad\mathfrak{R}_{+}\sqcup\{0\}.
\]
If this expression is zero then $\beta_{j}=\gamma_{t'}$, where $t':=1-t$,
and the claim follows for $\gamma_{t'}$ so we may assume that this
expression is a positive root. In fact, from the equation
\[
\sum_{k+1}^{m}\beta_{i_{j}}-\gamma_{t'}=-\bigl(\alpha+\sum_{i=1}^{k}\beta_{i_{j}}-\gamma_{t}\bigr)
\]
 it follows may furthermore assume that condition (\ref{eq:assumption})
holds, as again the claim would otherwise follow for $\gamma_{t'}$;
then we may invoke the previous lemma, which yields the claim.
\end{proof}
\begin{notation}
Given two positive roots $\gamma$ and $\gamma'$ in some root system,
we write $\gamma'<\gamma$ if their difference $\gamma-\gamma'$ lies
in the convex cone of positive roots.
\end{notation}

Finally, we prove part (i) of the main Lemma:
\begin{proof}
As the simple root $\alpha$ is not in $\mathfrak{R}_{w}$ by assumption,
the root $w(\alpha)$ is positive. Given an integer $m\geq0$ and
some roots $\beta_{1},\ldots,\beta_{m}\in\mathfrak{R}_{w}$ such that
$w(\alpha+\sum_{i=1}^{m}\beta_{i})$ is a positive root but is not
simple, we will find roots $\beta_{1}',\ldots,\beta_{m'}'\in\mathfrak{R}_{w}$
(with $m'\leq m+1$) such that $w(\alpha+\sum_{i=1}^{m'}\beta_{i}')$
is still a positive root, but is smaller in the sense that
\[
w(\alpha+\sum_{i=1}^{m'}\beta_{i}')<w(\alpha+\sum_{i=1}^{m}\beta_{i})=:\tilde{\gamma}.
\]
By downwards induction on height, the claim then follows. Since $\tilde{\gamma}$
is not simple, we may split $\tilde{\gamma}=\tilde{\gamma}_{0}+\tilde{\gamma}_{1}$
into some other positive roots $\tilde{\gamma}_{0},\tilde{\gamma}_{1}$.

If both roots $w^{-1}(\tilde{\gamma}_{0})$ and $w^{-1}(\tilde{\gamma}_{1})$
are positive then by the previous corollary (setting $\gamma_{i}:=w^{-1}(\tilde{\gamma}_{i})$),
for at least one $t\in\{0,1\}$ there is a subset $\{i_{1},\ldots,i_{k}\}$
of $\{1,\ldots,m\}$ such that $\sum_{i=1}^{k}\beta_{i_{j}}-w^{-1}(\tilde{\gamma}_{t})$
is a positive root or is zero; in the former case, it moreover lies
in $\mathfrak{R}_{w}$ as
\[
w\bigl(\sum_{i=1}^{k}\beta_{i_{j}}-w^{-1}(\tilde{\gamma}_{t})\bigr)=\sum_{i=1}^{k}w(\beta_{i_{j}})-\tilde{\gamma}_{t}
\]
 is a sum of negative roots.

If on the other hand some $w^{-1}(\tilde{\gamma}_{t})$ is negative,
then $-w^{-1}(\tilde{\gamma}_{t})$ lies in $\mathfrak{R}_{w}$ so
the same conclusion holds with $k=0$. Hence in either case
\[
w\Bigl(\alpha+\sum_{i=k+1}^{m}\beta_{i_{j}}+\bigl(\sum_{i=1}^{k}\beta_{i_{j}}-w^{-1}(\tilde{\gamma}_{t})\bigr)\Bigr)=w\bigl(\alpha+\sum_{i=1}^{m}\beta_{i}-w^{-1}(\tilde{\gamma}_{t})\bigr)=\tilde{\gamma}_{t'}<\tilde{\gamma},
\]
where $t':=1-t$.
\end{proof}
\begin{example}
Consider $w=s_{1}s_{2}$ in type $\mathsf{B}_{2}$. Then $\alpha_{1}\notin\mathfrak{R}_{w}$
and as $\mathrm{cross}_{w}(\alpha_{1})=\{\alpha_{122},\alpha_{2}\}$
it follows that there is no ``path of simple roots'' within $\mathrm{cross}_{w}(\alpha_{1})$
from $w(\alpha_{1})=\alpha_{122}$ down to a simple root.
\end{example}

\begin{rem}
\label{rem:noncrystallographic} I am not sure whether these results
naturally generalise to the noncrystallographic case; consider for
any positive root $\beta$ the set
\[
\mathrm{cross}_{w}(\beta):=\bigl\{ w(c_{0}\beta+\sum_{i=1}^{m}c_{i}\beta_{i})\in\mathfrak{R}:c_{0},\ldots,c_{m}\in\mathbb{R}_{>0},\beta_{1},\ldots,\beta_{m}\in\mathfrak{R}_{w},m\geq0\bigr\}\cap\mathfrak{R}_{+},
\]
 where one might want to put some restrictions on the $c_{i}$. For
$\mathsf{H}_{3}$ and $\mathsf{H}_{4}$ there is a well-known ``folding''
argument showing that one can realise their root systems inside those
of $\mathsf{D}_{6}$ and $\mathsf{E}_{8}$, preserving simple roots
and embedding the corresponding reflection groups; part (i) of the
main Lemma then holds for $c_{0}=1$ and $c_{i}\in\{1,\varphi^{\pm1}\}$
for $i>0$, where $\varphi$ denotes the golden ratio. On the other
hand, if we fix $c_{0}=1$ then part (i) of the main Lemma fails for
say Coxeter elements of minimal length in type $\mathsf{I}_{2}(7)$.
For dihedral groups there are only two simple roots and in nontrivial
cases the other simple root lies in $\mathfrak{R}_{w}$; thus for
(i) a positive linear combination of the simple roots can be used
if there are no restrictions on $c_{0}$. However, the main point
of (i) was to obtain (ii) by combining it with braid invariance, but
braid invariance still fails in $\mathsf{H}_{3}$.
\end{rem}

\begin{example}
Consider $w=s_{2}s_{1}$ in type $\mathsf{H}_{3}$ and denote the
golden ratio by $\varphi$. Then $\mathfrak{R}_{w}=\{\alpha_{1},\varphi\alpha_{1}+\alpha_{2}\}$
and the only simple root in $\mathrm{cross}_{w}(\alpha_{2})$ is 
\[
\alpha_{1}=w(\varphi\alpha_{1}+\varphi\alpha_{2})=w\bigl(\alpha_{2}+\varphi^{-1}\alpha_{1}+\varphi^{-1}(\varphi\alpha_{1}+\alpha_{2})\bigr).
\]
\end{example}

\subsection{Crossing for the Deligne-Garside normal form}

In this subsection we prove the remainder of the main Lemma, Corollary
\ref{cor:inverse} and the main Proposition.

Recall (e.g.\ from \cite[§4.1]{WM-mindom}) that two elements $w_{2},w_{1}$
of a twisted Weyl group are in right Deligne-Garside normal form (modulo
moving twists) if and only if for any simple reflection $s_{i}$ with
$\ell(w_{2}s_{i})<\ell(w_{2})$ we also have $\ell(s_{i}w_{1})<\ell(w_{1})$.
Furthermore, we will repeatedly use the following property from \cite[Proposition C(ii)]{WM-mindom}:
for any element $w$ of a twisted Weyl group and integer $d\geq0$,
we have
\begin{equation}
\mathfrak{R}_{\mathrm{DG}(b_{w}^{d})}\cap\mathfrak{R}_{\mathrm{st}}^{w}=\varnothing.\label{eq:dg-fixed-roots}
\end{equation}

From part (i) of the main Lemma we deduce
\begin{cor}
\label{cor:cross-DG} Let $w_{m},\ldots,w_{1}$ be elements of a twisted
Weyl group $W$ such that their product $b_{w_{m}}\cdots b_{w_{1}}$
is in Deligne-Garside normal form, after moving twists. Then for any
$y$ in $W$ such that $w_{1}\geq y$ and any simple root $\alpha$
such that $\beta:=y^{-1}(\alpha)$ lies in $\mathfrak{R}_{+}\backslash\mathfrak{R}_{w_{1}}$,
the set
\[
\mathrm{cross}_{w_{m-1}}\cdots\mathrm{cross}_{w_{1}}(\beta)
\]
 contains simple roots not lying in $\mathfrak{R}_{w_{m}}$.
\end{cor}

\begin{proof}
By induction on $m$, it suffices to show that $\mathrm{cross}_{w_{1}}(\beta)$
contains a simple root not in $\mathfrak{R}_{w_{2}}$. By assumption
we have a reduced decomposition $w_{1}=xy$, and the root $w_{1}(\beta)=x(\alpha)$
is positive so $\alpha$ is not in $\mathfrak{R}_{x}$. By part (i)
of the main Lemma there then exists a simple root $\alpha'\in\mathrm{cross}_{x}(\alpha)$,
which means that
\[
x(\alpha+\sum_{i=1}^{m}\beta_{i})=\alpha'
\]
 for some $\beta_{1},\ldots,\beta_{m}\in\mathfrak{R}_{x}$. As $x=w_{1}y^{-1}$,
this yields the equation
\[
w_{1}\bigr(\beta+\sum_{i=1}^{m}y^{-1}(\beta_{i})\bigl)=\alpha'
\]
with each $y^{-1}(\beta_{i})\in y^{-1}(\mathfrak{R}_{x})\subseteq\mathfrak{R}_{w_{1}}$
which then implies that $\alpha'\in\mathrm{cross}_{w_{1}}(\beta)$.
The normal form condition on the pair $w_{2},w_{1}$ is equivalent
to requiring that $w_{1}^{-1}(\tilde{\alpha})$ is a negative root
for any simple root $\tilde{\alpha}$ lying in $\mathfrak{R}_{w_{2}}$.
As $w_{1}^{-1}(\alpha')=\beta+\sum_{i=1}^{m}y^{-1}(\beta_{i})$ is
a sum of positive roots and is therefore positive, it now follows
that $\alpha'$ does not lie in $\mathfrak{R}_{w_{2}}$.
\end{proof}
When $\alpha$ is not simple, there may not be simple roots in these
sets:
\begin{example}
Consider again $w=s_{1}s_{2}s_{3}s_{1}s_{2}$ in type $\mathsf{B}_{3}$.
Then $\mathrm{DGN}(b_{w}^{i})=b_{w}^{i}$ and the root $\alpha_{1233}$
does not lie in $\mathfrak{R}_{w}\sqcup\mathfrak{R}^{w}$, yet
\[
\mathrm{cross}_{w}(\alpha_{1233})=\{w(\alpha_{1233})\}=\{\alpha_{233}\}\in\mathfrak{R}_{w}.
\]
\end{example}

\begin{notation}
Recall that we write $\mathrm{DG}_{>1}(b_{w}^{d}):=b_{w}^{d}\mathrm{DG}(b_{w}^{d})^{-1}$
for the left complement to $\mathrm{DG}(b_{w}^{d})$ in $b_{w}^{d}$.
\end{notation}

\begin{cor}
Let $w$ be an element of a twisted Weyl group $W$, let $\mathfrak{N}\subseteq\mathfrak{R}_{+}$
be a subset of positive roots and pick an integer $d\geq0$.

\begin{enumerate}[\normalfont(i)]

\item If $\mathfrak{N}\subseteq\mathfrak{R}_{\mathrm{DG}(b_{w}^{d})}$
then $\mathrm{cross}_{w}^{d}(\mathfrak{N})=\varnothing$.

\item If there exist a simple root $\alpha$ in $\mathfrak{N}$ and
element $y$ in $W$ satisfying 
\[
\mathrm{DG}(b_{w}^{d})\geq y\qquad\textrm{and}\qquad y^{-1}(\alpha)\in\mathfrak{R}_{+}\backslash\mathfrak{R}_{\mathrm{DG}(b_{w}^{d})},
\]
 then $\mathrm{cross}_{w}^{d}(\mathfrak{N})$ contains simple roots.

\end{enumerate}
\end{cor}

\begin{proof}
(i): If $\mathfrak{N}\subseteq\mathfrak{R}_{\mathrm{DG}(b_{w}^{d})}$
then by Lemma \ref{lem:braid-invariance-cross} we have 
\[
\mathrm{cross}_{w}^{d}(\mathfrak{N})=\mathrm{cross}_{\mathrm{DG}_{>1}(b_{w}^{d})}\mathrm{cross}_{\mathrm{DG}(b_{w}^{d})}(\mathfrak{N})=\varnothing.
\]

(ii): Corollary \ref{cor:cross-DG} yields that there are simple roots
lying in
\[
\mathrm{cross}_{\mathrm{DG}_{>1}(b_{w}^{d})}\mathrm{cross}_{\mathrm{DG}(b_{w}^{d})}(\mathfrak{N})=\mathrm{cross}_{w}^{d}(\mathfrak{N}).\qedhere
\]
\end{proof}
From this corollary we can deduce part (ii) of the main Lemma:
\begin{proof}
The implication $\Rightarrow$ follows immediately from the first
part of the corollary. For $\Leftarrow$, suppose that the inequality
$\mathrm{DG}(b_{w}^{d})\geq w'$ does not hold. So we may suppose
that there exists a simple reflection $s_{i}$ and an element $y\in W$
such that 
\[
\mathrm{DG}(b_{w}^{d})\geq y,\qquad w'\geq s_{i}y,\qquad\mathrm{DG}(b_{w}^{d})\not\geq s_{i}y.
\]
 Then $y^{-1}(\alpha_{i})$ is a positive root lying in $\mathfrak{R}_{w'}\backslash\mathfrak{R}_{\mathrm{DG}(b_{w}^{d})}$,
so from the second part of the corollary it follows that $\mathrm{cross}_{w}^{d}(\mathfrak{R}_{w'})$
contains a simple root.
\end{proof}
\begin{lem}
\label{lem:dg-simple-root} Let $w$ be an element of a twisted Weyl
group and pick an integer $d\geq0$. Then for any simple root $\alpha$
whose orbit under $w$ consists solely of other simple roots, the
root $\mathrm{DG}(b_{w}^{d})(\alpha)$ is again simple.
\end{lem}

\begin{proof}
From equation (\ref{eq:dg-fixed-roots}) it follows that $\alpha$
is not in $\mathfrak{R}_{\mathrm{DG}(b_{w}^{d})}$, so part (i) of
the main Lemma implies that there has to exist a simple root $\alpha''$
in $\mathrm{cross}_{\mathrm{DG}(b_{w}^{d})}(\alpha)$, which means
that there exist roots $\beta_{1,1},\ldots,\beta_{m,1}\in\mathfrak{R}_{\mathrm{DG}(b_{w}^{d})}$
such that 
\[
\mathrm{DG}(b_{w}^{d})(\alpha+\sum_{i=1}^{m}\beta_{i,1})=\alpha''.
\]

In particular, if $\mathrm{DG}(b_{w}^{d})(\alpha)$ is not simple
then $m>0$. From Corollary \ref{cor:cross-DG} we similarly obtain
a sequence in $\mathrm{cross}_{\mathrm{DG}_{>1}(b_{w}^{d})}(\alpha',\alpha'')$
for some simple root $\alpha'$, which we may concatenate with the
$(\beta_{i,1})$ to a sequence $(\beta_{i,j})\in\mathrm{cross}_{\mathrm{DGN}(b_{w}^{d})}(\alpha',\alpha)$.
Since the $(\beta_{i,1})$ part of this sequence is nontrivial, Lemma
\ref{lem:braid-invariance-cross} implies that we can transfer $(\beta_{i,j})$
to a nontrivial sequence $(\beta_{i,j}')$ in $\mathrm{cross}_{w}^{d}(\alpha',\alpha)$.
This gives positive roots $\alpha_{j}'$ inductively defined for $1\leq j\leq d$
as
\[
\alpha_{j}':=w(\alpha_{j-1}'+\sum_{i}\beta_{i,j}'),\qquad\alpha_{0}':=\alpha,
\]
 with each $\beta_{i,j}'\in\mathfrak{R}_{w}$. A priori the roots
$\alpha_{j}'$ are not necessarily simple, but we now inductively
prove that they are all simple roots lying in the $w$-orbit of $\alpha$,
and that the elements $\beta_{i,j}'$ are all zero: if it's true $<j$
then the induction hypothesis yields that 
\[
\alpha_{j}'=w(\alpha_{j-1}')+\sum_{i}w(\beta_{i,j}')\in\mathfrak{R}_{+}
\]
but as $w(\alpha_{j-1}')$ is simple by assumption and each nontrivial
root $\beta_{i,j}'$ lies in $\mathfrak{R}_{w}$ this implies that
each $\beta_{i,j}'=0$ and then $\alpha'_{j}=w(\alpha_{j-1}')$. But
that is a contradiction as the sequence $(\beta_{i,j}')$ was constructed
to be nontrivial, and therefore $\mathrm{DG}(b_{w}^{d})(\alpha)$
must be a simple root.
\end{proof}
\begin{cor}
If $w$ is convex (resp.\ firmly convex), then each of the elements
in the sequence
\begin{equation}
w,\qquad\mathrm{DG}(b_{w}^{2})\,w\,\mathrm{DG}(b_{w}^{2})^{-1},\qquad\mathrm{DG}(b_{w}^{3})\,w\,\mathrm{DG}(b_{w}^{3})^{-1},\qquad\ldots,\label{eq:dg-conjugates}
\end{equation}
 of cyclic shifts is convex (resp.\ firmly convex).
\end{cor}

\begin{proof}
It was proven in \cite[Proposition 4.39]{WM-mindom} that conjugation
by $\mathrm{DG}(b_{w}^{d})$ induces a sequence of cyclic shifts,
so from \cite[Proposition A(i)]{WM-mindom} we then deduce that
\[
\mathfrak{R}_{\mathrm{st}}^{\mathrm{DG}(b_{w}^{d})w\mathrm{DG}(b_{w}^{d})^{-1}}=\mathrm{DG}(b_{w}^{d})(\mathfrak{R}_{\mathrm{st}}^{w}).
\]
If $w$ is convex then by height considerations it follows that it
must map any of the simple roots in $\mathfrak{R}_{\mathrm{st}}^{w}\cap\mathfrak{R}_{+}$
to other simple roots. The lemma now implies that $\mathrm{DG}(b_{w}^{d})(\mathfrak{R}_{\mathrm{st}}^{w})$
is also a standard parabolic subsystem.
\end{proof}
The main Proposition follows from
\begin{prop}
\label{prop:dg-inversion-sets} Let $w$ be an element of a twisted
Weyl group.

\begin{enumerate}[\normalfont(i)]

\item Let $\mathfrak{N}=\mathfrak{R}_{y}$ for some $y\in W$. Then
$\mathfrak{N}$ is $w$-nimble if and only if $y\geq w$ and $y\geq yw^{-1}$
in the weak left Bruhat-Chevalley order.

In particular, the inversion sets associated to the elements in the
sequence
\begin{equation}
w=\mathrm{DG}(b_{w}),\qquad\mathrm{DG}(b_{w}^{2}),\qquad\mathrm{DG}(b_{w}^{3}),\qquad\ldots,\label{eq:dgs}
\end{equation}
yield $w$-nimble sets. On the other hand, $\mathfrak{R}_{+}\backslash\mathfrak{R}_{\mathrm{st}}^{w}$
is $w$-nimble if and only if $w$ is convex.

\item If indeed $w$ is convex, then
\[
\mathfrak{R}_{+}\backslash\mathfrak{R}_{\mathrm{st}}^{w}+\mathfrak{R}_{\mathrm{st}}^{w}\subseteq\mathfrak{R}_{+}\backslash\mathfrak{R}_{\mathrm{st}}^{w}\qquad\textrm{and}\qquad\mathfrak{R}_{\mathrm{DG}(b_{w}^{d})}+\mathfrak{R}_{\mathrm{st}}^{w}\subseteq\mathfrak{R}_{\mathrm{DG}(b_{w}^{d})}
\]
 for any natural number $d\geq1$.

\end{enumerate}
\end{prop}

\begin{proof}
(i): The inclusion $\mathfrak{R}_{w}\subseteq\mathfrak{R}_{y}$ is
equivalent to $y\geq w$. Under this assumption, there is a reduced
decomposition $y=(yw^{-1})w$ which means 
\[
\mathfrak{R}_{y}=w^{-1}(\mathfrak{R}_{yw^{-1}})\sqcup\mathfrak{R}_{w}
\]
 and applying $w$ to this identity then yields
\[
w(\mathfrak{R}_{y})\cap\mathfrak{R}_{+}=\mathfrak{R}_{yw^{-1}}\sqcup\bigl(w(\mathfrak{R}_{w})\cap\mathfrak{R}_{+}\bigr)=\mathfrak{R}_{yw^{-1}},
\]
 which reduces the nimbleness condition to the inclusion $\mathfrak{R}_{yw^{-1}}=w(\mathfrak{R}_{y})\cap\mathfrak{R}_{+}\subseteq\mathfrak{R}_{y}$;
this yields the first two claims.

From $\mathfrak{R}_{w}\cap\mathfrak{R}_{\mathrm{st}}^{w}=\varnothing$
and $w(\mathfrak{R}_{\mathrm{st}}^{w})=\mathfrak{R}_{\mathrm{st}}^{w}$
we obtain the inclusions 
\[
\mathfrak{R}_{w}\subseteq\mathfrak{R}_{+}\backslash\mathfrak{R}_{\mathrm{st}}^{w}\qquad\textrm{and}\qquad w(\mathfrak{R}_{+}\backslash\mathfrak{R}_{\mathrm{st}}^{w})\cap\mathfrak{R}_{+}\subseteq\mathfrak{R}_{+}\backslash\mathfrak{R}_{\mathrm{st}}^{w}.
\]
 Thus $\mathfrak{R}_{+}\backslash\mathfrak{R}_{\mathrm{st}}^{w}$
is $w$-nimble if and only if $\mathfrak{R}_{+}\backslash\mathfrak{R}_{\mathrm{st}}^{w}$
is convex, so the final claim follows from \cite[Proposition 4.21]{WM-mindom}.

(ii): The first identity follows from the assumption that $\mathfrak{R}_{\mathrm{st}}^{w}$
forms a standard parabolic subsystem. For the final one, let $\beta\in\mathfrak{R}_{\mathrm{DG}(b_{w}^{d})}$
and $\gamma\in\mathfrak{R}_{\mathrm{st}}^{w}$ and suppose that $c_{0}\beta+c_{1}\gamma$
is a root for some $c_{0},c_{1}\in\mathbb{R}_{>0}$. As $\mathfrak{R}_{\mathrm{st}}^{w}$
is a standard parabolic subroot system and $\beta$ is positive, equation
(\ref{eq:dg-fixed-roots}) implies that $c_{0}\beta+c_{1}\gamma$
must be a positive root. If $\gamma$ is negative then the same equation
implies that $\mathrm{DG}(b_{w}^{d})(\gamma)$ is still negative and
therefore so is

\[
\mathrm{DG}(b_{w}^{d})(c_{0}\beta+c_{1}\gamma)=c_{0}\mathrm{DG}(b_{w}^{d})(\beta)+c_{1}\mathrm{DG}(b_{w}^{d})(\gamma).
\]
On the other hand, if $\gamma$ is positive then the previous lemma
implies that $\mathrm{DG}(b_{w}^{d})(\gamma)$ lies in the positive
half of the standard parabolic subsystem $\mathrm{DG}(b_{w}^{d})(\mathfrak{R}_{\mathrm{st}}^{w})$.
If furthermore $\mathrm{DG}(b_{w}^{d})(c_{0}\beta+c_{1}\gamma)$ is
positive then as $\mathrm{DG}(b_{w}^{d})(\beta)$ is negative it must
lie in the negative half of $\mathrm{DG}(b_{w}^{d})(\mathfrak{R}_{\mathrm{st}}^{w})$,
but then $\beta$ lies in $\mathfrak{R}_{\mathrm{st}}^{w}$ which
contradicts equation (\ref{eq:dg-fixed-roots}) again.
\end{proof}
By \cite[Proposition C(i)]{WM-mindom}, the sequence (\ref{eq:dgs})
(and hence also (\ref{eq:dg-conjugates})) stabilises after $|\mathfrak{R}_{+}\backslash\mathfrak{R}^{w}|-\ell(w)$
terms; we will reprove this at the end of this subsection.
\begin{example}
\label{exa:not-nimble} Consider $w=s_{3}s_{2}s_{1}$ and $w'=s_{2}w$
in type $\mathsf{A}_{3}$. Then $w$ is elliptic so it is convex and
\[
w_{\circ}=\mathrm{DG}(b_{w}^{3})>w'>w,
\]
but as $w(\alpha_{3})=\alpha_{2}$ we have
\[
w(\mathfrak{R}_{w'})\cap\mathfrak{R}_{+}\not\subseteq\mathfrak{R}_{w'}
\]
so $\mathfrak{R}_{w'}$ is not nimble.
\end{example}

Typically however, the sets $\mathfrak{R}_{w}\sqcup(\mathfrak{R}^{w}\cap\mathfrak{R}_{+})$
and $\mathfrak{R}_{w}\sqcup(\mathfrak{R}_{\mathrm{st}}^{w}\cap\mathfrak{R}_{+})$
are not convex:
\begin{example}
Consider $s_{2}$ in type $\mathsf{B}_{2}$. It is not firmly convex,
and
\[
\mathfrak{R}_{w}=\{\alpha_{2}\},\qquad\mathfrak{R}^{w}\cap\mathfrak{R}_{+}=\{\alpha_{12}\},\qquad\mathfrak{R}_{w}+(\mathfrak{R}^{w}\cap\mathfrak{R}_{+})=\{\alpha_{122}\}.
\]
\end{example}

\begin{example}
Consider $s_{3}s_{1}s_{2}s_{1}$ in type $\mathsf{B}_{3}$. It is
not firmly convex, and
\[
\mathfrak{R}_{w}=\{\alpha_{23}\},\qquad\mathfrak{R}_{\mathrm{st}}^{w}\cap\mathfrak{R}_{+}=\{\alpha_{1},\alpha_{12},\alpha_{2},\alpha_{123}\},\qquad\mathfrak{R}_{w}+(\mathfrak{R}_{\mathrm{st}}^{w}\cap\mathfrak{R}_{+})=\{\alpha_{12233}\}.
\]
\end{example}

\begin{lem}
Let $w$ be an element of a twisted Weyl group and pick a natural
number $d\geq0$. Then the following are equivalent:

\begin{enumerate}[\normalfont(i)]

\item The set $\mathfrak{R}_{+}\backslash(\mathfrak{R}_{\mathrm{st}}^{w}\sqcup\mathfrak{R}_{\mathrm{DG}(b_{w}^{d})})$
is empty,

\item The set $\mathfrak{R}_{+}\backslash(\mathfrak{R}_{\mathrm{st}}^{w}\sqcup\mathfrak{R}_{\mathrm{DG}(b_{w}^{d})})$
does not contain any simple roots,

\item The set $\mathrm{cross}_{w}^{d}(\mathfrak{R}_{+}\backslash\mathfrak{R}_{\mathrm{st}}^{w})$
is empty,

\item The set $\mathrm{cross}_{w}^{d}(\mathfrak{R}_{+}\backslash\mathfrak{R}_{\mathrm{st}}^{w})$
does not contain any simple roots.

\end{enumerate}
\end{lem}

\begin{proof}
(ii) $\Rightarrow$ (i): This follows from combining \cite[Lemma 3.19(i)]{WM-mindom}
with Lemma \ref{lem:dg-simple-root}.

(iv) $\Rightarrow$ (ii): If $\mathfrak{R}_{+}\backslash(\mathfrak{R}_{\mathrm{st}}^{w}\sqcup\mathfrak{R}_{\mathrm{DG}(b_{w}^{d})})$
contains simple roots then by Lemma \ref{lem:braid-invariance-cross}
and Corollary \ref{cor:cross-DG} so does 
\[
\mathrm{cross}_{w}^{d}(\mathfrak{R}_{+}\backslash\mathfrak{R}_{\mathrm{st}}^{w})=\mathrm{cross}_{\mathrm{DG}_{>1}(b_{w}^{d})}\mathrm{cross}_{\mathrm{DG}(b_{w}^{d})}(\mathfrak{R}_{+}\backslash\mathfrak{R}_{\mathrm{st}}^{w}).
\]
(i) $\Rightarrow$ (iii) follows from the same identity, and (iii)
$\Rightarrow$ (iv) is immediate.
\end{proof}
\begin{example}
Let $w$ be reflection in a non-simple root in type $\mathsf{B}_{2}$,
then $\mathfrak{R}_{w}\sqcup(\mathfrak{R}^{w}\cap\mathfrak{R}_{+})=\mathfrak{R}_{+}$.
\end{example}

\begin{example}
Consider the elements $w=s_{2}s_{3}s_{2}s_{1}$ and $w'=s_{2}s_{1}w$
in type $\mathsf{B}_{3}$. Then 
\[
w_{\circ}s_{3}=\mathrm{DG}(b_{w}^{2})>w'>w,
\]
but $\mathfrak{R}_{+}\backslash(\mathfrak{R}_{w'}\sqcup\mathfrak{R}^{w})=\{\alpha_{23},\alpha_{233}\}$.
\end{example}

We deduce part (iii) of the main Lemma:
\begin{proof}
If $w$ is convex then the claim follows from part (ii) of the main
Lemma, combined with the fact that $\mathrm{pb}(w)$ is an upper bound
for $\mathrm{DG}(b_{w}^{d})$. If on the other hand it is not convex,
then $\mathfrak{R}_{+}\backslash\mathfrak{R}_{\mathrm{st}}^{w}$ is
not convex by \cite[Proposition 4.21]{WM-mindom} again so $\mathfrak{R}_{+}\backslash(\mathfrak{R}_{\mathrm{st}}^{w}\sqcup\mathfrak{R}_{\mathrm{DG}(b_{w}^{d})})$
must be nonempty (as $\mathfrak{R}_{\mathrm{DG}(b_{w}^{d})}$ is convex),
and then the claim follows from the previous lemma.
\end{proof}
\begin{lem}
\label{lem:inverse} Let $w$ be an element of a twisted Weyl group
and let $\mathfrak{N}$ a $w$-nimble set also containing $\mathfrak{R}_{w^{-1}}$.
Then $\mathfrak{N}$ is also nimble for the element $w^{-1}$, and
for any integer $d\geq0$ we have
\[
\mathrm{cross}_{w}^{d}(\mathfrak{N})=\varnothing\qquad\textrm{if and only if }\qquad\mathrm{cross}_{w^{-1}}^{d}(\mathfrak{N})=\varnothing.
\]
\end{lem}

\begin{proof}
Pick a root $\beta$ in $\mathfrak{N}\backslash\mathfrak{R}_{w^{-1}}$.
If $\beta$ also lies in $\mathfrak{R}_{\mathrm{st}}^{w^{-1}}=\mathfrak{R}_{\mathrm{st}}^{w}$,
then by nimbleness of $w$ and positivity the root $w^{-1}(\beta)=w^{\mathrm{ord}(w)-1}(\beta)$
also lies in $\mathfrak{N}$. If on the other hand it does not lie
in $\mathfrak{R}_{\mathrm{st}}^{w^{-1}}$, then $w^{-i}(\beta)\in\mathfrak{R}_{w^{-1}}\subseteq\mathfrak{N}$
for some $i>0$. Applying $w^{i-1}$, nimbleness yields that $w^{-1}(\beta)$
also lies in $\mathfrak{N}$.

If say $\mathrm{cross}_{w}^{d}(\mathfrak{N})\neq\varnothing$ then
there exists a sequence of roots $(\beta_{i,j})\in\mathrm{cross}_{w}^{d}(\beta',\beta)$
for some roots $\beta,\beta'$ in $\mathfrak{N}$. As $-w(\mathfrak{R}_{w})=\mathfrak{R}_{w^{-1}}$
this yields another sequence
\[
-w(\beta_{d+1-i,j})\in\mathrm{cross}_{w^{-1}}^{d}(\beta,\beta'),
\]
implying that $\mathrm{cross}_{w^{-1}}^{d}(\mathfrak{N})\neq\varnothing$. 
\end{proof}
\begin{example}
Consider $w=s_{3}s_{2}s_{3}s_{1}$ in type $\mathsf{B}_{3}$. It does
not fix any roots, and the set
\[
\mathfrak{N}:=\mathfrak{R}_{w}\cup\mathfrak{R}_{w^{-1}}=\mathfrak{R}_{+}\backslash\{\alpha_{2},\alpha_{12},\alpha_{12233}\}
\]
 is nimble.
\end{example}

Corollary \ref{cor:inverse} follows from
\begin{cor}
Let $w$ be an element of a twisted Weyl group $W$ and let $y$ be
an element of $W$ such that $y\geq w^{-1}$. Then for any integer
$d\geq0$ we have
\[
\mathrm{DG}(b_{w}^{d})\geq y\qquad\textrm{if and only if }\qquad\mathrm{DG}(b_{w^{-1}}^{d})\geq y.
\]
\end{cor}

\begin{proof}
This now follows by combining the previous lemma with part (ii) of
the main Lemma.
\end{proof}
In the remainder of this subsection we reprove the bounds of \cite[Proposition C]{WM-mindom},
in the case $i=1$:
\begin{prop}
Let $w$ be an element of a twisted Weyl group and let $\mathfrak{N}$
be a nimble set of roots. Then we have a sequence of inclusions
\[
\mathfrak{N}\supseteq\mathrm{cross}_{w}(\mathfrak{N})\supseteq\mathrm{cross}_{w}^{2}(\mathfrak{N})\supseteq\cdots,
\]
which stabilises after $|\mathfrak{N}|-\ell(w)$ terms.

In particular we have $\mathrm{cross}_{w}^{d}(\mathfrak{N})=\varnothing$
for some integer $d\geq0$, if and only if this holds for all $d>|\mathfrak{N}|-\ell(w)$.
\end{prop}

\begin{proof}
For any $i\geq0$ let $\underline{w}_{i}:=(w,\ldots,w,w)$ denote
the sequence of $i$ copies of $w$. If $\gamma$ lies in $\mathrm{cross}_{w}^{d}(\mathfrak{N})$,
then there exists a root $\beta$ in $\mathfrak{N}$ such that $\mathrm{cross}_{\underline{w}_{d}}(\gamma,\beta)$
is nonempty. In other words, there exists roots $\gamma=\beta_{d}',\ldots,\beta_{0}'=\beta$
inductively constructed via
\[
\beta_{j}':=w(\beta_{j-1}'+\sum_{\tilde{\beta}\in\mathfrak{R}_{w}}\tilde{\beta})
\]
By assumption we have $\mathfrak{N}\supseteq\mathrm{cross}_{w}(\mathfrak{N})$,
so $\beta_{1}'$ lies in $\mathfrak{N}$. But then $\mathrm{cross}_{\underline{w}_{d-1}}(\gamma,\beta_{1}')$
is nonempty, which means that $\gamma$ also lies in $\mathrm{cross}_{w}^{d-1}(\mathfrak{N})$.
\end{proof}
Although the sets $\mathfrak{N}$ and $\mathrm{cross}_{w}(\mathfrak{N})$
are convex, this does not necessarily hold for the other sets appearing
in such a sequence:
\begin{example}
Consider $w=s_{1}s_{2}s_{3}s_{1}$ and let $v=w_{\circ}s_{2}s_{1}$
in type $\mathsf{B}_{3}$. Then $\mathfrak{N}:=\mathfrak{R}_{v}$
is $w$-nimble and this sequence is
\[
\mathfrak{N}\supseteq\{\alpha_{3},\alpha_{23},\alpha_{233}\}\supseteq\{\alpha_{3},\alpha_{23}\}\supseteq\{\alpha_{3},\alpha_{23}\}\supseteq\cdots
\]
\end{example}

\begin{cor}
Let $w$ be an element of a twisted Weyl group. Then for any $d\geq1$
we have

\begin{enumerate}[\normalfont(i)]

\item an inclusion
\[
\mathfrak{R}_{w}\subseteq\mathfrak{R}_{\mathrm{DG}(b_{w}^{d})}\subseteq\mathfrak{R}_{+}\backslash\mathfrak{R}_{\mathrm{st}}^{w},
\]
\item and if $d>|\mathfrak{R}_{+}\backslash\mathfrak{R}_{\mathrm{st}}^{w}|-\ell(w)$
then for any $d'\geq d$ we have
\[
\mathrm{DG}(b_{w}^{d'})=\mathrm{DG}(b_{w}^{d}).
\]
\end{enumerate}
\end{cor}

\begin{proof}
(i): If $\beta$ lies in $\mathfrak{R}_{\mathrm{st}}^{w}$ then $w^{d}(\beta)$
lies in $\mathrm{cross}_{w}^{d}(\beta)$. By part (ii) of the main
Lemma we have $\mathrm{cross}_{w}^{d}(\mathfrak{R}_{\mathrm{DG}(b_{w}^{d})})=\varnothing$,
so $\beta$ does not lie in $\mathfrak{R}_{\mathrm{DG}(b_{w}^{d})}$.

(ii): The previous proposition yields
\[
\mathrm{cross}_{w}^{d}(\mathfrak{R}_{\mathrm{DG}(b_{w}^{d'})})=\mathrm{cross}_{w}^{d'}(\mathfrak{R}_{\mathrm{DG}(b_{w}^{d'})})=\varnothing
\]
 and then we conclude from part (ii) of the main Lemma that $\mathrm{DG}(b_{w}^{d})\geq\mathrm{DG}(b_{w}^{d'})\geq\mathrm{DG}(b_{w}^{d})$.
\end{proof}

\subsection{From roots to root subgroups}

The theory of reductive groups over schemes was originally developed
by Demazure and Grothendieck \cite{MR0274460}; some simplifications
were recently made in an exposition by Conrad \cite{MR3362641}. The
main property that we will use is
\begin{thm}
Consider a split reductive group $G$ over a scheme. Trivialisations
of the root spaces of its Lie algebra exponentiate to parametrisations
$p_{\beta}:\mathbb{G}_{a}\overset{\sim}{\rightarrow}N_{\beta}$ of
its root subgroups.

\begin{enumerate}[\normalfont(i)]

\item {\cite[p.\,27]{MR73602}} For any pair of roots $\beta,\gamma$
with $\beta\neq-\gamma$, there is the Chevalley commutator formula
\begin{equation}
[p_{\beta}(c_{\beta}),p_{\gamma}(c_{\gamma})]:=p_{\beta}(c_{\beta})p_{\gamma}(c_{\gamma})p_{\beta}(-c_{\beta})p_{\gamma}(-c_{\gamma})=\prod_{i,j>0}p_{i\beta+j\gamma}(c_{\beta,\gamma}^{i,j}c_{\beta}^{i}c_{\gamma}^{j})\label{eq:steinberg-commutation}
\end{equation}
 for some global functions $c_{\beta,\gamma}^{i,j}$ on the underlying
scheme and ordering on the roots. In particular,
\[
N_{\beta}N_{\gamma}=\bigl(\prod_{i\geq0,j>0}N_{i\beta+j\gamma}\bigr)N_{\beta}
\]
\item {\cite[§6]{MR0466335}} Let $\dot{w}$ denote a lift of a Weyl
group element $w$ to $G$. Then there is the Steinberg relation
\[
\dot{w}N_{\beta}=N_{w(\beta)}\dot{w}.
\]
\end{enumerate}
\end{thm}

In order to analyse such commutators, we will focus on the roots appearing
in the product on the right-hand-side and hence define
\begin{defn}
\label{def:Cross} Let $w$ be an element of a twisted Weyl group.
Given a convex set $\mathfrak{N}$ of positive roots containing $\mathfrak{R}_{w}$,
we set
\[
\mathrm{Cross}_{w}(\mathfrak{N}):=w(\mathfrak{N}+\mathfrak{R}_{w})\cap\mathfrak{R}_{+},
\]
and we let $\mathrm{Cross}_{w}^{d}(\mathfrak{N})$ denote its $d$-th
iterate.
\end{defn}

We won't be using
\begin{prop}
Let $w$ be an element of a twisted Weyl group and let $\mathfrak{N}$
be such a set.

\begin{enumerate}[\normalfont(i)]

\item The sets $\mathrm{cross}_{w}(\mathfrak{N})$ and $\mathrm{Cross}_{w}(\mathfrak{N})$
are also convex.

\item Let $\mathfrak{L}\subseteq\mathfrak{R}\backslash\mathfrak{N}$
be a subset satisfying $w(\mathfrak{L})=\mathfrak{L}=-\mathfrak{L}$
and such that $\mathfrak{L}\sqcup\mathfrak{R}_{w}$ is convex. Then
\[
\mathrm{cross}_{w}(\mathfrak{N})\cap\mathfrak{L}=\mathrm{Cross}_{w}(\mathfrak{N})\cap\mathfrak{L}=\varnothing.
\]

\end{enumerate}
\end{prop}

\begin{proof}
(i): Let $\gamma_{0}$ and $\gamma_{1}$ be elements of $\mathrm{cross}_{w}(\mathfrak{N})$,
so $\gamma_{i}\in\mathrm{cross}_{w}(\beta_{i}')$ for some $\beta_{i}'$
in $\mathfrak{N}$. If $\gamma_{0}+\gamma_{1}$ is a root, then there
is a sum
\[
w^{-1}(\gamma_{0}+\gamma_{1})=w^{-1}(\gamma_{0})+w^{-1}(\gamma_{1})=\beta_{0}'+\beta_{1}'+\sum_{i=1}^{m}\beta_{i},\qquad\beta_{i}\in\mathfrak{R}_{w},
\]
with $\beta_{0}'$ and $\beta_{1}'$ in $\mathfrak{N}$. Since $\mathfrak{N}$
is convex and contains $\mathfrak{R}_{w}$, it follows from Lemma
\ref{lem:root-sums-sequence}(iii) that it contains the right-hand-side,
so that $\mathrm{cross}_{w}(\mathfrak{N})$ contains $\gamma_{0}+\gamma_{1}$.
The second case is analogous.

(ii): Suppose the first intersection is nonempty, so there exists
a positive root $\beta\in\mathfrak{N}$ and a root $\gamma\in\mathfrak{L}$
such that
\[
\beta+\sum_{i=1}^{m}\beta_{i}=\gamma,\qquad\beta_{i}\in\mathfrak{R}_{w},
\]
then as 
\[
\sum_{i=1}^{m}\beta_{i}-\gamma=-\beta
\]
 is a root and $-\gamma$ lies in $\mathfrak{L}$ and $\mathfrak{R}_{w}\sqcup\mathfrak{L}$
is convex it again follows from Lemma \ref{lem:root-sums-sequence}(iii)
that $-\beta$ lies in $\mathfrak{R}_{w}\sqcup\mathfrak{L}$, which
implies that $\beta$ lies in $\mathfrak{L}$. The second case is
analogous.
\end{proof}
These statements do not extend to higher iterates because these sets
might not contain all of $\mathfrak{R}_{w}$. The remainder of this
subsection is devoted to proving
\begin{lem}
Let $w$ be an element of a twisted finite Weyl group and let $\mathfrak{N}$
be a subset of positive roots. Then
\begin{equation}
\mathrm{cross}_{w}^{d}(\mathfrak{N})=\varnothing\qquad\textrm{if and only if}\qquad\mathrm{Cross}_{w}^{d}(\mathfrak{N})=\varnothing.\label{eq:vanish-iff}
\end{equation}
\end{lem}

\begin{lem}
\label{lem:split-roots} Let $\sum(\gamma_{1},\ldots,\gamma_{m})$
be a summing sequence of positive roots in a crystallographic root
system, pick any $1\leq k\leq m$ and write $\gamma_{<k}:=\sum_{j=1}^{k-1}\gamma_{j}$.
Then we may partition the set $\{k+1,\ldots,m\}$ into two subsets
\[
\{i_{1},\ldots,i_{k'}\},\qquad\{i_{1},\ldots,i_{m-k}\},
\]
 such that 
\[
\gamma_{<k}+\sum_{i=1}^{k'}\gamma_{i_{j}}\qquad\textrm{and}\qquad\gamma_{k}+\sum_{j=k'+1}^{m-k}\gamma_{i_{j}}
\]
 are both positive roots.
\end{lem}

\begin{proof}
We induct on $m-k$: applying Lemma \ref{lem:root-sums-sequence}(ii)
to the root
\[
\gamma_{<k}+\gamma_{k}+\gamma_{k+1},
\]
it follows that at least one of $\gamma_{<k}+\gamma_{k+1}$ or $\gamma_{k}+\gamma_{k+1}$
is a root. By replacing the corresponding pair with this sum, this
also shortens the sequence and then the claim follows from the induction
hypothesis.
\end{proof}
\begin{lem}
Let $\beta_{1},\ldots,\beta_{k}$ and $\gamma_{1},\ldots,\gamma_{m}$
be positive roots in a crystallographic root system such that their
sum
\[
\sum_{i=1}^{k}\beta_{i}+\sum_{i=1}^{m}\gamma_{i}=:\delta
\]
 is a root. Then we may partition the set $\{1,\ldots,m\}$ into $k$
subsets $\{i_{j,1},\ldots,i_{j,m_{j}}\}$ with $1\leq j\leq k$, such
that for each $j$ the sum
\[
\beta_{j}+\sum_{l=1}^{m_{j}}\gamma_{i_{j,l}}=:\delta_{j}
\]
 is a root, and thus the sum over those roots yields
\[
\sum_{j=1}^{k}\delta_{j}=\delta.
\]
\end{lem}

\begin{proof}
We use Lemma \ref{lem:root-sums-sequence}(iii) to construct a summing
sequence. Let $\beta_{i_{k}}$ denote the final root from the first
set of roots that appears in there, and denote the sum of the elements
before it by $\delta_{<i_{k}}$. Thus
\[
\sum(\ldots,\beta_{i_{k}},\gamma_{i_{m'}},\ldots,\gamma_{i_{m}})=\delta.
\]
The previous lemma now yields a partition 
\[
\{i_{1}',\ldots,i_{m''}'\}\coprod\{i_{m''+1}',\ldots,i_{m-m'}'\}=\{i_{m'},\ldots,i_{m}\},
\]
 such that
\[
\delta_{<i_{k}}+\sum_{j=1}^{m''}\gamma_{i_{j}'}\qquad\textrm{and}\qquad\beta_{i_{k}}+\sum_{j=m''+1}^{m-m'}\gamma_{i_{j}'}
\]
 are both roots and sum to $\delta$. Applying the induction hypothesis
on the first sum then furnishes the claim.
\end{proof}
\begin{lem}
For any root $\beta'\in\mathrm{Cross}_{w}^{d}(\beta)$, there exist
roots $\beta_{1}',\ldots,\beta_{k}'\in\mathrm{cross}_{w}^{d}(\beta)$
such that
\begin{equation}
\sum_{i=1}^{k}\beta_{i}'=\beta'.\label{eq:sum-of-roots}
\end{equation}
In particular, (\ref{eq:vanish-iff}) holds.
\end{lem}

\begin{proof}
We induct on $d$, so we may assume that there exists an integer $\tilde{k}\in\mathbb{N}_{1}$
and a root $\tilde{\beta}'$ in $\mathrm{Cross}_{w}^{d-1}(\beta)$
such that
\[
\beta'=w(\tilde{k}\tilde{\beta}'+\sum_{i=1}^{m}\gamma_{i}),\qquad\gamma_{i}\in\mathfrak{R}_{w}.
\]
The induction hypothesis furnishes $\tilde{\beta}_{1},\ldots,\tilde{\beta}_{k'}\in\mathrm{cross}_{w}^{d}(\beta)$
such that $\sum_{i=1}^{k'}\tilde{\beta}_{i}=\tilde{\beta}'$. Thus
\[
w^{-1}(\beta')=\tilde{k}\sum_{i=1}^{k'}\tilde{\beta}_{i}+\sum_{i=1}^{m}\gamma_{i}.
\]
The previous lemma now implies that we may rename the $\tilde{k}$-fold
concatenation of $\tilde{\beta}_{1},\ldots,\tilde{\beta}_{k}$ into
$\beta_{1},\ldots,\beta_{\tilde{k}k'}$ and partition the $\gamma_{1},\ldots,\gamma_{m}$
such that for $1\leq j\leq\tilde{k}k'$ there are roots
\[
\tilde{\beta}_{j}':=\beta_{j}+\sum\gamma_{i,j},\qquad\gamma_{i,j}\in\mathfrak{R}_{w}
\]
not lying in $\mathfrak{R}_{w}$ (otherwise we add them to the list
of $\gamma$'s and start over with a smaller list of $\beta$'s).
Setting $\beta_{j}':=w(\tilde{\beta}_{j}')$ and $k:=\tilde{k}k'$,
we have 
\[
\sum_{i=1}^{k}\beta_{j}'=\sum_{i=1}^{\tilde{k}k'}w(\beta_{j}+\sum\gamma_{i,j})=w(\tilde{k}\sum_{i=1}^{k'}\tilde{\beta}_{i}+\sum_{i=1}^{m}\gamma_{i})=\beta'
\]
 so (\ref{eq:sum-of-roots}) is satisfied, and they lie in $\mathrm{cross}_{w}^{d}(\beta)$.
\end{proof}

\section{\label{sec:cross-section}Cross sections and transversality}

In this section we prove part (i) and (ii) of the main Theorem. 

In the first subsection we prove part (i). The main part of He-Lusztig's
proof employs the existence of certain ``good'' elements in each
elliptic conjugacy class of the Weyl group \cite{MR1425324}. Combining
this part with the geometric construction of such elements in \cite[§5.2]{MR2999317}
and unravelling the resulting proof, one finds that it is very similar
to Sevostyanov's (in the elliptic case). As He-Lusztig's techniques
are neater and yield an explicit inverse map, the proof of this subsection
is based upon their approach.

More specifically, He-Lusztig constructed a candidate inverse $\Psi'$
to the conjugation map $\Psi$ when the Weyl group element $w$ is
elliptic, and proved that $\Psi'\circ\Psi$ is the identity when $b_{w}^{d}$
is divisible by $b_{w_{\circ}}$ for some natural number $d$ \cite[§3.7]{MR2904572}.
The core of their argument states that a certain variety with a projection
map constructed out of root subgroups and sequences of Weyl group
elements only depends on the image of this sequence in the braid monoid
\cite[§2.9]{MR2904572}, and this argument can be generalised to work
in the nonelliptic case when $L$ is nontrivial. More directly however,
we observe that the rôle that these Weyl group elements play here
is in asserting the identity 
\[
\mathrm{cross}_{w}^{d}(\mathfrak{N})=\varnothing,
\]
which through Lemma \ref{lem:inverse} is crucial in our approach
to proving transversality. Rather surprisingly, the previous section
demonstrated that such equations about roots are equivalent to similar
identities about braids. Hence in the first subsection we've rewritten
this part of their proof in terms of $w$-crossing pairs (for arbitrary
$w$) satisfying this equation, yielding many new cross sections along
the way.

Rather than following this up with a proof that $\Psi\circ\Psi'$
also equals the identity, He-Lusztig then appeal to Ax-Grothendieck
type results about affine $n$-space to conclude that $\Psi'$ is
indeed inverse to $\Psi$, under suitable conditions on the base ring
and its ring automorphism. However, for nonelliptic $w$ the slices
$\dot{w}LN_{w}$ are not isomorphic to affine $n$-space; the following
proof shows directly that $\Psi\circ\Psi'$ is the identity, shedding
any conditions on the base ring and its automorphism.

In the second subsection, we will also prove that part (ii) implies
the following variant on part (i):\emph{
\begin{enumerate}[\normalfont(i')]
\setcounter{enumi}{0}
\item The conjugation action (\ref{eq:cross-section}) is an isomorphism when restricted to a first order infinitesimal neighbourhood of the subscheme $\{\mathrm{id}\}\times\dot{w}LN_{w}$.
\end{enumerate}
}A priori (i') is weaker; I have not studied whether they might be
equivalent.

\subsection{\label{subsec:crossing-root-subgroups}Crossing root subgroups}

The following construction was inspired by \cite[§2.7]{MR2904572}:
\begin{defn}
\label{def:construction} Fix an integer $d\geq1$. We consider the
set of orbits in the $d$-fold Cartesian product
\[
N(\underline{\dot{w}L}):=N\dot{w}LN\times\cdots\times N\dot{w}LN
\]
for the $N^{d-1}$-action given by
\[
(n_{d-1},\ldots,n_{1})\cdot(g_{d},\ldots,g_{1})=(g_{d}^{\phantom{1}}n_{d-1}^{\phantom{-1}},n_{d-1}^{-1}g_{d-1}^{\phantom{1}}n_{d-2}^{\phantom{1}},\ldots,n_{2}^{-1}g_{2}^{\phantom{1}}n_{1}^{\phantom{1}},n_{1}^{-1}g_{1}^{\phantom{1}}).
\]
We denote the (naive) orbit space by $N[\underline{\dot{w}L}]$ and
the quotient map by
\begin{equation}
N(\underline{\dot{w}L})\longrightarrow N[\underline{\dot{w}L}],\qquad(g_{d},\ldots,g_{1})\longmapsto[g_{d},\ldots,g_{1}].\label{eq:quotient-map}
\end{equation}
This map is equivariant with respect to the $N\times N$-actions on
$N(\underline{\dot{w}L})$ and $N[\underline{\dot{w}L}]$ coming from
left and right multiplication on the outer factors:
\[
n'(g_{d},\ldots,g_{1})n:=(n'g_{d},\ldots,g_{1}n),\qquad n'[g_{d},\ldots,g_{1}]n:=[n'g_{d},\ldots,g_{1}n].
\]
Given an element $g\in N\dot{w}LN$, we shall write
\[
[\underline{g}]:=[g,g,\ldots,g,g]\in N[\underline{\dot{w}L}]
\]
for the image of $(g,g,\ldots,g,g)\in N(\underline{\dot{w}L})$ under
the quotient map (\ref{eq:quotient-map}).

We can describe $N[\underline{\dot{w}L}]$ more explicitly: first
consider the $d$-fold product
\[
N_{\underline{\dot{w}L}}:=\dot{w}LN_{w}\times\cdots\times\dot{w}LN_{w}
\]
 and enlarge it to $N_{\underline{\dot{w}L}}^{+}:=N\times N_{\underline{\dot{w}L}}^{\,}$.
By multiplying the first two components of its $d+1$-Cartesian product,
we obtain a natural inclusion
\begin{equation}
N_{\underline{\dot{w}L}}^{+}=N\times\dot{w}LN_{w}\times\dot{w}LN_{w}\times\cdots\times\dot{w}LN_{w}\overset{\sim}{\longrightarrow}N\dot{w}LN_{w}\times\dot{w}LN_{w}\times\cdots\times\dot{w}LN_{w}\hooklongrightarrow N(\underline{\dot{w}L}).\label{eq:inclusion}
\end{equation}
\end{defn}

\begin{notation}
We write $N^{w}:=w^{-1}Nw\cap N$ for the product of root subgroups
corresponding to the roots in $\mathfrak{N}\backslash\mathfrak{R}_{w}$,
and given elements $x,g\in G$ we abbreviate left conjugation by $^{x}g:=xgx^{-1}$. 
\end{notation}

\begin{lem}
\begin{enumerate}[\normalfont(i)] \label{lem:factorise}

\item There is a natural factorisation
\begin{equation}
L\times N\overset{\sim}{\longrightarrow}LN=L\overset{\sim}{\longrightarrow}N^{w}\times L\times N_{w},\qquad(l,n)\longmapsto(n_{1},l',n_{2}),\label{eq:factorisation}
\end{equation}
implying $N\dot{w}LN=N\dot{w}LN_{w}$; if $\bigl((\mathfrak{N}\backslash\mathfrak{R}_{w})+\mathfrak{L}\bigl)\cap\mathfrak{L}=\varnothing$
(e.g.\ $\mathfrak{L}$ is a standard parabolic subsystem) then $l'=l$.

\item Hence the inclusion (\ref{eq:inclusion}) yields an algebraic
cross section 
\[
\xymatrix{ & N(\underline{\dot{w}L})\ar@{>>}[d]\\
N_{\underline{\dot{w}L}}^{+}\ar[r]^{\sim}\ar@{^{(}->}[ur] & N[\underline{\dot{w}L}]
}
\]
of the quotient map (\ref{eq:quotient-map}).

\item Assume that $L\subseteq T$, pick a root $\beta\in\mathfrak{N}$,
elements $n\in LN_{w}$ and $m\in N_{\beta}$, and use (\ref{eq:factorisation})
to factorise $nm\in LN$ into a pair of elements $(m_{1},n_{1})$
in $N^{w}\times LN_{w}$. Then 
\[
^{\dot{w}}m_{1}\in\prod_{\gamma\in\mathrm{Cross}_{w}(\beta)}N_{\gamma}\subseteq N.
\]
\end{enumerate}
\end{lem}

\begin{proof}
(i): Since $\mathfrak{R}_{w}\sqcup\mathfrak{L}$ is convex and $\mathfrak{L}=-\mathfrak{L}$
and $\mathfrak{L}\cap\mathfrak{N}=\varnothing$, it follows that 
\[
\bigl((\mathfrak{N}\backslash\mathfrak{R}_{w})+\mathfrak{L}\bigr)\cap\mathfrak{R}_{w}=\varnothing,
\]
which through convexity of $\mathfrak{N}\sqcup\mathfrak{L}$ yields
$LN^{w}=N^{w}L$, and the first claim follows. Nimbleness then implies
that 
\[
N\dot{w}LN=N\dot{w}N^{w}LN_{w}=N\dot{w}LN_{w}.
\]
(ii) The case of $d=2$ now follows from

\[
N\dot{w}LN\underset{N}{\times}N\dot{w}LN=N\dot{w}LN\underset{N}{\times}N\dot{w}LN_{w}=N\dot{w}LN\underset{N}{\times}\dot{w}LN_{w}\simeq N\dot{w}LN\times\dot{w}LN_{w}=N\dot{w}LN_{w}\times\dot{w}LN_{w}
\]
 and this implies, by induction on $d$, that
\[
N[\underline{\dot{w}L}]=N\dot{w}LN\underset{N}{\times}N\dot{w}LN\underset{N}{\times}\cdots\underset{N}{\times}N\dot{w}LN=N\dot{w}LN_{w}\times\dot{w}LN_{w}\times\cdots\times\dot{w}LN_{w}\simeq N_{\underline{\dot{w}L}}^{+}.
\]

(iii): Follows similarly.
\end{proof}
\begin{notation}
We now denote by $\mathrm{Cross}_{\dot{w}L}^{d}$ the composition
\[
N[\underline{\dot{w}L}]\overset{\sim}{\longrightarrow}N_{\underline{\dot{w}L}}^{+}=N\times N_{\underline{\dot{w}L}}\twoheadlongrightarrow N
\]
 of the inverse of this isomorphism with projection onto the first
component of the Cartesian product.
\end{notation}

\begin{cor}
\label{cor:cross-root-subgroups} Assume that $L\subseteq T$, pick
an integer $d>0$ and a positive root $\beta$ in $\mathfrak{N}$,
fix an element $h\in N[\underline{\dot{w}L}]$ and consider the morphism
of schemes
\[
\mathbb{G}_{a}\simeq N_{\beta}\longrightarrow N,\qquad m\longmapsto\mathrm{Cross}_{\dot{w}L}^{d}(hm).
\]

\begin{enumerate}[\normalfont(i)]

\item Then
\[
\mathrm{Cross}_{\dot{w}L}^{d}(hm)\in\mathrm{Cross}_{\dot{w}L}^{d}(h)\prod_{\gamma\in\mathrm{Cross}_{w}(\beta)}N_{\gamma}
\]
\item If $\mathrm{Cross}_{\dot{w}L}^{d}(h)$ is the identity, then
the derivative of this map at the identity of $N_{\beta}$
\[
\mathfrak{n}_{\beta}\longrightarrow\mathfrak{n},\qquad x\longmapsto\mathrm{cross}_{\dot{w}L}^{d}(hx)
\]
 satisfies
\[
\mathrm{cross}_{\dot{w}L}^{d}(hx)\in\bigoplus_{\gamma\in\mathrm{cross}_{w}^{d}(\beta)}\mathfrak{n}_{\gamma}.
\]
\end{enumerate}
\end{cor}

\begin{proof}
(i): Denote the inverse of the element $h$ under the isomorphism
\[
N\times\dot{w}LN_{w}\times\cdots\times\dot{w}LN_{w}=N_{\underline{\dot{w}L}}^{+}\overset{\sim}{\longrightarrow}N[\underline{\dot{w}L}]
\]
 by $(m',\dot{w}n_{d},\ldots,\dot{w}n_{1})$, so $m'$ lies in $N$
and each $n_{i}$ lies in $LN_{w}$. Let $(m_{1},n_{1}')\in N^{w}\times LN_{w}$
be the factorisation of $n_{1}m\in LN$ in (\ref{eq:factorisation})
and inductively define for $1<i\leq d$ elements $(m_{i},n_{i}')\in N^{w}\times LN_{w}$
as the factorisation of the element $n_{i-1}({}^{\dot{w}}m_{i-1})\in LN$.
By induction on $d$, the second part of the previous proposition
implies that
\[
^{\dot{w}}m_{i}\in\prod_{\gamma\in\mathrm{Cross}_{w}^{i}(\beta)}N_{\gamma}.
\]
 Then
\begin{align*}
hm & =[m',\dot{w}n_{d},\ldots,\dot{w}n_{2},\dot{w}n_{1}m]\\
 & =[m',\dot{w}n_{d},\ldots,\dot{w}n_{2}({}^{\dot{w}}m_{1}),\dot{w}n_{1}']\\
 & =[m'(^{\dot{w}}m_{d}),\dot{w}n_{d}',\ldots,\dot{w}n_{2}',\dot{w}n_{1}']
\end{align*}
 so that 
\[
\mathrm{Cross}_{\dot{w}L}^{d}\big(hm\big)=m'(^{\dot{w}}m_{d})=\mathrm{Cross}_{\dot{w}L}^{d}(h)(^{\dot{w}}m_{d}).
\]
(ii): Taking derivatives with respect to the first subgroup in (\ref{eq:steinberg-commutation}),
the component on the right-hand-side with $i>1$ vanishes. This implies
that for $d=1$, the image lands in 
\[
\dot{w}\bigl(\bigoplus_{\gamma\in\{\beta+\sum_{i=1}^{m}\beta_{i}\in\mathfrak{R}_{+}:\beta_{i}\in\mathfrak{R}_{w}\}}\mathfrak{n}_{\gamma}\bigr)\dot{w}^{-1}\cap\mathfrak{n}_{+}=\bigoplus_{\gamma\in\mathrm{cross}_{w}(\beta)}\mathfrak{n}_{\gamma}.
\]
The claim then follows by induction.
\end{proof}
\begin{prop}
\label{prop:braid-invariance-again} Both $\mathrm{cross}_{w}(\cdot)$
and $\mathrm{Cross}_{w}(\cdot)$ lift to the braid monoid.
\end{prop}

\begin{proof}
For any reduced decomposition $w=uv$ we have
\[
\dot{w}N_{w}=\dot{u}N_{u}\dot{v}N_{v}
\]
 regardless of characteristic, and hence
\[
\prod_{\beta'\in\mathrm{Cross}_{w}(\beta)}N_{\beta'}\times\dot{w}N_{w}=\dot{w}N_{w}N_{\beta}=\dot{u}N_{u}\dot{v}N_{v}N_{\beta}=\prod_{\beta'\in\mathrm{Cross}_{u}(\mathrm{Cross}_{v}(\beta))}N_{\beta'}\times\dot{w}N_{w},
\]
so that
\[
\mathrm{Cross}_{w}(\beta)=\mathrm{Cross}_{u}(\mathrm{Cross}_{v}(\beta)).
\]
Taking derivatives as before then yields
\[
\mathrm{cross}_{w}(\beta)=\mathrm{cross}_{u}(\mathrm{cross}_{v}(\beta)).
\]
The rest of the proof is analogous to that of Lemma \ref{lem:braid-invariance-cross}.
\end{proof}
We now prove the crucial
\begin{lem}
\label{lem:cross-empty} If
\[
\mathrm{cross}_{w}^{d}(\beta)=\varnothing,
\]
 then for all $h\in N[\underline{\dot{w}L}]$ and $n_{\beta}\in N_{\beta}$
we have
\[
\mathrm{Cross}_{\dot{w}L}^{d}(hn_{\beta})=\mathrm{Cross}_{\dot{w}L}^{d}(h).
\]
\end{lem}

\begin{proof}
Consider the notions introduced in the first paragraph of Definition
\ref{def:construction}; we add a tilde to denote the orbits for $NL$
instead. Then from $\dot{w}L=L\dot{w}$ and $LN_{w}=N_{w}L$ we similarly
derive a factorisation
\[
N\times L\times\dot{w}N_{w}\times\cdots\times\dot{w}N_{w}=N\times L\times N_{\underline{\dot{w}}}=:\tilde{N}_{\underline{\dot{w}L}}^{+}\overset{\sim}{\longrightarrow}\tilde{N}[\underline{\dot{w}L}]
\]
yielding projection maps
\begin{align*}
\widetilde{\mathrm{Cross}}_{\dot{w}L}^{d}:\quad & \tilde{N}[\underline{\dot{w}L}]\overset{\sim}{\longrightarrow}\tilde{N}_{\underline{\dot{w}L}}^{+}=N\times L\times N_{\underline{\dot{w}}}\twoheadlongrightarrow N,\\
\mathrm{proj}_{L}:\quad & N[\underline{\dot{w}L}]\longrightarrow\tilde{N}[\underline{\dot{w}L}]\overset{\sim}{\longrightarrow}\tilde{N}_{\underline{\dot{w}L}}^{+}=N\times L\times N_{\underline{\dot{w}}}\twoheadlongrightarrow L.\\
\mathrm{proj}_{\dot{w}}:\quad & N[\underline{\dot{w}L}]\longrightarrow\tilde{N}[\underline{\dot{w}L}]\overset{\sim}{\longrightarrow}\tilde{N}_{\underline{\dot{w}L}}^{+}=N\times L\times N_{\underline{\dot{w}}}\twoheadlongrightarrow N_{\underline{\dot{w}}}.\\
\mathrm{proj}_{\dot{w}L}:\quad & N[\underline{\dot{w}L}]\overset{\sim}{\longrightarrow}N_{\underline{\dot{w}L}}^{+}=N\times N_{\underline{\dot{w}L}}\twoheadlongrightarrow N_{\underline{\dot{w}L}}.
\end{align*}
This gives a natural commutative diagram
\[
\xymatrix{N(\underline{\dot{w}L})\ar@{>>}[d]\ar@{>>}[dr]\\
N[\underline{\dot{w}L}]\ar@{>>}[r]\ar[d]_{\mathrm{Cross}_{\dot{w}L}^{d}} & \tilde{N}[\underline{\dot{w}L}]\ar[d]^{\widetilde{\mathrm{Cross}}_{\dot{w}L}^{d}}\\
N\ar[r]^{=} & N
}
\]
Write $h_{L}:=\mathrm{proj}_{L}(h)$, $h_{\dot{w}}:=\mathrm{proj}_{N_{\underline{\dot{w}}}}(h)$
and $h_{\dot{w}L}:=\mathrm{proj}_{\dot{w}L}(h)$. Then
\begin{align*}
\mathrm{Cross}_{\dot{w}L}^{d}(hn_{\beta}) & =\mathrm{Cross}_{\dot{w}L}^{d}(h)\mathrm{Cross}_{\dot{w}L}^{d}(h_{\dot{w}L}n_{\beta}),\\
 & =\mathrm{Cross}_{\dot{w}L}^{d}(h)\widetilde{\mathrm{Cross}}_{\dot{w}}^{d}(h_{L}h_{\dot{w}}n_{\beta})\\
 & =\mathrm{Cross}_{\dot{w}L}^{d}(h)h_{L}^{-1}\widetilde{\mathrm{Cross}}_{\dot{w}}^{d}(h_{\dot{w}}n_{\beta})h_{L}
\end{align*}
so the claim follows from part (ii) of the previous statement.
\end{proof}
\begin{rem}
Instead of considering orbits for $NL$, we could have also upgraded
$\mathrm{Cross}_{w}(\cdot)$ to a version $\mathrm{Cross}_{w\mathfrak{L}}(\cdot)$
taking into account the roots in $\mathfrak{L}$, and then proven
that $\mathrm{cross}_{w}^{d}(\beta)=\varnothing$ if and only if $\mathrm{Cross}_{w\mathfrak{L}}(\beta)=\varnothing$
purely by studying roots. By combining this with
\[
\mathrm{Cross}_{\dot{w}L}^{d}(hm)\in\mathrm{Cross}_{\dot{w}L}^{d}(h)\prod_{\gamma\in\mathrm{Cross}_{w\mathfrak{L}}(\beta)}N_{\gamma},
\]
 for $L$ arbitrarily, the previous lemma can be obtained almost entirely
by analysing roots, but the arguments become a bit longer and perhaps
less transparent, despite being essentially identical.
\end{rem}

We can now prove (i) of the main Theorem:
\begin{proof}
Set $d>0$ such that (\ref{eq:crossing}) holds, then we construct
the (algebraic) inverse $\Psi'$ to the conjugation map
\[
\Psi:N\times\dot{w}LN_{w}\longrightarrow N\dot{w}LN=N\dot{w}LN_{w},\qquad(n,\dot{w}\tilde{n})\longmapsto n^{-1}\dot{w}\tilde{n}n
\]
as follows. For an element $\tilde{g}\in N\dot{w}LN_{w}$ we set
\begin{equation}
n_{\tilde{g}}:=\mathrm{Cross}_{\dot{w}L}^{d}\bigl([\underline{\tilde{g}}]\bigr)^{-1}\in N.\label{eq:def-n}
\end{equation}
Denoting the image of $n_{\tilde{g}}^{\,}\tilde{g}n_{\tilde{g}}^{-1}\in N\dot{w}LN_{w}$
under the inverse of the usual multiplication map
\begin{equation}
N\times\dot{w}LN_{w}\overset{\sim}{\longrightarrow}N\dot{w}LN_{w},\qquad(\tilde{m},\dot{w}\tilde{n})\longmapsto\tilde{m}\dot{w}\tilde{n}\label{eq:factor}
\end{equation}
 by $(n',g_{\tilde{g}})$, now set $\Psi'(\tilde{g}):=(n_{\tilde{g}},g_{\tilde{g}})$.
We will calculate these elements more explicitly and see that equation
(\ref{eq:crossing}) implies that $n'=\mathrm{id}$.

$\Psi'\circ\Psi=\mathrm{id}$: Pick $(n,g)\in N\times\dot{w}LN_{w}$
and set $\tilde{g}:=\Psi(n,g):=n^{-1}gn\in N\dot{w}LN_{w}$. Then
\[
[\underline{\tilde{g}}]=[n^{-1}gn,n^{-1}gn,\ldots,n^{-1}gn]=n^{-1}[\underline{g}]n,
\]
so as $\mathrm{Cross}_{\dot{w}L}^{d}\bigl([\underline{g}]\bigr)=e$,
Lemma \ref{lem:cross-empty} implies that
\[
n_{\tilde{g}}^{-1}=\mathrm{Cross}_{\dot{w}L}^{d}\bigl([\underline{\tilde{g}}]\bigl)=\mathrm{Cross}_{\dot{w}L}^{d}\big(n^{-1}[\underline{g}]n\big)=\mathrm{Cross}_{\dot{w}L}^{d}\big(n^{-1}[\underline{g}]\big)=\mathrm{Cross}_{\dot{w}L}^{d}\big([n^{-1}g,g,\ldots,g]\big)=n^{-1}.
\]
Hence $n_{\tilde{g}}=n$, and thus
\[
n_{\tilde{g}}^{\,}\tilde{g}n_{\tilde{g}}^{-1}=n\tilde{g}n^{-1}=g
\]
which lies in $\dot{w}LN_{w}$ by assumption, so $g_{\tilde{g}}=g$
and therefore $(\Psi'\circ\Psi)(n,g)=(n_{\tilde{g}},g_{\tilde{g}})=(n,g)$.

$\Psi\circ\Psi'=\mathrm{id}$: Pick $\tilde{g}\in N\dot{w}LN_{w}$
and use (\ref{eq:factor}) to decompose $\tilde{g}=m\dot{w}n$ for
some $m\in N$ and $n\in LN_{w}$. Let $(m_{i},n_{i})\in N^{w}\times LN_{w}$
be the factorisation of $nm\in LN$ in (\ref{eq:factorisation}) for
$i=1$ and inductively for $i>1$ as the factorisation of $n_{i-1}({}^{\dot{w}}m_{i-1})\in LN$.
These elements were constructed to obtain the inverse image of $[\underline{\tilde{g}}]$
under the isomorphism $N_{\underline{\dot{w}L}}^{+}\overset{\sim}{\rightarrow}N[\underline{\dot{w}L}]$,
as
\begin{align*}
[\underline{\tilde{g}}] & =[m\dot{w}n,m\dot{w}n,\ldots,m\dot{w}n,m\dot{w}n]\\
 & =[m\dot{w}nm,\dot{w}nm,\ldots,\dot{w}nm,\dot{w}n]\\
 & =[m(^{\dot{w}}m_{1})\dot{w}n_{1},(^{\dot{w}}m_{1})\dot{w}n_{1},\ldots,(^{\dot{w}}m_{1})\dot{w}n_{1},\dot{w}n]\\
 & =[m\bigl({}^{\dot{w}}(m_{1}\cdots m_{d-1})\bigl)\dot{w}n_{d-1},\dot{w}n_{d-2},\ldots,\dot{w}n_{1},\dot{w}n].
\end{align*}
In particular, this yields
\[
n_{\tilde{g}}^{-1}=\mathrm{Cross}_{\dot{w}L}^{d}\big([\underline{\tilde{g}}]\big)=m\bigl({}^{\dot{w}}(m_{1}\cdots m_{d-1})\bigl).
\]
A similar calculation furnishes that $m_{d}=\mathrm{id}$: briefly
setting
\[
a:=\mathrm{Cross}_{\dot{w}L}^{d}\bigl([\underline{\tilde{g}}]\bigr)^{-1}[\underline{\tilde{g}}]=[\dot{w}n_{d-1},\dot{w}n_{d-2},\ldots,\dot{w}n_{1},\dot{w}n]\in N[\underline{\dot{w}L}],
\]
 we have by construction
\begin{align*}
am & =[\dot{w}n_{d-1},\dot{w}n_{d-2},\ldots,\dot{w}n_{1},(^{\dot{w}}m_{1})\dot{w}n_{1}]\\
 & =[\dot{w}n_{d-1},\dot{w}n_{d-2},\ldots,(^{\dot{w}}m_{2})\dot{w}n_{2},\dot{w}n_{1}]\\
 & =[(^{\dot{w}}m_{d})\dot{w}n_{d},\dot{w}n_{d-1},\ldots,\dot{w}n_{2},\dot{w}n_{1}]
\end{align*}
and then Lemma \ref{lem:cross-empty} implies that 
\[
^{\dot{w}}m_{d}=\mathrm{Cross}_{\dot{w}L}^{d}(am)=\mathrm{Cross}_{\dot{w}L}^{d}(a)=\mathrm{id}.
\]
We now obtain an expression for $\Psi'(\tilde{g})$ by computing
\begin{align*}
n_{\tilde{g}}\tilde{g}n_{\tilde{g}}^{-1} & =\bigl({}^{\dot{w}}(m_{d-1}^{-1}\cdots m_{1}^{-1})\bigl)m^{-1}m\dot{w}nm\bigl({}^{\dot{w}}(m_{1}\cdots m_{d-1})\bigl)\\
 & =\bigl({}^{\dot{w}}(m_{d-1}^{-1}\cdots m_{1}^{-1})\bigl)\dot{w}m_{1}n_{1}\bigl({}^{\dot{w}}(m_{1}\cdots m_{d-1})\bigl)\\
 & =\bigl({}^{\dot{w}}(m_{d-1}^{-1}\cdots m_{1}^{-1})\bigl)\dot{w}m_{1}m_{2}\cdots m_{d}n_{d}\\
 & =\dot{w}m_{d}n_{d}=\dot{w}n_{d}
\end{align*}
 which already lies in $\dot{w}LN_{w}$. Thus $g_{\tilde{g}}=\dot{w}n_{d}=n_{\tilde{g}}^{\,}\tilde{g}n_{\tilde{g}}^{-1}$,
and hence $(\Psi\circ\Psi')(\tilde{g})=\Psi(n_{\tilde{g}}^{\,},n_{\tilde{g}}^{\,}\tilde{g}n_{\tilde{g}}^{-1})=\tilde{g}$.
\end{proof}
\begin{rem}
\begin{enumerate}[(i)]

\item We could have written the same proof with $\tilde{N}(\underline{\dot{w}L})$
(as defined in the proof of Lemma \ref{lem:cross-empty}) instead
of $N(\underline{\dot{w}L})$; nothing changes except for the final
paragraph, where more factorising is required.

\item For convenience, let's briefly add a $d$ to the notions introduced
in Definition \ref{def:construction}, so $_{d}N(\underline{\dot{w}L}):=N(\underline{\dot{w}L})$,
etc. The previous proof easily generalises to show that if (\ref{eq:crossing})
holds for some $d$, then
\[
N\times{}_{i}N_{\underline{\dot{w}L}}^{\,}\longrightarrow{}_{i}N[\underline{\dot{w}L}],\qquad\bigl(n,[g_{i},\ldots,g_{1}]\bigr)\longmapsto[n^{-1}g_{i},\ldots,g_{1}n]
\]
 is an isomorphism for any $i\geq1$; this proof was just the case
$i=1$.

\item If $\mathfrak{L}$ is a standard parabolic subsystem then Lemma
\ref{lem:factorise}(i) implies that the $L$-component of image of
$n\dot{w}ln'\in N\dot{w}LN$ in $\dot{w}LN$ under the inverse map
$\Psi'$ is again $l$; this plays a rôle in \cite[Lemma 5.7]{MR3883243}.

\end{enumerate}
\end{rem}

\subsection{Charts on the quotient stack $[G/G]$}

Sevostyanov deduced from the cross section isomorphism (\ref{eq:cross-section})
that his slices transversely intersect the conjugacy classes of $G$
\cite[Proposition 2.3]{MR2806525}. In this subsection we adapt his
approach to prove that (ii) still holds in our more general setting,
and simultaneously refine it to show that (ii) $\Rightarrow$ (i'):
\begin{notation}
We denote by $\mathrm{Ad}_{m}(\cdot)$ the right adjoint action map
of an element $m$ of $G$ on its Lie algebra $\mathfrak{g}$. We
let $\mathfrak{n}_{w},\mathfrak{n}^{w},\mathfrak{l}$, $\mathfrak{\overline{n}}$,
$\overline{\mathfrak{n}}^{w}$, $\mathfrak{\overline{n}}_{w^{-1}}$
denote the free submodules of $\mathfrak{g}$ corresponding to roots
in $\mathfrak{R}_{w},\mathfrak{N}\backslash\mathfrak{R}_{w},\mathfrak{L}$,
$-\mathfrak{N}$, $-(\mathfrak{N}\backslash\mathfrak{R}_{w})$ and
$w(\mathfrak{R}_{w})=-\mathfrak{R}_{w^{-1}}$ respectively. (These
are actually all Lie subalgebras, as convexity of $\mathfrak{R}_{+}\backslash\mathfrak{R}_{w}$
and $\mathfrak{N}$ implies that $\mathfrak{N}\backslash\mathfrak{R}_{w}$
is convex.) We let $\mathfrak{t}_{w}'$ denote the orthogonal complement
inside $\mathfrak{t}$ to $\mathfrak{l}\cap\mathfrak{t}$; since $L$
contains $T^{w}$ we have $\mathfrak{t}_{w}'\subseteq\mathfrak{t}_{w}=(\mathfrak{t}^{w})^{\bot}$,
which is $w$-invariant as both $T^{w}$ and $\mathfrak{L}$ are.
Finally, we denote by $\overline{N}$ the unipotent subgroup of $G$
corresponding to $\overline{\mathfrak{n}}$ (and $-\mathfrak{N}$).
\end{notation}

\begin{lem}
The image of the differential of the conjugation map
\begin{equation}
G\times\dot{w}LN_{w}\longrightarrow G,\qquad(g,m)\longmapsto g^{-1}mg\label{eq:full-conjugation}
\end{equation}
at any point $(\mathrm{id},m)$ is given in the left trivalisation
of the tangent bundle of $G$ by
\[
(\mathrm{id}-\mathrm{Ad}_{m})(\mathfrak{n}\oplus\overline{\mathfrak{n}})+\mathfrak{t}_{w}'\oplus\mathfrak{l}\oplus\mathfrak{n}_{w}\oplus\overline{\mathfrak{n}}_{w^{-1}}.
\]

Moreover, we have 
\[
(\mathrm{id}-\mathrm{Ad}_{m})(\overline{\mathfrak{n}})\subseteq\overline{\mathfrak{n}}\oplus\mathfrak{l}\oplus\mathfrak{n}_{w}.
\]
\end{lem}

\begin{proof}
The left trivialisation of the tangent bundle of $G$ induces for
all points $m$ in $\dot{w}LN_{w}$ identifications of their tangent
spaces $T_{m}(\dot{w}LN_{w})\simeq\mathfrak{l}\oplus\mathfrak{n}_{w}$
with a free submodule of $\mathfrak{g}$, and this differential is
then given at a point $(\mathrm{id},m)$ by the linear map
\begin{equation}
\mathfrak{g}\oplus(\mathfrak{l}\oplus\mathfrak{n}_{w})\longrightarrow\mathfrak{g},\qquad(x,l)\longmapsto(\mathrm{id}-\mathrm{Ad}_{m})(x)+l.\label{eq:differential}
\end{equation}

For any $t\in\mathfrak{t}$ and $m'$ in $L$ we have $\mathrm{Ad}_{m'}(t)\in t+\mathfrak{l}$.
As $\mathfrak{R}_{w}\cup\mathfrak{L}$ is convex and $\mathfrak{L}=-\mathfrak{L}$,
it then follows for $m''\in N_{w}$ that $\mathrm{Ad}_{m''}(t+\mathfrak{l})\in t+\mathfrak{l}\oplus\mathfrak{n}_{w}$,
so that finally for $m=\dot{w}m'm''$ we have
\[
\mathrm{Ad}_{m}(t)=\mathrm{Ad}_{m''}\mathrm{Ad}_{m'}\mathrm{Ad}_{w}(t)\in\mathrm{Ad}_{w}(t)+\mathfrak{l}\oplus\mathfrak{n}_{w}.
\]
By definition of $\mathfrak{t}_{w}$ the linear operator $\mathrm{id}-w$
restricts to an isomorphism, and as $\mathfrak{t}_{w}'$ is invariant
under $w$ it then further restricts to an isomorphism of $\mathfrak{t}_{w}'$.
Hence the previous equation now yields
\[
(\mathrm{id}-\mathrm{Ad}_{m})(\mathfrak{t}_{w}')+\mathfrak{l}\oplus\mathfrak{n}_{w}=(\mathrm{id}-\mathrm{Ad}_{w})(\mathfrak{t}_{w}')+\mathfrak{l}\oplus\mathfrak{n}_{w}=\mathfrak{t}_{w}'\oplus\mathfrak{l}\oplus\mathfrak{n}_{w}.
\]
Similarly, from $\mathrm{Ad}_{m}(\mathfrak{l})=\mathrm{Ad}_{m''}(\mathfrak{l})\subseteq\mathfrak{l}\oplus\mathfrak{n}_{w}$
it follows that
\[
(\mathrm{id}-\mathrm{Ad}_{m})(\mathfrak{l})+\mathfrak{l}\oplus\mathfrak{n}_{w}=\mathfrak{l}\oplus\mathfrak{n}_{w}.
\]
Finally, from $\mathrm{Ad}_{m}(\overline{\mathfrak{n}}_{w^{-1}})=\mathrm{Ad}_{m''}\mathrm{Ad}_{m'}(\mathfrak{n}_{w})\subseteq\mathfrak{l}\oplus\mathfrak{n}_{w}$
we deduce that 
\[
(\mathrm{id}-\mathrm{Ad}_{m})(\overline{\mathfrak{n}}_{w^{-1}})+\mathfrak{l}\oplus\mathfrak{n}_{w}=\mathfrak{l}\oplus\mathfrak{n}_{w}\oplus\overline{\mathfrak{n}}_{w^{-1}}.
\]
Since the pair $(\mathfrak{N},\mathfrak{L})$ is slicing there is
a decomposition $\mathfrak{g}\simeq\mathfrak{n}\oplus\overline{\mathfrak{n}}\oplus\mathfrak{t}_{w}'\oplus\mathfrak{l}$,
so the first claim now follows.

As $\mathfrak{L}\cup\mathfrak{R}_{w}$ and $\mathfrak{L}\cup\mathfrak{N}$
are convex, it follows from $\mathfrak{L}=-\mathfrak{L}$ that so
is $\mathfrak{L}\cup(\mathfrak{N}\backslash\mathfrak{R}_{w})$. But
then $\mathfrak{L}\cup-(\mathfrak{N}\backslash\mathfrak{R}_{w})$
is also convex, so that $\mathrm{Ad}_{m'}(\overline{\mathfrak{n}}^{w})\subseteq\overline{\mathfrak{n}}^{w}\oplus\mathfrak{l}$
for any $m'$ in $L$. As 
\[
-(\mathfrak{N}\backslash\mathfrak{R}_{w})\cap-\mathfrak{R}_{w}=(\mathfrak{N}\backslash\mathfrak{R}_{w})\cap\mathfrak{R}_{w}=\varnothing,
\]
 for any $m''$ in $N_{w}$ the operator $\mathrm{Ad}_{m''}$ acts
as the identity on $\overline{\mathfrak{n}}^{w}$. Then for $m=\dot{w}m'm''$
we have
\[
\mathrm{Ad}_{m}(\overline{\mathfrak{n}})=\mathrm{Ad}_{m''}\mathrm{Ad}_{m'}\mathrm{Ad}_{w}(\overline{\mathfrak{n}})=\mathrm{Ad}_{m''}\mathrm{Ad}_{m'}(\overline{\mathfrak{n}}^{w}\oplus\mathfrak{n}_{w})\subseteq\mathrm{Ad}_{m''}(\overline{\mathfrak{n}}^{w}\oplus\mathfrak{l}\oplus\mathfrak{n}_{w})=\overline{\mathfrak{n}}^{w}\oplus\mathfrak{l}\oplus\mathfrak{n}_{w},
\]
yielding the second claim.
\end{proof}
We now prove (ii):
\begin{proof}
Since $G$ and $\dot{w}LN_{w}$ are smooth, the claim is equivalent
to requiring that the image of the differential of (\ref{eq:full-conjugation})
is surjective at each point of $G\times\dot{w}LN_{w}$. By equivariance
for the $G$-action on the first component by left translation, it
suffices to prove this at each point of the form $(\mathrm{id},m)$.
When we restrict (\ref{eq:full-conjugation}) to $N$, it yields the
cross section morphism (\ref{eq:cross-section}) which is an isomorphism
by the previous subsection. In the left trivialisation we have 
\[
\mathfrak{n}\subseteq\mathrm{Ad}_{m}(\mathfrak{n}_{w^{-1}})\oplus\mathfrak{l}\oplus\mathfrak{n}\simeq T_{m}(N_{w^{-1}}\dot{w}LN)=T_{m}(N\dot{w}LN),
\]
 so by this isomorphism the image of the differential certainly contains
$\mathfrak{n}$.

Now consider the Chevalley anti-automorphism which switches positive
and negative root vectors of a Chevalley basis. Expressing a lift
of a simple reflection as a product of exponentials of such elements,
an $\mathrm{SL}_{2}$-calculation shows that its image under this
involution is again a lift of this simple reflection. Hence the involution
maps a lift of $w^{-1}$ to a lift of $w$, but then the image of
the slice for (a suitable lift of) $w^{-1}$ is 
\[
\overline{N}_{w^{-1}}\overline{L}\dot{w}=\overline{N}_{w^{-1}}L\dot{w}=\overline{N}_{w^{-1}}\dot{w}L=\dot{w}N_{w}L=\dot{w}LN_{w}.
\]
Since by assumption equation (\ref{eq:crossing}) holds and the pair
is slicing, Lemma \ref{lem:inverse} implies that equation (\ref{eq:crossing})
also holds with $w$ replaced by $w^{-1}$. Thus by the previous subsection
the cross section isomorphism (\ref{eq:cross-section}) also holds
for $w^{-1}$. Hence from the involution we now obtain an isomorphism

\[
\overline{N}\times\dot{w}LN_{w}\overset{\sim}{\longrightarrow}\overline{N}\dot{w}L\overline{N},
\]
so by the same reasoning as in the previous paragraph, the image of
the differential (in the left trivialisation) also contains $\overline{\mathfrak{n}}$.
Combining this with the first part of the previous lemma, the claim
again follows from the decomposition $\mathfrak{g}\simeq\mathfrak{n}\oplus\overline{\mathfrak{n}}\oplus\mathfrak{t}_{w}'\oplus\mathfrak{l}$.
\end{proof}
And finally, we prove (ii) $\Rightarrow$ (i'):
\begin{proof}
Concretely, (i') says that image of the differential of the conjugation
map
\[
N\times\dot{w}LN_{w}\longrightarrow N\dot{w}LN,\qquad(g,m)\longmapsto g^{-1}mg
\]
at any point $(\mathrm{id},m)$ is an isomorphism. In the left trivialisation
we have 
\[
T_{m}(N\dot{w}LN)=T_{m}(N\dot{w}LN_{w})\simeq\mathrm{Ad}_{m}(\mathfrak{n})\oplus\mathfrak{l}\oplus\mathfrak{n}_{w}
\]
 and this differential is given by
\begin{equation}
\mathfrak{n}\oplus(\mathfrak{l}\oplus\mathfrak{n}_{w})\longrightarrow\mathrm{Ad}_{m}(\mathfrak{n})\oplus\mathfrak{l}\oplus\mathfrak{n}_{w},\qquad(x,l)\longmapsto(\mathrm{id}-\mathrm{Ad}_{m})(x)+l.\label{eq:differential-cross}
\end{equation}
By assumption the differential of (\ref{eq:full-conjugation}) is
surjective. Since $(\overline{\mathfrak{n}}\oplus\mathfrak{n}_{w})\cap\mathfrak{n}^{w}=\varnothing$,
the decomposition $\mathfrak{g}=\mathfrak{n}^{w}\oplus\overline{\mathfrak{n}}\oplus\mathfrak{l}\oplus\mathfrak{n}_{w}\oplus\mathfrak{t}_{w}'$
and the statements of the previous lemma imply that $\mathfrak{n}^{w}\subseteq(\mathrm{id}-\mathrm{Ad}_{m})(\mathfrak{n})+\mathfrak{l}\oplus\mathfrak{n}_{w}$.
But then we have
\[
\mathfrak{n}\subseteq(\mathrm{id}-\mathrm{Ad}_{m})(\mathfrak{n})+\mathfrak{l}\oplus\mathfrak{n}_{w}\subseteq\mathrm{Ad}_{m}(\mathfrak{n})\oplus\mathfrak{l}\oplus\mathfrak{n}_{w}
\]
 which implies that the second inclusion is an equality. Hence the
differential (\ref{eq:differential-cross}) is surjective; as it is
a morphism of (sheaves of) (locally) free modules of finite rank,
it is thus an isomorphism.
\end{proof}

\section{Poisson reduction}

Having obtained that the action of $N$ on $N\dot{w}LN$ is free (whilst
working over $\mathbb{C}$, for his particular choice of firmly convex
Weyl group element $w$ with the slicing pair $(\mathfrak{R}_{+}\backslash\mathfrak{R}^{w},\mathfrak{R}^{w})$),
Sevostyanov proceeds to proving that the Semenov-Tian-Shansky bracket
on $G$ reduces to the slice $\dot{w}LN_{w}$ when the $r$-matrix
is changed from the standard one $r_{\mathrm{st}}$ to
\[
r=r_{\mathrm{st}}+\frac{1+w}{1-w}\mathrm{proj}_{\mathfrak{t}_{w}}=\mathrm{proj}_{\mathfrak{n}_{+}}+\frac{1+w}{1-w}\mathrm{proj}_{\mathfrak{t}_{w}}-\mathrm{proj}_{\mathfrak{n}_{-}},
\]
by employing a general Poisson reduction method for manifolds \cite[§2]{MR836071}.
This approach was based on earlier work he did with his advisor on
their loop analogues \cite[Theorem 2.5]{MR1620523}. We will continue
to work in the algebraic setting:
\begin{defn}
A \emph{Poisson scheme} is a scheme $X$ with a Poisson bracket on
its sheaf of functions. On its smooth locus this bracket corresponds
to a Poisson bivector field which we will denote by $\Pi$; there
it induces a musical morphism $\Pi^{\#}:T^{*}X\rightarrow TX$ from
the cotangent sheaf to the tangent sheaf. A function $f$ then defines
a \emph{Hamiltonian vector field} $\mathrm{Ham}_{f}=\Pi^{\#}(\mathrm{d}f)$
on this locus.
\end{defn}

We modify this reduction method in Proposition \ref{lem:reduction};
one obtains a statement that is very similar to a standard characterisation
of smooth Poisson subschemes (recalled in Proposition \ref{prop:characterise-coisotropic-poisson}),
which explains the focus on Hamiltonian vector fields in the final
proof. By analysing tangent spaces with some new root combinatorics
we can work with a larger class of factorisable $r$-matrices, and
settle which of them yield reducible Poisson brackets.

As in the previous sections, we are implicitly working over a base
scheme but will omit it from all notation.

\subsection{Coisotropic subgroups}

Motivated by work of physicists on integrable systems, Drinfeld initiated
the study of Poisson-Lie groups (and their quantisations); they translate
to the algebraic setting as 
\begin{defn}
[{\cite[§3]{MR688240}}] A group scheme $G$ equipped with a Poisson
bracket is called a \emph{Poisson algebraic group }if this bracket
is \emph{multiplicative}, i.e.\ if the multiplication map $G\times G\rightarrow G$
is a morphism of Poisson schemes.
\end{defn}

The identity element of $G$ is then a symplectic point, so that the
Poisson bracket
\[
\CMcal O_{G,\mathrm{id}}\otimes\CMcal O_{G,\mathrm{id}}\longrightarrow\CMcal O_{G,\mathrm{id}},\qquad f\otimes f'\longmapsto(\mathrm{d}f\otimes\mathrm{d}f')(\Pi)
\]
 induces on its tangent space the structure of its Lie bialgebra.

Semenov-Tian-Shansky used the formalism of Poisson algebraic groups
to study the ``hidden symmetry groups'' (dressing transformations)
of certain integrable systems, as these don't preserve Poisson structures;
instead, they are Poisson actions:
\begin{defn}
[{\cite[p.\,1238]{MR842417}}] A group action of a Poisson algebraic
group $G$ on a Poisson scheme $X$ is called \emph{Poisson} if the
action map $G\times X\rightarrow X$ is a morphism of Poisson schemes.
\end{defn}

Concretely, a point $x$ in $X$ then induces a map $x:G\rightarrow X$
via $g\mapsto gx$ and in terms of the Poisson brackets on $G$ and
$X$, the Poisson condition can then be rephrased as
\begin{equation}
\{f,f'\}(gx)=\{x^{*}f,x^{*}f'\}(g)+\{g^{*}f,g^{*}f'\}(x)\label{eq:Poisson}
\end{equation}
 for $f,f'\in\CMcal O_{X}$ and arbitrary points $g\in G$, $x\in X$.

As shown at the end of this subsection, in order to construct interesting
quotients out of Poisson actions one sometimes uses subgroups of $G$
that are not necessarily Poisson themselves:
\begin{defn}
Let $X$ be a Poisson scheme. A smooth closed subscheme $Z\hookrightarrow X$
is called \emph{coisotropic} (resp. \emph{Poisson}) if 
\[
\Pi_{z}\in T_{z}Z\wedge T_{z}X\qquad(\textrm{resp. }\Pi_{z}\in T_{z}Z\wedge T_{z}Z)
\]
 for all points $z$ lying in $Z$.
\end{defn}

\begin{notation}
Given a scheme $X$ we denote its sheaf of functions by $\CMcal O_{X}$.
The inclusion of a closed subscheme $\iota:Z\hookrightarrow X$ induces
a morphism $\iota^{-1}(\CMcal O_{X})\rightarrow\CMcal O_{Z}$, and
its ideal ideal sheaf is denoted by $\CMcal I_{Z}:=\mathrm{ker}\bigl(\iota^{\sharp}:\iota^{-1}(\CMcal O_{X})\rightarrow\CMcal O_{Z}\bigr)$.
Given a function $f$ in $\iota^{-1}(\CMcal O_{X})$ we denote its
image under $\iota^{\sharp}$ by $f|_{Z}$.
\end{notation}

The focus in the final proof of this section will be on Hamiltonian
vector fields; heuristically, this is due to a group action analogue
of the following
\begin{prop}
[{\cite[Lemma 1.1]{MR723816}}] \label{prop:characterise-coisotropic-poisson}
Let $X$ be a Poisson scheme and $\iota:Z\hookrightarrow X$ a smooth
closed subscheme. Then the following are equivalent:

\begin{enumerate}[\normalfont(i)]

\item $Z$ is Poisson.

\item $\mathrm{ker}\bigl(\iota^{-1}(\CMcal O_{X})\twoheadrightarrow\CMcal O_{Z}\bigr)$
is a subsheaf of Poisson ideals; in other words, the Poisson bracket
on $\CMcal O_{X}$ reduces to $\CMcal O_{Z}$.

\item For any function $f$ in $\iota^{-1}(\CMcal O_{X})$, its Hamiltonian
vector field lies in $T_{Z}\subseteq T_{X}|_{Z}$. 

\end{enumerate}
\end{prop}

An algebraic proof of (ii) $\Leftrightarrow$ (iii) can be recovered
as a special case of Lemma \ref{lem:reduction}. One can characterise
coisotropic smooth closed subschemes similarly \cite[Proposition 1.2.2]{MR959095}. 
\begin{prop}
\label{prop:coisotropic} Let $G$ be a Poisson algebraic group and
let $H$ be a closed algebraic subgroup.

\newcounter{temp}

\begin{enumerate}[\normalfont(i)]

\item If $H$ is coisotropic, then the annihilator $\mathfrak{h}^{0}\subseteq\mathfrak{g}^{*}$
of its Lie algebra $\mathfrak{h}\subseteq\mathfrak{g}$ is a Lie subalgebra
of $\mathfrak{g}^{*}$.

\item \cite[Proposition 2]{MR842417} If $H$ is Poisson, then this
annihilator $\mathfrak{h}^{0}$ is an ideal of $\mathfrak{g}^{*}$.

\end{enumerate}

If $H$ is connected, then the converse to (i) and (ii) holds as well.
\end{prop}

\begin{proof}
(i) is well-known, it follows from: let $V$ be a locally free module
over a ring, $\Pi$ an element of $V\wedge V$ and $U$ a locally
free submodule of $V$; denote the annihilator of $U$ in the dual
$V^{*}$ by $U^{0}$. Then $\Pi^{\#}(U^{0})\subseteq U$ if and only
if $\Pi\in U\wedge V$.

(ii): Similarly, this follows from $\Pi^{\#}(V^{*})\subseteq U$ if
and only if $\Pi\in U\wedge U$.
\end{proof}
\begin{notation}
If a group scheme $H$ acts on a scheme $X$, then we denote the resulting
sheaf of $H$-invariant functions on $X$ by $\CMcal O_{X}^{H}$.
\end{notation}

\begin{prop}
[{\cite[Theorem 6]{MR842417}}]\label{prop:coisotropic-subgroup-preserves}
Let $G$ be a Poisson algebraic group with a Poisson action on a Poisson
scheme $X$, and let $H$ be a closed coisotropic subgroup of $G$.
Furthermore assume that the restriction of the action on $X$ to $H$
preserves a closed subscheme $\iota:Z\hookrightarrow X$. Then $(\iota^{\sharp})^{-1}(\CMcal O_{Z}^{H})$
is a sheaf of Poisson subalgebras of $\iota^{-1}(\CMcal O_{X})$.
\end{prop}

\begin{proof}
Let $f,f'\in(\iota^{\sharp})^{-1}(\CMcal O_{Z}^{H})$, $h\in H$ and
let $z\in Z$. Since $H$ is coisotropic we have 
\[
\{z^{*}f,z^{*}g\}(h)=(z_{*}\Pi_{h})(f,g)=0,
\]
so as the action is Poisson it then follows from (\ref{eq:Poisson})
that
\[
\{f,f'\}(hz)=\{z^{*}f,z^{*}f'\}(h)+\{h^{*}f,h^{*}f'\}(z)=\{h^{*}f,h^{*}f'\}(z)=\{f,f'\}(z).\qedhere
\]
\end{proof}

\subsection{Factorisable $r$-matrices}

In order to obtain explicit coisotropic subgroups for reductive group
schemes, we will now specialise to a particular class of Poisson structures.
\begin{notation}
Given an element $x=\sum_{i=1}^{n}a_{i}\otimes b_{i}$ in $\mathfrak{g}\otimes\mathfrak{g}$,
we will write $x_{12}:=\sum_{i=1}^{n}a_{i}\otimes b_{i}\otimes1$
and $x_{13}:=\sum_{i=1}^{n}a_{i}\otimes1\otimes b_{i}$ and similarly
$x_{23}$ for the usual elements in $\otimes^{3}\mathfrak{g}$. Furthermore,
we denote its ``flip'' by $x^{21}:=\sum_{i=1}^{n}b_{i}\otimes a_{i}\in\mathfrak{g}\otimes\mathfrak{g}$.
\end{notation}

The existence of a multiplicative Poisson structure can be rephrased
cohomologically: using the left or right trivialisation, they are
1-cocycles on $G$ with values in $\mathfrak{g}\wedge\mathfrak{g}$.
Subsequently employing e.g.\ Whitehead's first lemma, there often
exists an element $r\in\mathfrak{g}\wedge\mathfrak{g}$ such that
the Lie cobracket $\delta$ of a Lie bialgebra $\mathfrak{g}$ equals
the differential $\partial r$, i.e.\ such that
\[
\delta(x)=\partial r(x):=\frac{1}{2}[x\otimes1+1\otimes x,r]
\]
for all $x\in\mathfrak{g}$. Conversely, in order for an arbitrary
element $r\in\mathfrak{g}\wedge\mathfrak{g}$ to define a compatible
cobracket it is necessary and sufficient \cite[§6]{MR688240} that
this \emph{$r$-matrix} satisfies the \emph{generalised Yang-Baxter
equation}
\[
[r,r]\in(\wedge^{3}\mathfrak{g})^{\mathfrak{g}},
\]
where for any element $r\in\mathfrak{g}\otimes\mathfrak{g}$ we denote
by
\[
[r,r]:=[r_{12},r_{13}]+[r_{12},r_{23}]+[r_{13},r_{23}]\in\otimes^{3}\mathfrak{g}
\]
 its \emph{Drinfeld bracket}; up to scalar, this coincides with the
canonical Gerstenhaber (or Schouten-Nijenhuis) bracket when restricted
to $\wedge^{2}\mathfrak{g}\subset\wedge^{\bullet}\mathfrak{g}$. Note
that for any element $c$ in $(\mathfrak{g}\otimes\mathfrak{g})^{\mathfrak{g}}$,
we have $[c,c]=[c_{12},c_{23}]$.
\begin{thm}
Let $\mathfrak{g}$ be a reductive Lie algebra over a field $\Bbbk$.

\begin{enumerate}[\normalfont(i)]

\item \cite[Theorem 19.1]{MR24908} There exists a natural quasi-isomorphism
between the exterior algebra $(\wedge^{\bullet}\mathfrak{g})^{\mathfrak{g}}$
and the Chevalley-Eilenberg complex computing Lie algebra homology,
and hence
\[
(\wedge^{\bullet}\mathfrak{g})^{\mathfrak{g}}\simeq\mathrm{H}_{\bullet}(\mathfrak{g}).
\]

\item \cite[§11]{MR36511} If the characteristic is 0 (or sufficiently
large) then the Koszul map
\[
\mathrm{Sym}^{2}(\mathfrak{g})^{\mathfrak{g}}\longrightarrow(\wedge^{3}\mathfrak{g})^{\mathfrak{g}},\qquad c\longmapsto[c_{12},c_{13}]
\]
 is an isomorphism.

\end{enumerate}
\end{thm}

Throughout the rest of this section, we will implicitly make use of
the identification 
\[
\mathfrak{g}\wedge\mathfrak{g}\subset\mathfrak{g}\otimes\mathfrak{g}\simeq\mathrm{Hom}(\mathfrak{g}^{*},\mathfrak{g}),\qquad x\otimes y\longmapsto\bigl(\xi\mapsto\xi(x)y\bigr).
\]

\begin{prop}
[{\cite{MR674005}}] \label{prop:r-matrix-map} Let $\mathfrak{g}$
be a Lie algebra over a ring with 2 invertible, let $r\in\mathfrak{g}\wedge\mathfrak{g}$
and let $c\in\mathrm{Sym}^{2}(\mathfrak{g})^{\mathfrak{g}}$ and consider
the two maps
\[
r_{\pm}:\mathfrak{g}^{*}\longrightarrow\mathfrak{g},\qquad\xi\longmapsto\frac{r\pm c}{2}\xi.
\]
Then the following are equivalent:

\begin{enumerate}[\normalfont(i)]

\item $[r,r]=-[c,c]$.

\item The map $r_{\pm}$ is a Lie algebra homomorphism.

\item The map $r_{\pm}$ is a Lie coalgebra antihomomorphism.

\end{enumerate}
\end{prop}

The equation $[r,r]=-[c,c]$ is called the \emph{modified classical
Yang-Baxter equation}.
\begin{defn}
If $G$ is a Poisson algebraic group whose Lie bialgebra $\mathfrak{g}$
admits an $r$-matrix $r$, then we will abbreviate this by writing
$G_{r}$ and $\mathfrak{g}_{r}$. If $[r,r]=-[c,c]$ and $c$ defines
a perfect pairing, then we say that $r$ is \emph{factorisable}, and
we will say that the Lie bialgebra $\mathfrak{g}_{r}$ and any corresponding
Poisson algebraic group $G_{r}$ are \emph{factorisable}.
\end{defn}

\begin{notation}
We denote the torus component of such $c$ by $c_{\mathfrak{t}}$.
Let $L_{g}$ and $R_{g}$ denote the translations on $G$ of left
and right multiplication by $g$. Given an element $r\in\mathfrak{g}\otimes\mathfrak{g}$,
we then write $r^{R}:=r^{R,R}:=(\mathrm{d}R_{g}\otimes\mathrm{d}R_{g})(r)$,
$r^{R,L}:=(\mathrm{d}R_{g}\otimes\mathrm{d}L_{g})(r)$, etc. Then
set $r^{\mathrm{ad}}:=r^{R,R}+r^{L,L}-r^{R,L}-r^{L,R}$.
\end{notation}

Semenov-Tian-Shansky used Proposition \ref{prop:r-matrix-map} to
geometrically prove over $\mathbb{C}$ the following
\begin{thm}
[{\cite[p.\,1247]{MR842417}}] Given a factorisable Poisson algebraic
group $G_{r}$ over a scheme where 2 is invertible, let $c$ in $\mathrm{Sym}^{2}(\mathfrak{g})^{\mathfrak{g}}$
be such that $[r,r]=-[c,c]$, let $G_{*}$ denote the underlying scheme
of $G$ but now equipped with the bivector
\begin{equation}
\frac{1}{2}(r^{\mathrm{ad}}-c^{L,R}+c^{R,L})=(r_{+}^{R,R}-r_{+}^{L,R})-(r_{-}^{R,L}-r_{-}^{L,L}).\label{eq:STS-bracket}
\end{equation}
This yields a Poisson structure on $G$, and the right conjugation
map $G_{r}\times G_{*}\rightarrow G_{*}$ is Poisson.
\end{thm}

This can also be proven directly, without using factorisability.
\begin{defn}
This is called the \emph{(right) Semenov-Tian-Shansky bracket} on
$G$.
\end{defn}

\begin{notation}
In order to obtain more such brackets in low characteristic, we slightly
enlarge $\mathfrak{t}$ to by adding the dual basis $\{\check{\omega}_{i}\}_{i=1}^{\mathrm{rk}}$
to the simple roots and denote the result by $\mathfrak{t}_{\mathrm{sc}}$.
We then assume that $c_{\mathfrak{t}}$ lies in $\mathfrak{t}_{\mathrm{sc}}\otimes\mathfrak{t}$:
\end{notation}

\begin{prop}
\label{prop:integral-formulas} If $\mathfrak{g}$ is defined over
an arbitrary ring then such $c_{\mathfrak{t}}$ and $r_{\pm}$ might
not lie in $\mathfrak{g}$, but the corresponding Semenov-Tian-Shansky
bracket still yields integral formulas.
\end{prop}

\begin{proof}
We only prove the second part. In the left-trivialisation of the tangent
bundle of $G$, the right-hand-side of equation (\ref{eq:STS-bracket})
is given at any point $g$ in $G$ by
\[
\bigl((\mathrm{Ad}_{g}-\mathrm{id})\otimes\mathrm{Ad}_{g}\bigl)(r_{+})-\bigl((\mathrm{Ad}_{g}-\mathrm{id})\otimes\mathrm{id}\bigl)(r_{-}),
\]
whose torus component is
\[
\bigl((\mathrm{Ad}_{g}-\mathrm{id})\otimes(\mathrm{Ad}_{g}+\mathrm{id})\bigl)(c_{\mathfrak{t}}/2).
\]

We may decompose $g$ into a product of root subgroups, and induct
on the length of such an expression. By projecting $x$ to root spaces,
it suffices to prove the claim for elements of the form $(\check{\omega}_{i}\otimes t_{i}/2)$.
For $g$ an element of a root subgroups, it follows from
\[
\mathrm{Ad}_{\exp_{\beta}(y)}(x)=x+\beta(x)y,
\]
 where $\exp_{\beta}:\mathfrak{g}_{\beta}\rightarrow N_{\beta}$ is
the exponential map and $y$ lies in $\mathfrak{g}_{\beta}$. The
claim now follows by induction on the length of the decomposition
of $g$: if $g=g'g''$ with $g''$ lying in a root subgroup, then
\[
\mathrm{Ad}_{g}(x)\pm x=\mathrm{Ad}_{g''}\bigl(\mathrm{Ad}_{g'}(x)\pm x\bigr)\mp\bigl(\mathrm{Ad}_{g''}(x)-x\bigr)\in\mathrm{Ad}_{g''}(\mathfrak{g})+\mathfrak{g}=\mathfrak{g}.\qedhere
\]
\end{proof}
\begin{example}
Let $\mathfrak{g}$ be a reductive Lie algebra over a field. Extend
Chevalley generators to a basis $\{e_{\beta},f_{\beta}\}_{\beta\in\mathfrak{R}_{+}}$
and $\{t_{i}\}_{i=1}^{\mathrm{rk}}$ and then denote by $\{\check{\omega}_{i}\}_{i=1}^{\mathrm{rk}}$
the dual basis in $\mathfrak{t}$ to the simple roots. If we set
\[
r:=\sum_{\beta\in\mathfrak{R}_{+}}d_{\beta}f_{\beta}\wedge e_{\beta},\qquad c:=\sum_{\beta\in\mathfrak{R}_{+}}d_{\beta}e_{\beta}\otimes f_{\beta}+\sum_{\beta\in\mathfrak{R}_{+}}d_{\beta}f_{\beta}\otimes e_{\beta}+\sum_{i=1}^{\mathrm{rk}}\check{\omega}_{i}\otimes t_{i},
\]
where $d_{\beta}\in\{1,2,3\}$ are the usual symmetrisers with $d_{\beta}=1$
for long roots, then this Casimir element yields $[r,r]=-[c,c]$.
\end{example}

\begin{notation}
Given a factorisable $r$-matrix $r$ and $c$ in $\mathrm{Sym}^{2}(\mathfrak{g})^{\mathfrak{g}}$
such that $[r,r]=-[c,c]$, we denote the inner product corresponding
to $c$ by $(\cdot,\cdot)_{r}$, so that the orthogonal complement
of a submodule $\mathfrak{h}$ of $\mathfrak{g}$ is given by $\mathfrak{h}^{\bot}:=c(\mathfrak{h}^{0})$.
We furthermore write $\mathfrak{b}_{\pm}^{r}:=r_{\pm}(\mathfrak{g}^{*})$
and $\mathfrak{m}_{\pm}^{r}:=r_{\pm}(\ker\,r_{\mp})$.
\end{notation}

\begin{prop}
\label{prop:annihilator} Let $\mathfrak{g}_{r}$ be a factorisable
Lie bialgebra.

\begin{enumerate}[\normalfont(i)]

\item \cite{MR674005} The annihilator of $\mathfrak{b}_{\pm}^{r}$
in $\mathfrak{g}^{*}$ is $\mathrm{ker}\,r_{\mp}$.

\item Let $\mathfrak{h}$ be a subspace of $\mathfrak{g}$ containing
$\mathfrak{b}_{\pm}^{r}$. Then $\mathfrak{h}^{\bot}=r_{\pm}(\mathfrak{h}^{0})\subseteq\mathfrak{m}_{\pm}$,
with equality if and only if $\mathfrak{h}^{0}=(\mathfrak{b}_{\pm}^{r})^{0}$.

In particular, we have
\[
\iota_{r}(\mathfrak{h}^{0})=(\mathfrak{h}^{\bot},0)\qquad\bigl(\textrm{resp. }\iota_{r}(\mathfrak{h}^{0})=(0,\mathfrak{h}^{\bot})\bigr).
\]

\end{enumerate}
\end{prop}

\begin{proof}
(i): Any element of $\mathfrak{b}_{\pm}^{r}$ is of the form $r_{\pm}x$
for some $x\in\mathfrak{g}^{*}$. For $y\in\mathfrak{g}^{*}$ we then
have
\[
\<y,r_{\pm}x\>=\<r_{\pm}^{21}y,x\>=\<-r_{\mp}y,x\>.
\]
Since $x$ can be chosen arbitrarily, the claim follows.

(ii): From (i) it follows that $\mathfrak{h}^{0}\subseteq(\mathfrak{b}_{\pm}^{r})^{0}=\mathrm{ker}\,r_{\mp}$,
so
\[
\mathfrak{h}^{\bot}=c(\mathfrak{h}^{0})\subseteq(r_{+}-r_{-})(\mathrm{ker}\,r_{\mp})=\pm r_{\pm}(\mathrm{ker}\,r_{\mp})=:\mathfrak{m}_{\pm}.
\]
The claim on equality follows from nondegeneracy of $c$. For $\mathfrak{h}$
containing $\mathfrak{b}_{+}^{r}$ we now find
\[
\iota_{r}(\mathfrak{h}^{0})=(r_{+},r_{-})(\mathfrak{h}^{0})=\bigl(r_{+}(\mathfrak{h}^{0}),0\bigr)=(\mathfrak{h}^{\bot},0).\qedhere
\]
\end{proof}
\begin{defn}
A \emph{Belavin-Drinfeld triple} is a triple $\mathfrak{T}=(\mathfrak{T}_{0},\mathfrak{T}_{1},\tau)$
where $\mathfrak{T}_{0},\mathfrak{T}_{1}$ are sets of simple roots
and $\tau:\mathfrak{T}_{0}\rightarrow\mathfrak{T}_{1}$ is a bijection
such that

\begin{itemize}

\item $\bigl(\tau(\alpha),\tau(\tilde{\alpha})\bigr)=(\alpha,\tilde{\alpha})$
for any pair of simple roots $\alpha,\tilde{\alpha}\in\mathfrak{T}_{0}$,
and

\item $\tau$ is \emph{nilpotent}: for any $\alpha\in\mathfrak{T}_{0}$
there exists $m\in\mathbb{N}_{1}$ such that $\tau^{m}(\alpha)\in\mathfrak{T}_{1}\backslash\mathfrak{T}_{0}$.

\end{itemize}

Such a triple gives the set of positive roots a partial ordering:
for positive roots $\beta,\beta'$ we set $\beta<\beta'$ if $\beta\in\mathbb{N}_{0}\mathfrak{T}_{0}$,
$\beta'\in\mathbb{N}_{0}\mathfrak{T}_{1}$ and $\tau^{m}(\beta)=\beta'$
for some $m\in\mathbb{N}_{1}$.
\end{defn}

\begin{thm}
[{\cite{MR674005}}] Let $\mathfrak{g}_{r}$ be a factorisable reductive
Lie bialgebra over a field of characteristic 0 (or sufficiently large).
Then there exists a Cartan decomposition and a Belavin-Drinfeld triple
$\mathfrak{T}=(\mathfrak{T}_{0},\mathfrak{T}_{1},\tau)$ such that
\begin{align}
r_{+} & =\frac{1}{2}(r_{0}+c_{\mathfrak{t}})+\sum_{\beta\in\mathfrak{R}_{+}}d_{\beta}f_{\beta}\otimes e_{\beta}+\sum_{\beta<\beta'}d_{\beta}f_{\beta}\wedge e_{\beta'},\label{eq:factorisable-r-matrix-positive}
\end{align}
where $f_{\beta}$ and $e_{\beta}$ are root vectors with weight $-\beta$
and $\beta$ respectively, are normalised by $d_{\beta}(f_{\beta},e_{\beta})=1$
for an invariant bilinear form corresponding to some element $c\in\mathrm{Sym}^{2}(\mathfrak{g})^{\mathfrak{g}}$,
whilst the element $r_{0}\in\mathfrak{t}\wedge\mathfrak{t}$ satisfies
\begin{equation}
\bigl(\tau(\alpha)\otimes1\bigr)(r_{0}+c_{\mathfrak{t}})+(1\otimes\alpha)(r_{0}+c_{\mathfrak{t}})=0,\qquad\forall\alpha\in\mathfrak{T}_{0},\label{eq:factorisable-torus-part}
\end{equation}
 where $c_{\mathfrak{t}}$ is the image of the $\mathfrak{t}$-component
of $c$.

Furthermore, such solutions $r_{0}$ form a torsor for the $k_{\mathfrak{T}}(k_{\mathfrak{T}}-1)/2$-dimensional
vector space $\mathfrak{t}_{\mathfrak{T}}\wedge\mathfrak{t}_{\mathfrak{T}}$,
where $k_{\mathfrak{T}}:=\mathrm{rank}(\mathfrak{g})-|\mathfrak{T}_{0}|$
and
\begin{equation}
\mathfrak{t}_{\mathfrak{T}}:=\{t\in\mathfrak{t}:\beta(t)=\beta'(t)\text{ for all }\beta<\beta'\}.\label{eq:factorisable-torus}
\end{equation}
\end{thm}

Explicitly, we then have 
\begin{align}
r_{-} & =\frac{1}{2}(r_{0}-c_{\mathfrak{t}})-\sum_{\beta\in\mathfrak{R}_{+}}d_{\beta}e_{\beta}\otimes f_{\beta}+\sum_{\beta<\beta'}d_{\beta}f_{\beta}\wedge e_{\beta'}.\label{eq:factorisable-r-matrix-negative}
\end{align}

\begin{prop}
\label{prop:factorisable-coisotropic-poisson} Let $G_{r}$ be a reductive
factorisable Poisson algebraic group and let $H\subseteq G_{r}$ be
a connected closed subgroup with Lie algebra $\mathfrak{h}$.

\begin{enumerate}[\normalfont(i)]

\item  If $\mathfrak{h}$ contains $\mathfrak{b}_{+}^{r}$ or $\mathfrak{b}_{-}^{r}$
then $H$ is Poisson.

\item Now suppose that the $r$-matrix comes from a Belavin-Drinfeld
triple $\mathfrak{T}$, and that $H$ is a product of root subgroups
corresponding to a subset of positive (resp.\ negative) roots $\mathfrak{N}$
of the form $\mathfrak{R}_{w}$. If $\mathfrak{T}\subseteq\mathfrak{N}^{c}$
then $H$ is coisotropic.

\end{enumerate}
\end{prop}

\begin{proof}
 (i): Let's assume that $\mathfrak{h}$ contains $\mathfrak{b}_{+}^{r}$,
then by Proposition \ref{prop:annihilator}(ii) we have $\iota_{r}(\mathfrak{h}^{0})=(\mathfrak{h}^{\bot},0)$.
Thus, in order to prove that $\mathfrak{h}^{0}$ is an ideal of $\mathfrak{g}^{*}$,
it suffices to show that $[\mathfrak{b}_{+}^{r},\mathfrak{h}^{\bot}]\subseteq\mathfrak{h}^{\bot}$.
Given $x\in\mathfrak{b}_{+}^{r}\subseteq\mathfrak{h}$ and $y\in\mathfrak{h}^{\bot}$
and $z\in\mathfrak{h}$, invariance of $(\cdot,\cdot)_{r}$ yields
\[
([x,y],z)_{r}=(y,[z,x])_{r}=0
\]
 so that $[x,y]\in\mathfrak{h}^{\bot}$.

(ii): Let $\{e_{\beta}^{*},f_{\beta}^{*}\}_{\beta\in\mathfrak{R}_{+}}$
be the basis of $\mathfrak{n}_{+}^{*}\oplus\mathfrak{n}_{-}^{*}\subset\mathfrak{g}^{*}$
dual to the usual basis $\{e_{\beta},f_{\beta}\}_{\beta\in\mathfrak{R}_{+}}$
of $\mathfrak{n}_{+}\oplus\mathfrak{n}_{-}$, so $e_{\beta}^{*}$
is the element of $\mathfrak{g}^{*}$ vanishing on $\mathfrak{t}$
and all root spaces other than $\mathfrak{n}_{\beta}$, where it is
given on its basis vector by $e_{\beta}^{*}(e_{\beta})=1$, etc. Then
\begin{equation}
\iota_{r}(e_{\beta}^{*})=-d_{\beta}(\sum_{i>0}f_{\tau^{-i}(\beta)},\sum_{i\geq0}f_{\tau^{-i}(\beta)}),\qquad\iota_{r}(f_{\beta}^{*})=d_{\beta}(\sum_{i\geq0}e_{\tau^{i}(\beta)},\sum_{i>0}e_{\tau^{i}(\beta)})\label{eq:dual-embedding-explicit}
\end{equation}
 where we set $e_{\tau^{i+1}(\beta)}=0$ (resp.\ $f_{\tau^{-i-1}(\beta)}=0$)
when $\tau^{i}(\beta)$ is not in the span of $\mathfrak{T}_{0}$
(resp.\ $\tau^{-i}(\beta)$ not in the span of $\mathfrak{T}_{1}$).
In particular, if $\beta$ is not in $\mathfrak{N}^{c}$ then it is
not in $\mathfrak{T}_{0}$ or $\mathfrak{T}_{1}$ and we find
\begin{equation}
\iota_{r}(e_{\beta}^{*})=-d_{\beta}(0,f_{\beta}),\qquad\iota_{r}(f_{\beta}^{*})=d_{\beta}(e_{\beta},0).\label{eq:dual-embedding-trivial}
\end{equation}

Since by assumption $\mathfrak{h}$ is a sum of root subspaces, its
dual is spanned by a subset of $\{e_{\beta}^{*},f_{\beta}^{*}\}_{\beta\in\mathfrak{R}_{+}}$
plus a basis for $\mathfrak{t}^{*}$. The adjoint action of elements
of $\iota_{r}(\mathfrak{t}^{*})$ on elements in (\ref{eq:dual-embedding-explicit})
is through scalars due to (\ref{eq:factorisable-torus}), so we can
safely ignore them. 

Let's do the positive case: since $\mathfrak{T}\subseteq\mathfrak{N}^{c}$
and the set of roots $\mathfrak{N}^{c}$ is convex \cite{MR1169886},
it follows from (\ref{eq:dual-embedding-explicit}) that bracket of
$\iota_{r}(e_{\beta}^{*})$ and $\iota_{r}(e_{\beta'}^{*})$ for $\beta$
in $\mathfrak{N}^{c}$ corresponds to root vectors lying in (\ref{eq:dual-embedding-explicit}).
As $\mathfrak{g}$ is closed under bracketing, (\ref{eq:dual-embedding-trivial})
implies that it must be a linear combination of elements $\iota_{r}(e_{\beta''}^{*})$
with $\beta''$ in $\mathfrak{N}^{c}$: there are simply no other
elements in $\iota_{r}(\mathfrak{g}^{*})$ projecting to the right
weight spaces. 

For $\iota_{r}(e_{\beta}^{*})$ and $\iota_{r}(f_{\beta'}^{*})$ with
$\beta\in\mathfrak{N}^{c}$ and $\beta'\in\mathfrak{R}_{+}$, note
that only the parts lying in $\mathfrak{T}$ can bracket nontrivially.
As $\mathfrak{T}$ forms a standard parabolic subsystem it similarly
follows that the bracket is a linear combination of elements $\iota_{r}(e_{\beta''}^{*})$
with $\beta''$ in $\mathfrak{N}^{c}$ and $\iota_{r}(f_{\beta'''}^{*})$
with $\beta'''$ arbitrary.
\end{proof}
Restricting the left or right multiplication of $G$ on itself to
$H$ as in (i), we thus obtain Poisson structures on the GIT quotients
$H\backslash\!\backslash G$ and $G/\!/H$. In the particular case
where $r$ is the standard $r$-matrix (so that $N_{\pm}^{r}=N_{\pm}$
and $B_{\pm}^{r}=B_{\pm}$) and $H$ is a parabolic subgroup, the
corresponding Poisson structure has been extensively studied (e.g.\
\cite{MR2529913}).

\subsection{\label{subsec:reducing-sts}Reducing the Semenov-Tian-Shansky bracket}
\begin{lem}
\label{lem:reduction} Let $G$ be a Poisson group scheme with an
action on a Poisson scheme $X$. Let $H\subseteq G$ be a closed subgroup
preserving a smooth closed subscheme $\iota:Z\hookrightarrow X$ such
that $(\iota^{\sharp})^{-1}(\CMcal O_{Z}^{H})$ is a sheaf of Poisson
subalgebras of $\iota^{-1}(\CMcal O_{X})$. Then the following are
equivalent:

\begin{enumerate}[\normalfont(i)]

\item The ideal sheaf $\CMcal I_{Z}$ is a subsheaf of Poisson ideals
of $(\iota^{\sharp})^{-1}(\CMcal O_{Z}^{H})$; in other words, the
Poisson bracket on $(\iota^{\sharp})^{-1}(\CMcal O_{Z}^{H})$ reduces
to $\CMcal O_{Z}^{H}$.

\item For any function $f$ in $(\iota^{\sharp})^{-1}(\CMcal O_{Z}^{H})$,
its Hamiltonian vector field lies in $T_{Z}$. 

\end{enumerate}
\end{lem}

\begin{proof}
(i) $\Rightarrow$ (ii): Pick a function $f$ in $(\iota^{\sharp})^{-1}(\CMcal O_{Z}^{H})$
and a covector $\alpha$ at $z$ which is annihilated by all tangent
vectors in $T_{z}Z$. Lifting $\alpha$ to a one-form in a neighbourhood
of $z$, by the conormal exact sequence we can find a function $f'$
in $\CMcal I_{Z}$ such that $\mathrm{d}f'|_{z}=\alpha$ and $f'|_{Z}=0$.
Then the hypothesis yields that
\[
\alpha\bigl(\mathrm{Ham}_{f}\bigr)=\{f,f'\}|_{z}=\{f|_{Z},0\}=0,
\]
 which means that $\mathrm{Ham}_{f}$ is annihilated by the covectors
annihilated by $T_{z}Z$, which by smoothness implies that it lies
in the tangent space $T_{z}Z$.

(ii) $\Rightarrow$ (i): Let $f$ be a function in $(\iota^{\sharp})^{-1}(\CMcal O_{Z}^{H})$
and $f'$ in $\CMcal I_{Z}$, then from the inclusion $\mathrm{Ham}_{f}\in T_{Z}$
it follows that
\[
\{f,f'\}|_{Z}=\mathrm{d}f'(\mathrm{Ham}_{f})=\mathrm{d}(f'|_{Z})(\mathrm{Ham}_{f})=0,
\]
implying that $\{f,f'\}$ lies in $\CMcal I_{Z}$.
\end{proof}
\begin{cor}
\label{cor:STS-coisotropic} Equip $G$ with the right Semenov-Tian-Shansky
bracket coming from a factorisable $r$-matrix with $\mathfrak{T}\subseteq\mathfrak{L}$.
Consider the right conjugation action of $N\subset G$ on $G_{*}$
and its restriction to the closed subscheme $\iota:N\dot{w}LN\hookrightarrow G_{*}$.
The subgroup $N$ is coisotropic and $(\iota^{\sharp})^{-1}(\CMcal O_{N\dot{w}LN}^{N})$
is a subsheaf of $\iota^{-1}(\CMcal O_{G_{*}})$ of Poisson subalgebras.
\end{cor}

\begin{proof}
This follows by combining Proposition \ref{prop:factorisable-coisotropic-poisson}
with Proposition \ref{prop:coisotropic-subgroup-preserves}.
\end{proof}
\begin{notation}
For any function $f$ in $\CMcal O_{G}$ and element $g$ in $G$,
we define $\mathrm{d}f^{g}\in\mathfrak{g}^{*}$ by evaluating on a
tangent vector in $\mathfrak{g}$ as follows: 
\[
\mathrm{d}f^{g}(\cdot):=\mathrm{d}f\bigl(R_{g}(\cdot)-L_{g}(\cdot)\bigr)\in\mathfrak{g}^{*},
\]
so in the left trivialisation we have $\mathrm{d}f^{g}(\cdot)=\mathrm{d}f\bigl(\mathrm{Ad}_{g}(\cdot)-\mathrm{id}(\cdot)\bigr)$.
The tangent space to $N\dot{w}LN=NL\dot{w}N$ at one of its elements
$g$ is given by 
\begin{equation}
T_{g}:=\mathrm{Ad}_{g}(\mathfrak{n})+\mathfrak{l}\oplus\mathfrak{n}=\mathrm{Ad}_{g}(\mathfrak{n}\oplus\mathfrak{l})+\mathfrak{n}\subseteq\mathfrak{g},\label{eq:tangent-space}
\end{equation}
where $\mathrm{Ad}_{g}(\cdot)$ still denotes the right adjoint action. 
\end{notation}

\begin{lem}
\label{lem:parab} Assume that $\mathfrak{L}$ is a standard parabolic
subroot system. Then
\[
T_{g}\cap\mathfrak{t}=\mathfrak{l}\cap\mathfrak{t}.
\]
\end{lem}

\begin{proof}
Let $g=nl\dot{w}\tilde{n}$. Since $\mathfrak{N}$ is convex, we have
$\mathrm{Ad}_{n}(\mathfrak{n})\subseteq\mathfrak{n}$. From the assumption
on $\mathfrak{L}$ it follows that
\[
\mathrm{Ad}_{l}(\mathfrak{n})\subseteq\bigoplus_{\beta\in\mathfrak{R}_{+}\backslash\mathfrak{L}}\mathfrak{g}_{\beta}.
\]
As $\mathfrak{L}$ is preserved by $w$, it now follows that
\[
\mathrm{Ad}_{nl\dot{w}}(\mathfrak{n})\subseteq\mathrm{Ad}_{\dot{w}}\bigl(\mathrm{Ad}_{l}(\mathfrak{n})\bigr)\subseteq\mathrm{Ad}_{\dot{w}}\bigl(\bigoplus_{\beta\in\mathfrak{R}_{+}\backslash\mathfrak{L}}\mathfrak{g}_{\beta}\bigr)\subseteq\bigoplus_{\beta\in\mathfrak{R}\backslash\mathfrak{L}}\mathfrak{g}_{\beta}.
\]
Now let $x$ be an arbitrary element in the right-hand-side. If $\beta$
is of minimal height among the roots in $\mathfrak{R}\backslash\mathfrak{L}$
such that the projection of $x$ to $\mathfrak{g}_{\beta}$ is nonzero,
then the projection of $\mathrm{Ad}_{\tilde{n}}(x)$ to $\mathfrak{g}_{\beta}$
is still nonzero. Thus if $\beta$ is negative, $\mathrm{Ad}_{\tilde{n}}(x)$
does not lie in $\mathfrak{t}+\mathfrak{l}\oplus\mathfrak{n}$; if
$\beta$ is positive then $\mathrm{Ad}_{\tilde{n}}(x)$ still lies
in $\mathfrak{n}_{+}$. Hence
\[
T_{g}\cap\mathfrak{t}=\bigl(\mathrm{Ad}_{g}(\mathfrak{n})+\mathfrak{l}\oplus\mathfrak{n}\bigr)\cap\mathfrak{t}=(\mathfrak{l}\oplus\mathfrak{n})\cap\mathfrak{t}=\mathfrak{l}\cap\mathfrak{t}.\qedhere
\]
\end{proof}
\begin{prop}
Fix an ordering $\beta_{1},\ldots,\beta_{l}$ of the roots of $\mathfrak{R}_{w}$
by height, and fix the element $g':=p_{\beta_{1}}(1)\cdots p_{\beta_{l}}(1)$.
Then
\[
\mathfrak{t}_{\mathrm{sc}}\longrightarrow\mathfrak{n}_{w},\qquad t\longmapsto\mathrm{proj}_{\mathfrak{n}_{w}}\bigl(\mathrm{Ad}_{g}(t)-t\bigr)
\]
 is injective.
\end{prop}

\begin{proof}
By construction of $\mathfrak{t}_{\mathrm{sc}}$, we can recover the
coordinates of $t$ in the standard basis of $\mathfrak{t}_{\mathrm{sc}}$
through studying $t\mapsto\alpha(t)$ as $\alpha$ ranges over the
simple root. By \cite[Corollary 3.13]{WM-mindom}, these values can
be obtained from the roots in $\mathfrak{R}_{w}$. For $\beta$ in
$\mathfrak{R}_{w}$, let $e_{\beta}$ denote the element of $\mathfrak{g}_{\beta}$
exponentiating to $p_{\beta}(1)$. We write $O(\beta_{i+1})$ to mean
a polynomial in the $e_{\beta_{j}}$ with $j\geq i+1$. It follows
from
\begin{align*}
\mathrm{Ad}_{g'}(t)-t & =\mathrm{Ad}_{p_{\beta_{2}}(1)\cdots p_{\beta_{l}}(1)}\bigl(t+\beta_{1}(t)e_{\beta_{1}}\bigr)\\
 & =\mathrm{Ad}_{p_{\beta_{3}}(1)\cdots p_{\beta_{l}}(1)}\bigl(t+\beta_{1}(t)e_{\beta_{1}}+\beta_{2}(t)e_{\beta_{2}})+\beta_{1}(t)\mathrm{Ad}_{p_{\beta_{2}}(1)}\bigl(e_{\beta_{1}}\bigr)-t\\
 & =\mathrm{Ad}_{p_{\beta_{3}}(1)\cdots p_{\beta_{l}}(1)}\bigl(t+\beta_{1}(t)e_{\beta_{1}}+\beta_{2}(t)e_{\beta_{2}})+\beta_{1}(t)O(\beta_{2})-t\\
 & =\sum_{i=1}^{l}\bigl(\beta_{i}(t)e_{\beta_{i}}+O(\beta_{i+1})\bigr)
\end{align*}
 that the values $\beta_{i}(t)$ can be recovered inductively.
\end{proof}
\begin{cor}
\label{cor:dual} Now set $g:=\dot{w}g'$. The map
\[
(\iota^{\sharp})^{-1}(\CMcal O_{N\dot{w}LN}^{N})\longrightarrow\mathfrak{t}_{\mathrm{sc}}^{*},\qquad f\longmapsto\bigl(t\mapsto\mathrm{d}f^{g}(t)\bigr)
\]
 is surjective.
\end{cor}

\begin{proof}
By the previous statement, we can recover $w^{-1}(t)$ and hence $t$
from the projection of
\[
\mathrm{Ad}_{g}\bigl(t\bigr)-w^{-1}(t)=\mathrm{Ad}_{g'}\bigl(w^{-1}(t)\bigr)-w^{-1}(t)
\]
to $\mathfrak{n}_{w}$, and then the same is true for
\[
t^{g}:=(\mathrm{Ad}_{g}-\mathrm{id})(t)=\bigl(\mathrm{Ad}_{g}\bigl(t\bigr)-w^{-1}(t)\bigr)+\bigr(w^{-1}(t)-t\bigr)
\]
 since we are adding an element of $\mathfrak{t}$. Thus 
\[
\CMcal O_{\dot{w}N_{w}}\longrightarrow\mathfrak{t}_{\mathrm{sc}}^{*}
\]
 is surjective. As $\dot{w}LN_{w}=L\dot{w}N_{w}\simeq L\times\dot{w}N_{w}$,
the the cross section isomorphism
\[
\CMcal O_{N\dot{w}LN}\overset{\sim}{\longrightarrow}\CMcal O_{N}\otimes\CMcal O_{\dot{w}LN_{w}}\overset{\sim}{\longrightarrow}\CMcal O_{N}\otimes\CMcal O_{L}\otimes\CMcal O_{\dot{w}N_{w}}
\]
 now yields the claim as elements of the form $1\otimes1\otimes f'$
pull back to $N$-invariants in $\CMcal O_{N\dot{w}LN}$ satisfying
\[
\mathrm{d}f^{g}(t)=\mathrm{d}f(t^{g})=\mathrm{d}(1\cdot f')(t^{g}).
\]
\end{proof}
\begin{example}
Consider the usual group 
\[
G=\mathrm{SL}_{2}=\left\{ \begin{bmatrix}a & b\\
c & d
\end{bmatrix}:ad-bc=1\right\} 
\]
of type $\mathsf{A}_{1}$ and let 
\[
\dot{w}=\begin{bmatrix}0 & 1\\
-1 & 0
\end{bmatrix}
\]
be the usual lift of its Coxeter element. Let $t$ be the usual diagonal
element of its Lie algebra and
\[
g=\dot{w}p_{\alpha}(1)=\begin{bmatrix}0 & 1\\
-1 & -1
\end{bmatrix},
\]
 then
\[
R_{g}t-L_{g}t=\begin{bmatrix}0 & 2\\
2 & 0
\end{bmatrix}.
\]
The matrix coordinate $c$ is invariant under the conjugation action
of $N,$ and $\mathrm{d}c^{g}(t)=2$.
\end{example}

The following lemma was obtained by dissecting Sevostyanov's proof
\cite[Theorem 5.2]{MR2806525}:
\begin{notation}
Given a function $f$ in $\CMcal O_{G}$ defined at a closed point
$g$ of $G$, we write 
\[
\mathrm{d}f_{\pm}^{g}:=r_{\pm}(\mathrm{d}f^{g})\in\mathfrak{g}
\]
and then decompose these elements along the Cartan decomposition $\mathfrak{g}=\mathfrak{t}\oplus(\mathfrak{n}_{+}\oplus\mathfrak{n}_{-})$
as
\[
\mathrm{d}f_{\pm}^{g}=\mathrm{d}f_{\pm0}^{g}+\mathrm{d}f_{\pm\emptyset}^{g},\qquad\mathrm{d}f_{\pm0}^{g}\in\mathfrak{t},\quad\mathrm{d}f_{\pm\emptyset}^{g}\in\mathfrak{n}_{+}\oplus\mathfrak{n}_{-}.
\]
 From (\ref{eq:factorisable-r-matrix-positive}) and (\ref{eq:factorisable-r-matrix-negative})
one then sees that
\[
\mathrm{d}f_{\pm0}^{g}=\mathrm{d}f^{g}(r_{0}\pm c_{\mathfrak{t}})
\]
 and similarly for $\mathrm{d}f_{\pm\emptyset}^{g}$.
\end{notation}

From the right-hand-side of equation (\ref{eq:STS-bracket}), it follows
that Hamiltonian vector field of a function $f$ in $\CMcal O_{G_{*}}$
is given at a point $g$ in this trivialisation by
\[
\mathrm{Ham}_{f}(g)=\mathrm{Ad}_{g}\bigl(r_{+}(\mathrm{d}f^{g})\bigr)-r_{-}(\mathrm{d}f^{g}).
\]

\begin{lem}
For any function $f$ in $(\iota^{\sharp})^{-1}(\CMcal O_{N\dot{w}LN}^{N})$
defined at a point $g$ of $N\dot{w}LN$, we have
\[
\mathrm{Ham}_{f}(g)\equiv\mathrm{Ad}_{\dot{w}}(\mathrm{d}f_{+0}^{g})-\mathrm{d}f_{-0}^{g}\mod T_{g}.
\]
\end{lem}

\begin{proof}
The slicing assumption $\mathfrak{L}\sqcup\mathfrak{N}=\mathfrak{R}_{+}$
implies that $\mathfrak{l}+\mathfrak{t}\oplus\mathfrak{n}_{+}=\mathfrak{t}+\mathfrak{l}\oplus\mathfrak{n}$,
so from $\mathfrak{T}_{0}\subseteq\mathfrak{L}$ we deduce that 
\[
\mathfrak{b}_{+}^{r}\subseteq\mathfrak{l}+\mathfrak{t}\oplus\mathfrak{n}_{+}=\mathfrak{t}+\mathfrak{l}\oplus\mathfrak{n}.
\]
 Thus $\mathrm{d}f_{+\emptyset}^{g}\in\mathfrak{l}\oplus\mathfrak{n}$
and hence 
\[
\mathrm{Ad}_{g}(\mathrm{d}f_{+\emptyset}^{g})\in\mathrm{Ad}_{g}(\mathfrak{l}\oplus\mathfrak{n})\subseteq T_{g}.
\]
Now pick $n\in N$ and $\tilde{n}\in LN$ such that $g=n\dot{w}\tilde{n}$.
As $\mathrm{d}f_{+0}^{g}$ lies in $\mathfrak{t}$ by construction,
the element $\mathrm{Ad}_{n}(\mathrm{d}f_{+0}^{g})-\mathrm{d}f_{+0}^{g}$
lies in $\mathfrak{n}$, so that
\[
\mathrm{Ad}_{n\dot{w}\tilde{n}}(\mathrm{d}f_{+0}^{g})-\mathrm{Ad}_{\dot{w}\tilde{n}}(\mathrm{d}f_{+0}^{g})\in\mathrm{Ad}_{\dot{w}\tilde{n}}(\mathfrak{n})=\mathrm{Ad}_{g}(\mathfrak{n})\subset T_{g}.
\]
Furthermore as $w(\mathfrak{t})=\mathfrak{t}$ normalises $\mathfrak{l}\oplus\mathfrak{n}$,
we have
\[
\mathrm{Ad}_{\dot{w}\tilde{n}}(\mathrm{d}f_{+0}^{g})-\mathrm{Ad}_{\dot{w}}(\mathrm{d}f_{+0}^{g})\in\mathfrak{l}\oplus\mathfrak{n}\subseteq T_{g}
\]
so we conclude from the last three inclusions that
\begin{equation}
\mathrm{Ad}_{g}(\mathrm{d}f_{+}^{g})=\mathrm{Ad}_{g}(\mathrm{d}f_{+\emptyset}^{g})+\mathrm{Ad}_{g}(\mathrm{d}f_{+0}^{g})\equiv\mathrm{Ad}_{\dot{w}}(\mathrm{d}f_{+0}^{g})\mod T_{g}.\label{eq:plus-part}
\end{equation}

As $f$ is invariant under conjugation by $N$, we have $f(ng)-f(gn)=0$
for all $n$ in $N$ which implies that the element $\mathrm{d}f^{g}\in\mathfrak{g}^{*}$
lies in the annihilator of $\mathfrak{n}$. From (\ref{eq:factorisable-r-matrix-negative})
one sees that $\mathrm{d}f_{-\emptyset}^{g}$ lies in 
\[
\oplus_{\beta\in\mathfrak{R}_{+}\backslash\mathfrak{N}}\mathfrak{n}_{-\beta}+\oplus_{\beta\in\mathfrak{T}_{0}}\mathfrak{n}_{-\beta}+\oplus_{\beta\in\mathfrak{T}_{1}}\mathfrak{n}_{\beta}\subseteq(\mathfrak{l}\cap\mathfrak{n}_{-}\bigr)+\mathfrak{l}+\mathfrak{l}=\mathfrak{l}\subseteq T_{g}.
\]
Combined with (\ref{eq:plus-part}), the claim follows.
\end{proof}
Finally, we prove (iii):
\begin{proof}
 Recall Corollary \ref{cor:STS-coisotropic} and Proposition \ref{lem:reduction};
they now imply that the bracket of $\CMcal O_{G_{*}}$ reduces to
$\CMcal O_{N\dot{w}LN}^{N}$ if and only if all Hamiltonian vector
fields of functions in $(\iota^{\sharp})^{-1}(\CMcal O_{N\dot{w}LN}^{N})$
are tangent to $N\dot{w}LN$. From Corollary \ref{cor:dual} we deduce
that $\mathrm{Ham}_{f}(g)\in T_{g}$ for all $f\in(\iota^{\sharp})^{-1}(\CMcal O_{N\dot{w}LN}^{N})$
if and only if the image of
\[
\mathfrak{t}_{\mathrm{sc}}^{*}\longrightarrow\mathfrak{t},\qquad t^{*}\longmapsto w\bigl((\frac{r_{0}+c_{\mathfrak{t}}}{2})(t^{*})\bigr)-(\frac{r_{0}-c_{\mathfrak{t}}}{2})(t^{*})
\]
 lands inside of $T_{g}\cap\mathfrak{t}$, which by Lemma \ref{lem:parab}
equals $\mathfrak{l}\cap\mathfrak{t}\subseteq\mathfrak{t}^{w}$. In
other words, $\mathrm{Ham}_{f}(g)\in T_{g}$ if and only if the projection
of
\[
\bigl(\frac{1}{2}(w-1)(r_{0})+\frac{1}{2}(w+1)(c_{\mathfrak{t}})\bigl)(\mathfrak{t}_{\mathrm{sc}}^{*})
\]
 to $\mathfrak{t}_{w}$ is trivial in the orthogonal decomposition
$\mathfrak{t}=\mathfrak{t}^{w}\oplus\mathfrak{t}_{w}$. Since the
right-hand-side would send $\mathfrak{t}_{\mathrm{sc}}^{w*}$ to $\mathfrak{t}^{w}$
and $w-1$ acts nontrivially on anything outside of $\mathfrak{t}^{w}$,
it follows that $r_{0}$ must map $\mathfrak{t}^{w*}$ to $\mathfrak{t}^{w}$,
for otherwise a nontrivial $\mathfrak{t}_{w}$-component would appear.
By skew-symmetry, $r_{0}$ then also preserves $\mathfrak{t}_{w}$.

By $w$-invariance we may decompose $c_{\mathfrak{t}}=c_{\mathfrak{t}_{w}}+c_{\mathfrak{t}^{w}}$.
On $\mathfrak{t}^{w}$ the operator $w-1$ then acts trivially, so
it follows from nondegeneracy of $c_{\mathfrak{t}^{w}}=\frac{1}{2}(w+1)(c_{\mathfrak{t}^{w}})$
that the image is all of $\mathfrak{t}^{w}$, so $\mathfrak{t}^{w}\subseteq\mathfrak{l}$.

Now consider $r_{0}|_{\mathfrak{t}_{w}}$. It needs to satisfy 
\[
(w-1)(r_{0}|_{\mathfrak{t}_{w}})+(w+1)(c_{\mathfrak{t}})=0,
\]
 which rewrites as
\[
r_{0}|_{\mathfrak{t}_{w}}=\frac{1+w}{1-w}.\qedhere
\]
\end{proof}

\small

\bibliographystyle{alphakey}
\phantomsection\addcontentsline{toc}{section}{\refname}\bibliography{slices}

\end{document}